\documentclass[a4paper,11 pt,twoside]{amsart}

\usepackage[utf8]{inputenc}
\usepackage[english]{babel}
\usepackage{times}
\usepackage{amsmath}
\usepackage{amsfonts}
\usepackage{amssymb}
\usepackage{amsmath}
\usepackage{amsthm}
\usepackage{graphicx}
\usepackage{array}
\usepackage{color}
\usepackage{mathrsfs}
\usepackage{hyperref}
\usepackage{esint} 
\usepackage{tikz}
\usepackage{ upgreek }
\usepackage{enumitem}

\allowdisplaybreaks

\definecolor{darkgreen}{rgb}{0.1,0.7,0.1}
\definecolor{darkred}{rgb}{0.7,0.1,0.1}

\newcommand{\llrrbracket}[1]{
  \left[\mkern-3mu\left[#1\right]\mkern-3mu\right]}

\newtheorem{theorem}{Theorem}
\newtheorem{lemma}{Lemma}[section]
\newtheorem{proposition}[lemma]{Proposition}
\newtheorem{corollary}[lemma]{Corollary}

\newtheorem{remark}[lemma]{Remark}

\newcommand{\ru}{\hat{r}}

\newcommand{\thetau}{\hat{\theta}}

\newcommand{\yu}{\hat{y}}

\newcommand{\cc}{\complement}

\newcommand{\cotan}{\mbox{cotan\,}}
\newcommand{\arccotan}{\mbox{arccotan\,}}

\newcommand{\cLa}{\cL^*}
\newcommand{\Eg}{\mathbf{E}}

\newcommand{\eps}{\varepsilon}

\newcommand\symb[2][\bf]{{\mathchoice{\hbox{#1#2}}{\hbox{#1#2}}%
        {\hbox{\scriptsize#1#2}}{\hbox{\tiny#1#2}}}}

\def\R{{\symb R}}
\def\N{{\symb N}}
\def\Z{{\symb Z}}
\def\C{{\symb C}}

\def\P{{\symb P}}

\def\un{\mathbf{1}}

\renewcommand{\P}{\mathbb{P}}
\newcommand{\E}{\mathbb{E}}

\newcommand{\bbQ}{\mathbb{Q}}

\newcommand{\cB}{\mathcal{B}}
\newcommand{\cC}{\mathcal{C}}
\newcommand{\cD}{\mathcal{D}}
\newcommand{\cE}{\mathcal{E}}
\newcommand{\cF}{\mathcal{F}}
\newcommand{\cG}{\mathcal{G}}
\newcommand{\cH}{\mathcal{H}}

\newcommand{\cL}{\mathcal{L}}
\newcommand{\cM}{\mathcal{M}}
\newcommand{\cN}{\mathcal{N}}
\newcommand{\cO}{\mathcal{O}}

\newcommand{\cW}{\mathcal{W}}

\newcommand{\ccC}{\mathscr{C}}
\newcommand{\ccD}{\mathscr{D}}

\newcommand{\ccN}{\mathscr{N}}

\newcommand{\ccQ}{\mathscr{Q}}

\newcommand{\ccV}{\mathscr{V}}

\newcommand{\gl}{\lambda}

\hypersetup{
citecolor=blue,
colorlinks=true, 
breaklinks=true, 
urlcolor= blue, 
linkcolor= black, 
bookmarksopen=true, 
pdftitle={Anderson Hamiltonian}, 
}

\begin{document}

\title[Localization crossover for the Anderson Hamiltonian]{Localization crossover\\for the continuous Anderson Hamiltonian in $1$-d}

\author{Laure Dumaz}
\address{
CNRS \& Department of Mathematics and Applications, \'Ecole Normale Sup\'erieure (Paris), 45 rue d’Ulm, 75005 Paris, France}
\email{laure.dumaz@ens.fr}

\author{Cyril Labb\'e}
\address{
Universit\'e Paris-Dauphine, PSL University, UMR 7534, CNRS, CEREMADE, 75016 Paris, France\\
\& \'Ecole Normale Sup\'erieure, UMR 8553, CNRS, DMA, 75005 Paris, France}
\email{labbe@ceremade.dauphine.fr}

\vspace{2mm}

\date{\today}

\maketitle

\begin{abstract}
We investigate the behavior of the spectrum of the continuous Anderson Hamiltonian $\cH_L$, with white noise potential, on a segment whose size $L$ is sent to infinity. We zoom around energy levels $E$ either of order $1$ (Bulk regime) or of order $1\ll E \ll L$ (Crossover regime). We show that the point process of (appropriately rescaled) eigenvalues and centers of mass converge to a Poisson point process. We also prove exponential localization of the eigenfunctions at an explicit rate. In addition, we show that the eigenfunctions converge to well-identified limits: in the Crossover regime, these limits are \emph{universal}. Combined with the results of our companion paper \cite{DLCritical}, this identifies completely the transition between the localized and delocalized phases of the spectrum of $\cH_L$. The two main technical challenges are the proof of a two-points or Minami estimate, as well as an estimate on the convergence to equilibrium of a hypoelliptic diffusion, the proof of which relies on Malliavin calculus and the theory of hypocoercivity.

\medskip

\end{abstract}

\setcounter{tocdepth}{1}
\tableofcontents

\section{Introduction}

In a celebrated article~\cite{Anderson58}, Anderson proposed the Hamiltonian $-\Delta + V$ on the lattice $\Z^d$ as a simplified model for electron conduction in a crystal. The so-called \emph{disorder} $V$ is a random potential that models the defects of the crystal. The question was whether those defects can trap the electron i.e. localize the electronic wave functions. He argued that for a large enough disorder $V$, the spectrum is pure point and the eigenfunctions exponentially \emph{localized} -- a phenomenon now referred to as \emph{Anderson localization}. Mathematically, one can see this model as the interpolation between the discrete Laplacian $-\Delta$ on the grid $\Z^d$, which has delocalized eigenfunctions and the multiplication by a potential $V$ on each site, whose eigenfunctions are the coordinate vectors.\\

One of the first rigorous results of Anderson localization was obtained by Goldsheid, Molchanov and Pastur~\cite{GMP1977} and Molchanov~\cite{molcanov1980}: it concerned the continuum analogue of the above model in dimension $d=1$, namely the operator $-\partial_x^2 +V$ on $\R$, for some specific random potential $V$. This was followed by a series of major articles~\cite{Kunz,FS,AM,Minami} to name but a few, in the discrete or the continuum setting, and for general dimension $d\ge 1$. One can summarize the main results as follows: (1) In dimension $d=1$, Anderson localization holds in the whole spectrum; (2) In dimension $d \geq 2$, for a large enough disorder or at a low enough energy, Anderson localization holds. In dimension $d \geq 3$, it is expected (but not proved) that there is a delocalized phase for $V$ weak enough while in dimension $d=2$, the question remains open. We refer to~\cite{CarLac,Kirsch} for more details.\\

In the present article, we consider the case where the potential is a white noise $\xi$ in dimension $d=1$. This is a random Gaussian distribution with covariance given by the Dirac delta: it models physical situations where the disorder is totally uncorrelated. Due to its universality property (white noise arises as scaling limit of appropriately rescaled noises with finite variance), this is a natural choice of potential in the continuum. However, white noise is a highly irregular potential as it is only distribution valued, and therefore falls out of scope of virtually all general results of the literature.\\

It is standard to tackle Anderson localization by considering first the Hamiltonian truncated to a finite box, before passing to the infinite volume limit. In this article we focus on the truncated Hamiltonian $\cH_L = -\partial^2_x + \xi$ on $(-L/2,L/2)$ with Dirichlet b.c., and investigate its \emph{entire} spectrum in the limit $L\to\infty$. More precisely, we study the local statistics of the operator recentered around energy levels $E$ that are either finite or diverge with $L$: note that the infinite volume limit only captures energy levels that do not depend on $L$. The results of this article, complemented by those presented in our companion papers~\cite{DL17,DLCritical}, reveal a rich variety of behaviors for the eigenvalues/eigenfunctions of $\cH_L$ according to the energy regime considered: in particular, \emph{delocalized} eigenfunctions arise at large enough energies. This is to be compared with the aforementioned general results of Anderson localization in dimension $d=1$ that assert localization in the full spectrum. The rigorous proof of Anderson localization for the infinite volume Hamiltonian with white noise potential will be presented in another article~\cite{DLAndLoc}. \\

Let us now present the main results of the present article, and their connections with the results already obtained in~\cite{DL17,DLCritical}. We recenter $\cH_L$ around some energy level $E=E(L)$ and distinguish five regimes:\begin{enumerate}
\item Bottom: $E \sim -(\frac38 \ln L)^{2/3}$,
\item Bulk: $E$ is fixed with respect to $L$,
\item Crossover: $1 \ll E \ll L$,
\item Critical: $E \asymp L$,
\item Top: $E\gg L$.
\end{enumerate}
For each regime, we investigate the local statistics of the eigenvalues $\gl_i$ near $E$, and the behavior of the corresponding normalized eigenfunctions $\varphi_i$. In~\cite{DL17,DLCritical}, we covered the Bottom, Critical and Top regimes, while in the present article we derive the Bulk and Crossover regimes. The main results are summarized on Figure \ref{Table}.

\begin{figure}
\footnotesize\renewcommand{\arraystretch}{1.5}
\begin{tabular}{c|c|c|c|c|c|}
\cline{2-6}& \textbf{Bottom} & \textbf{Bulk} & \textbf{Crossover} & \textbf{Critical} & \textbf{Top}\\
& $E\sim -(\frac38 \ln L)^{2/3}$ & $E = O(1)$ & $1 \ll E \ll L$ & $\quad E\asymp L\quad$ & $E\gg L$\\
\hline
\multicolumn{1}{|l|}{Eigenval.} & \multicolumn{3}{c|}{Poisson} & Sch & Picket Fence\\
\hline
\multicolumn{1}{|l|}{Eigenfct.} & \multicolumn{3}{c|}{Localized} & \multicolumn{2}{c|}{Delocalized}\\\cline{2-4}\cline{4-6}
\multicolumn{1}{|l|}{localis.~length} & $O(1/\sqrt{|E|})$ & $O(1)$ & $O(E)$ & \multicolumn{2}{c|}{$L$}\\\cline{2-6}
\multicolumn{1}{|l|}{asymp.~shape} & Determ. \& Specific & Random \& Specific & \multicolumn{2}{c|}{Random \& Universal} & Determ. \& Universal\\
\multicolumn{1}{|l|}{}&$\frac1{\cosh(t)}$ & $Y_E(t)$ & \multicolumn{2}{c|}{$\exp(-\frac{|t|}{8} + \frac1{2\sqrt2} \cB(t))$} & $\sin(t)$\\
\hline
\end{tabular}\label{Table}\caption{The five regimes of $\cH_L$.}
\end{figure}
\normalsize

\medskip

The transition between Poisson statistics and Picket fence occurs in the Critical regime where $E\asymp L$. In~\cite{DLCritical}, we prove that, in that regime, the eigenvalue statistics are given by the Sch point process introduced by Kritchevski, Valk\'o and Vir\'ag~\cite{KVV}. Actually, we obtain a convergence at the operator level, and show not only that the eigenvalues converge towards Sch but also that the eigenfunctions converge to a \emph{universal} limit given by the exponential of a two-sided Brownian motion plus a negative linear drift. Note that this limit lives on a finite interval thus making the eigenfunctions \emph{delocalized}. Our denomination universal is justified by the fact that this shape already appeared in the work of Rifkind and Vir\'ag~\cite{RV} who conjectured that it should arise in various critical models.
\medskip

In the present article, we show that in the Bulk and Crossover regimes, the local statistics of the eigenvalues (jointly with the centers of mass of the eigenfunctions) converge to a Poisson point process. Moreover, we establish exponential decay of the eigenfunctions (from their centers of mass) at an explicit rate, which is of order $1$ in the Bulk regime and of order $E$ in the Crossover regime.

Actually we provide much more information about the eigenfunctions. We show that the eigenfunctions (recentered at their centers of mass and rescaled in space by $1$ in the Bulk and by $E$ in the Crossover) converge to \emph{explicit} limits. In the Bulk regime, the limits are some well-identified diffusions $Y_E$, whose law depend of $E$. In the Crossover regime, the limits are given by the \emph{same} universal shape as in the Critical regime: namely, the exponential of a two-sided Brownian motion plus a negative linear drift. However since the space scale is $E \ll L$,  the eigenfunctions are still localized and the limiting shape lives on an infinite interval, in contrast with the Critical regime. As $E\uparrow L$, one formally recovers delocalized states and this justifies a posteriori the denomination \emph{Crossover}: this regime of energy interpolates between the (localized) Bulk regime and the (delocalized) Critical regime and shares features with both (Poisson statistics with the former and universal shape with the latter).

\medskip

Finally, in the Bottom regime, investigated in~\cite{DL17}, we also obtained Poisson statistics for the eigenvalues (and centers of mass). Furthermore we showed that the eigenfunctions are strongly localized: at space scale $1/(\log L)^{1/3}$ and recentered at their centers of mass, they converge to the deterministic limit $1/\cosh$.

\bigskip

Overall, our results provide a complete transition from strongly localized states (Bottom regime) to totally delocalized states (Top regime) and identifies explicitly the local statistics of the eigenvalues together with the asymptotic shapes of the eigenfunctions.

\bigskip

Let us now comment on the technical challenges that these results represent. Since white noise is out of scope of usual standing assumptions, we do not rely on general results from Anderson localization literature, so that our article is self-contained. Let us point out two major difficulties that we encounter. First, the derivation of the two-points estimate, often called \emph{Minami estimate}~\cite{Minami}, is delicate in the context of the irregular potential given by the white noise, whereas some general $1$-d results~\cite{KloppTunnel} are available for some related models with smoother potential. To prove this estimate, we rely on a thorough study of a joint diffusion, see Section \ref{Sec:TwoPoints}. Second, in the Crossover regime the phase function rotates at an unbounded speed on the unit circle and this yields many technical challenges. In particular, to obtain \emph{quantitative} (with respect to the unbounded parameter $E$) estimates on the convergence to equilibrium of this phase, we cannot simply apply H\"ormander's Theory of hypoellipticity in contrast with the situation in~\cite{molcanov1980,CarmonaDuke}: we obtain these estimates using Malliavin Calculus and the theory of hypocoercivity, and this constitutes one of the main technical achievements of the present article, see Section \ref{Sec:Expo}.

\bigskip

Let us relate our results to other studies in the literature. The discrete counterpart\footnote{Note that in our case, due to the scaling property of the white noise, our study covers the case where there is an additional parameter $\sigma_L \to 0$ in front of the white noise.} of our model is given by the $N\times N$ random tridiagonal matrix $-\Delta_N + \sigma_N \,V_N$ where $\Delta_N$ is the discrete Laplacian on $\{1,\ldots,N\}$, $V_N$ is a diagonal matrix with i.i.d.~entries of mean $0$ and variance $1$ and $\sigma_N$ a positive parameter, possibly depending on $N$. If $\sigma_N$ does not depend on the size $N$ of the matrix, then the limit $N\to\infty$ of the model falls into the scope of general $1$-d Anderson localization results, and the spectrum is localized. On the other hand, one recovers delocalized states when $\sigma_N = O(N^{-1/2})$, see~\cite{DSS}. Actually the point process of eigenvalues of the matrix in the bulk converges to the Sch random point process~\cite{KVV} and the eigenfunctions are delocalized~\cite{RV}. There are also connections with recent investigations in \cite{Nakano2014, KN17, Nakano2019} of the aforementioned Russian model of Goldsheid, Molchanov and Pastur~\cite{GMP1977}, in which, as in the random matrix model, a parameter depending on the size of the system is added in front of the potential to reduce its influence.

\medskip

Our description of the local statistics of the eigenvalues/centers of mass is in the vein of recent results by Germinet and Klopp~\cite{GK}, who proved precise results on the local and global eigenvalue statistics for a large class of Schr\"odinger operators. On the other hand, we provide a complete and explicit description of the asymptotic shape taken by the eigenfunctions: to the best of our knowledge, such results are very rare in the literature on Anderson localization.

\section{Main results}\label{Sec:MainResults}

Let $\xi$ be a Gaussian white noise on $\R$, that is, the derivative of a Brownian motion $B$. We consider the truncated Anderson Hamiltonian (sometimes also called Hill's operator)
\begin{equation}\label{Eq:Hamiltonian}
\cH_L = -\partial^2_x + \xi\;,\quad x\in(-L/2,L/2)\;,
\end{equation}
endowed with homogeneous Dirichlet boundary conditions. It was shown in~\cite{Fukushima} that this operator is self-adjoint on $L^2(-L/2,L/2)$ with pure point spectrum of multiplicity one bounded from below $\lambda_1 < \lambda_2 < \ldots$ We let $(\varphi_{k})_k$ be the corresponding eigenfunctions normalized in $L^2$. These r.v.~depend on $L$, but for notational simplicity we omit writing this dependence. 

\smallskip

This operator has a deterministic density of states $n(E)$, see~\cite{Halperin,Fukushima}. This is defined as $n(E) := dN(E)/dE$ where
\begin{equation}\label{Eq:NE}
N(E) := \lim_{L\to\infty} \frac1{L} \#\{\lambda_i: \lambda_i \le E\}\;,\quad E\in\R\;.
\end{equation}
Here the convergence is almost sure and the limit is deterministic. Roughly speaking, $1/(L n(E))$ measures the typical spacing between two consecutive eigenvalues lying near $E$ for the operator $\cH_L$. From the explicit integral expression of $n(E)$, see~\cite{Fukushima}, one deduces that $E\mapsto n(E)$ is smooth and that
\begin{align*}
n(E) \sim \frac{1}{2\pi\sqrt E}\;,\quad E\to+\infty\;.
\end{align*}

\bigskip

In the present article, we focus on two regimes of energy:\begin{enumerate}
\item Bulk regime: the energy $E$ is fixed with respect to $L$,
\item Crossover regime: the energy $E = E(L)$ satisfies $1 \ll E \ll L$,
\end{enumerate}
and investigate the asymptotic behavior as $L\to\infty$ of the eigenvalues $\lambda_i$ and of the eigenfunctions, seen as probability measures by considering $\varphi^2_i(t) dt$. For every eigenfunction $\varphi_i$, a relevant statistics is\footnote{Other statistics could be considered without altering the results.} the \emph{center of mass} $U_i$ defined through
$$ U_i := \int_{[-L/2,L/2]} t \varphi^2_i(t) dt\;.$$

Our first result shows convergence, in both regimes, of the point process of rescaled eigenvalues and centers of mass.

\begin{theorem}[Poisson statistics]\label{Th:Poisson}
In the Bulk and the Crossover regimes, the following random measure on $\R \times [-1/2,1/2]$
$$ \sum_{i\ge 1} \delta_{(L\, n(E)(\lambda_i - E), U_i/L)}$$
converges in law as $L\to\infty$ to a Poisson random measure on $\R \times [-1/2,1/2]$ of intensity $d\gl \otimes du$.
\end{theorem}
In this statement, the convergence holds for the vague topology on the set of Radon measures on $\R \times [-1/2,1/2]$, that is, the smallest topology that makes continuous the maps $m\mapsto \langle m,f\rangle$ with $f:\R\times [-1/2,1/2]\to \R$ bounded continuous and compactly supported in its first variable.

\bigskip

Our second result establishes exponential localization of the eigenfunctions from their centers of mass. In the Bulk regime, the exponential rate is given by $(1/2)\nu_E$ where
\begin{equation}\label{Eq:nuE}
\nu_E = \frac{ \int_0^{+\infty} \sqrt{u} \exp(-2 E u - \frac{u^3}{6} ) du}{\int_0^{\infty} \frac{1}{\sqrt{u}} \exp(-2 E u - \frac{u^3}{6}) du}\,,
\end{equation}
while in the Crossover regime it is given by $1/(2E)$. This rate is of course related to the Lyapounov exponent of the underlying diffusions.

\begin{theorem}[Exponential localization]\label{Th:Loc}
Fix $h >0$ and set $\Delta :=  [E-h/Ln(E),E+h/Ln(E)]$. For every $\eps > 0$ small enough, there exist some r.v.~$c_i>0$ such that:\begin{enumerate}
\item for every eigenvalue $\lambda_i \in \Delta$ we have in the Bulk regime
$$ \big(\varphi_i(t)^2 + \varphi'_i(t)^2\big)^{1/2} \le c_i \exp\Big(-\frac{(\nu_E - \eps)}{2}\,|t-U_i|\Big)\;,\quad \forall t\in [-L/2,L/2]\;,$$
and in the Crossover regime
$$ \Big(\varphi_i(t)^2 + \frac{\varphi'_i(t)^2}{E}\Big)^{1/2} \le \frac{c_i}{\sqrt E} \exp\Big(-\frac{(1-\eps)}{2}\frac{|t-U_i|}{E}\Big)\;,\quad \forall t\in [-L/2,L/2]\;,$$
\item in both regimes, there exists $q=q(\eps) > 0$ such that
$$ \limsup_{L\to\infty} \E\Big[\sum_{\gl_i \in \Delta} c_i^q \Big] < \infty\;.$$
\end{enumerate}
\end{theorem}

\bigskip

Our third result shows that the eigenfunctions (rescaled into probability measures) converge to a Poisson point process with an explicit intensity. Actually the convergence will be taken jointly with the eigenvalues and the centers of mass, so the result is a strengthening of Theorem \ref{Th:Poisson}.

Let us start with the definition of the probability measures associated with the eigenfunctions. Given the rate of exponential localization appearing in the previous result, one needs to recenter the eigenfunctions appropriately to get convergence, so we define for every eigenvalue $\lambda_i$
$$w_i(dt) := \begin{cases} \varphi_i(U_i+t)^2 dt\quad &\mbox{ in the Bulk regime,}\\
E \varphi_i(U_i+tE)^2 dt\quad &\mbox{ in the Crossover regime,}
\end{cases}$$
which is an element of the set $\cM = \cM(\R)$ of all probability measures on $\R$ endowed with the topology of weak convergence.\\
In the statement below, we rely on a probability measure on $\cM$ that describes the law of the limits. In the Crossover regime, this probability measure $\sigma_\infty$ admits a simple definition: it is the law of the random probability measure
\begin{align}
\frac{Y_\infty(t+U_\infty)^2 dt}{\int Y_\infty(t+U_\infty)^2 dt}\;,\quad \mbox{with}\quad Y_\infty(t) := \exp\Big(- \frac{|t|}{8} + \frac{\cB(t)}{2\sqrt 2}\Big)\;,\quad U_\infty :=  \frac{\int t Y_\infty(t)^2 dt}{\int Y_\infty(t)^2 dt}\;,\label{def:sigmainfty}
\end{align}
where $\cB$ is a two-sided Brownian motion on $\R$.\\
In the Bulk regime, this probability measure $\sigma_E$ is the law of the random probability measure
\begin{align}
\frac{Y_E(t+U_E)^2 dt}{\int Y_E(t+U_E)^2 dt}\;,\quad U_E := \frac{\int t Y_E(t)^2 dt}{\int Y_E(t)^2 dt}\;,\label{def:sigmaE}
\end{align}
where $Y_E$ is the concatenation at $t=0$ of two processes $(Y_E(t),t\ge 0)$ and $(Y_E(-t),t\ge 0)$ with explicit laws. The precise definition of $Y_E$ requires some notations and is given in Subsection \ref{Subsec:Intensity}.

\begin{theorem}[Shape]\label{Th:Shape}
In the Bulk and the Crossover regimes, the random measure
$$ \cN_L := \sum_{i\ge 1} \delta_{(L\, n(E)(\lambda_i - E), U_i/L, w_i)}\;,$$
converges in law as $L\to\infty$ to a Poisson random measure on $\R \times [-1/2,1/2] \times \cM$ of intensity $d\gl \otimes du \otimes \sigma_E$ in the Bulk regime and of intensity $d\gl \otimes du \otimes \sigma_\infty$ in the Crossover regime.
\end{theorem}
Here $\cN_L$ is seen as a random element of the set of measures on $\R\times [-1/2,1/2] \times \cM$ that give finite mass to $K\times [-1/2,1/2] \times \cM$, for any given compact set $K\subset \R$. The topology considered in the convergence above is then the smallest topology on this set of measures that makes continuous the maps $m\mapsto \langle f,m\rangle$, for every bounded and continuous function $f:\R\times [-1/2,1/2] \times \cM\to \R$ that is compactly supported in its \emph{first} coordinate.

\medskip

Let us make some comments on the limits of these eigenfunctions. First of all, the localization length is of order $1$ in the Bulk regime and $E$ in the Crossover regime: this is in line with the exponential decay of Theorem \ref{Th:Loc}. Moreover, this localisation length increases with the level of energy, and this is coherent with the general fact that: the lower we look into the spectrum, the more localized the eigenfunctions are. Second, it suggests that for $E$ of order $L$ the eigenfunctions are no longer localized: this is rigorously proved in our companion paper~\cite{DLCritical}.

\bigskip

We now outline the remaining of this article. We introduce in Section \ref{sec:thediffusions} the diffusions associated to our eigenvalue equation as they play a central role in this article. Then we detail in Section \ref{sec:strategy} our strategy of proof: this section presents the main intermediate results of this paper and contains the proofs of Theorems \ref{Th:Poisson}, \ref{Th:Loc} and \ref{Th:Shape}. The subsequent sections are then devoted to proving these intermediate results, more details on their contents and relationships will be given at the end of Section \ref{sec:strategy}.

\subsection*{Acknowledgements} The work of CL is supported by the project SINGULAR ANR-16-CE40-0020-01. The authors would like to warmly thank Emeric Bouin, Jean Dolbeault, Amic Frouvelle and Martin Hairer for inspiring discussions on hypoellipticity and hypocoercivity that led to the result of Section \ref{Sec:Expo}.

\section{The diffusions}\label{sec:thediffusions}

The eigenvalue problem associated to the operator $\cH_L$ gives rise to a family $y_\gl$, $\gl\in \R$ of ODEs, corresponding to the solution of the eigenvalue equation:
\begin{align}\label{ODEH_L}
 -y_\gl''(t) + \xi(dt)  y_\gl(t) = \gl y_\gl(t)\;,\quad t\in (-L/2,L/2)\;.
\end{align}
Up to fixing two degrees of freedom there is a unique solution to this equation for every parameter $\gl \in \R$. If we fix the starting conditions $y_\gl(-L/2) = 0$ and $y_\gl'(-L/2)$ to an arbitrary value different from $0$, then the map $\gl \mapsto y_\gl(L/2)$ is continuous. The zeros of this map are the eigenvalues of $\cH_L$, and the corresponding solutions $y_\gl$ are the eigenfunctions of $\cH_L$ (which are defined up to a multiplicative factor corresponding to $y_\gl'(-L/2)$ of course).

\medskip
It is crucial in our analysis to consider the evolution of \emph{both} $y_\gl$ and $y'_\gl$, and this is naturally described by the complex function $y_\lambda' + i y_\lambda$.
It will actually be convenient to work in polar coordinates (also called Pr\"ufer variables): we consider $(r_\lambda,\theta_\lambda)$ where
$$ y_\lambda' + i y_\lambda = r_\gl e^{i \theta_\gl} \;.$$
The process $\theta_\gl$ is called the \emph{phase} of the above ODE. It is instrumental for the study of the spectrum of $\cH_L$ as we will see later on. It is also convenient to define
$$ \rho_\gl(t) := \ln r_\gl^2(t)\;.$$
In these new coordinates, we have the following coupled stochastic differential equations:
\begin{align}
d\theta_\gl(t) &= \big(1 + (\gl-1) \sin^2 \theta_\gl + \sin^3\theta_\gl \cos\theta_\gl\big) dt - \sin^2 \theta_\gl dB(t)\;, \label{eq_theta}\\
d\rho_\gl(t) &= \big(-(\gl-1) \sin 2\theta_\gl - \frac12 \sin^2 2\theta_\gl + \sin^2 \theta_\gl \big)dt+ \sin 2\theta_\gl dB(t)\;.\label{EDSlnr}
\end{align}
where $dB(t) = \xi(dt)$.

\medskip

The two degrees of freedom given by the initial conditions $y_\gl(-L/2)$ and $y_\gl'(-L/2)$ are transferred to $\theta_\gl(-L/2)$ and $r_\gl(-L/2)$ (or equivalently $\rho_\gl(-L/2)$). Note that
$$y_\gl(-L/2) = 0 \;\;\&\;\; y_\gl'(-L/2) > 0 \quad \Longleftrightarrow\quad \theta_\gl(-L/2) \equiv 0 [2\pi] \;\;\&\;\; r_\gl(-L/2) = y_\gl'(-L/2)\;,$$
while
$$y_\gl(-L/2) = 0 \;\;\&\;\; y_\gl'(-L/2) < 0 \quad \Longleftrightarrow\quad \theta_\gl(-L/2) \equiv \pi [2\pi] \;\;\&\;\; r_\gl(-L/2) = -y_\gl'(-L/2)\;.$$

\medskip

For any angle $\theta \in \R$, we define
$$\lfloor \theta \rfloor_\pi := \lfloor \theta/\pi\rfloor \pi \in \pi\Z\;,\quad \mbox{ and } \quad \{\theta\}_\pi := \theta - \lfloor \theta \rfloor_\pi \in [0,\pi)\;.$$
Let us point out an important property of the process $\theta_\gl$: at the times $t$ such that $\{\theta_\gl(t)\}_\pi = 0$, the drift term is strictly positive while the diffusion coefficient vanishes. Consequently $\{\theta_\gl\}_\pi$ cannot hit $0$ from above: this readily implies that $t \mapsto \lfloor \theta_\gl(t) \rfloor_\pi$ is non-decreasing. Moreover, the evolution of $\theta_\gl$ depends only on $\{\theta_\gl\}_\pi$ so that the latter is a Markov process. We let $p_{\gl,t}(\theta_0, \theta)$ be the density of its transition probability at time $t$ when it starts from $\theta_0$ at time $0$. When this process starts from $0$ at time $0$, we drop the first variable and simply write $p_{\gl,t}(\theta)$. At some occasions, we will write $\P_{(t_0,\theta_0)}$ for the law of the process $\theta_\gl$ starting from $\theta_0$ at time $t_0$, and we will write $\P_{(t_0,\theta_0)\to (t_1,\theta_1)}$ to denote the law of the bridge of diffusion that start from $\theta_0$ at time $t_0$ and is conditioned to reach $\theta_1$ at time $t_1$.

\smallskip
Most of the time, we will take $\theta_\gl(-L/2) = 0$ and $r_\gl(-L/2) > 0$. In this setting, $\lambda$ is an eigenvalue of $\cH_L$ iff $\{\theta_\gl(L/2)\}_\pi = 0$ and then, the function $y_\lambda$ is the associated eigenfunction. Since $\gl\mapsto \lfloor \theta_\gl(L/2) \rfloor_\pi$ is non-decreasing, we deduce the so-called \emph{Sturm-Liouville property}: almost surely,
\begin{align}
\#\{\gl_i \;:\; \gl_i \leq \gl\} = \lfloor \theta_\gl(L/2) \rfloor_\pi\;, \quad \mbox{ for all $\gl \in \R$\,,}\label{SturmLiouvilleTheta}
\end{align}

This phase function $\theta_\gl(\cdot)$ is a powerful tool to investigate the spectrum of $\cH_L$. It has been used in numerous articles on $1$d-Schr\"odinger operators, sometimes under the form of the so-called Riccati transform $y'_\gl/y_\gl$ which is nothing but $\cotan \theta_\gl$.

\subsection{The distorted coordinates}

For large energies, that is, whenever $\gl$ is of order $E$ for some $E=E(L) \to \infty$, the phase $\theta_\gl$ takes a time of order $1/\sqrt E$ to make one rotation so that the solutions $y_\gl, y_\gl'$ oscillate very fast. It is then more convenient to use \emph{distorted} and \emph{sped up} coordinates. We set
$$ \frac{y_\gl'(tE)}{\sqrt{E}} + i y_\gl(tE) =: r_\gl^{(E)}(t) e^{i \theta_\gl^{(E)}(t)}\;,\quad t\in [-L/(2E),L/(2E)]\;,$$
and
$$ y_\gl^{(E)}(t) := r_\gl^{(E)}(t) \sin \theta_\gl^{(E)}(t)\;,\quad (y_\gl^{(E)})'(t) = E^{3/2} r_\gl^{(E)}(t) \cos \theta_\gl^{(E)}(t)\;,\quad t\in [-L/(2E),L/(2E)]\;.$$

Defining the Brownian motion $B^{(E)}(t) = E^{-1/2} B(tE)$ and setting $\rho_\gl^{(E)} := \ln (r_\gl^{(E)})^2$, the evolution equations take the form
\begin{align*}
&d\theta_\gl^{(E)} = \Big(E^{3/2} + \sqrt{E} (\gl-E) \sin^2 \theta_\gl^{(E)} +  \sin^3 \theta_\gl^{(E)} \cos\theta_\gl^{(E)}\Big) dt -  \sin^2 \theta_\gl^{(E)} dB^{(E)}(t)\;,\\
&d\rho_\gl^{(E)} = \big(-\sqrt{E} (\gl-E) \sin 2\theta_\gl^{(E)}  - \frac{1}{2} \sin^2 2 \theta_\gl^{(E)} + \sin^2 \theta_\gl^{(E)} \big) dt + \sin 2\theta_\gl^{(E)} dB^{(E)}(t)\;.
\end{align*}

Let $p_{\gl,t}^{(E)}(\theta_0, \theta)$ be the density of the transition probability at time $t$ of the processes $\{\theta_\gl^{(E)}\}_\pi$ starting from $\theta_0$ at time $t=0$. When this process starts from $0$ a time $0$, we drop the first variable and simply write  $p_{\gl,t}^{(E)}(\theta)$. The change of variable formula shows that
\begin{equation}\label{Eq:CoVp}
p_{\gl,t}^{(E)}(\theta) = p_{\gl,t}(\arccotan(\sqrt E \cotan \theta)) \frac{\sqrt E}{\sin^2 \theta + E \cos^2 \theta}\;.
\end{equation}

Again, we use the notation $\P_{(t_0,\theta_0)}$, resp.~$\P_{(t_0,\theta_0)\to (t_1,\theta_1)}$, to denote the law of the diffusion, resp.~bridge of diffusion.

\subsection{Condensed notation}

The distorted coordinates will be used in the \emph{Crossover} regime, while we will rely on the original coordinates in the \emph{Bulk} regime. However most of our arguments apply indifferently to both cases. Consequently, we will adopt condensed notations as much as possible. First we introduce
$$\Eg := \begin{cases} 1 & \mbox{ original coordinates}\;,\\ E&\mbox{ distorted coordinates}\;.\end{cases}$$
Moreover, when the arguments apply to both sets of coordinates, we will use the generic notation $\theta_\gl^{(\Eg)}$, $\rho_\gl^{(\Eg)}$ to denote $\theta_\gl$, $\rho_\gl$ in the original coordinates, and $\theta_\gl^{(E)}$, $\rho_\gl^{(E)}$ in the distorted coordinates. We will sometimes introduce notations using the generic notation at once. For instance, we could have introduced the SDE satisfied by $\theta_\gl, \rho_\gl$ and $\theta_\gl^{(E)}, \rho_\gl^{(E)}$ by simply writing
\begin{align}
&d\theta_\gl^{(\Eg)} = \Big(\Eg^{3/2} + \sqrt{\Eg} (\gl-\Eg) \sin^2 \theta_\gl^{(\Eg)} +  \sin^3 \theta_\gl^{(\Eg)} \cos\theta_\gl^{(\Eg)}\Big) dt -  \sin^2 \theta_\gl^{(\Eg)} dB^{(\Eg)}(t)\;,\label{SDEthetaE}\\
&d\rho_\gl^{(\Eg)} = \big(-\sqrt{\Eg} (\gl-\Eg) \sin 2\theta_\gl^{(\Eg)}  - \frac{1}{2} \sin^2 2 \theta_\gl^{(\Eg)} + \sin^2 \theta_\gl^{(\Eg)} \big) dt + \sin 2\theta_\gl^{(\Eg)} dB^{(\Eg)}(t)\;. \label{SDErhoE}
\end{align}
In the proofs, we will sometimes drop the superscript $(\Eg)$ to alleviate the notation: we hope this will never raise any confusion.\\
Let us finally emphasize that, when working with the distorted coordinates, the parameter $E$ is not necessarily sent to $\infty$. However, in the Crossover regime where $E=E(L)\to\infty$, we will always use the distorted coordinates.
 
\subsection{Invariant measure}
The Markov process $\{\theta_\gl^{(\Eg)}\}_\pi$ admits a unique invariant probability measure whose density $\mu_\gl^{(\Eg)}$ has a simple integral expression, see Subsection \ref{Appendix:integral} of the Appendix for more details.
Straightforward computations relying on this expression show the following estimates:
\begin{lemma}[Bounds on the invariant measure]\label{Lemma:InvMeas}
\emph{Original coordinates.} For any compact interval $\Delta \subset\R$, there are two constants $0 < c < C$ such that for all $\theta\in [0,\pi)$ and all $\gl\in \Delta$ we have
$$c \le \mu_\gl(\theta) < C\;,\qquad |\partial_\theta \mu_\gl(\theta)| < C\;.$$
\emph{Distorted coordinates.} For any $h>0$ there are two constants $c,C > 0$ such that for all $E>1$, all $\theta\in [0,\pi)$ and all $\gl \in \Delta := [E-\frac{h}{En(E)},E+\frac{h}{En(E)}]$
$$c \le \mu_\gl^{(E)}(\theta) < C\;,\qquad |\partial_\theta \mu_\gl^{(E)}(\theta)| < C\;.$$
Finally, we have $\sup_{\gl \in \Delta} \sup_{\theta \in [0,\pi)} |\mu_\gl^{(E)}(\theta) - \frac1{\pi}| \to 0$ as $E \to \infty$.
\end{lemma}
The last convergence shows that our distorted coordinates are the ``right ones'' in the Crossover regime: as $E\to\infty$ the corresponding invariant measure converges to a non-degenerate limit given by the uniform measure on the circle.

\subsection{Rotation time and density of states}

Let us introduce the first rotation time of the diffusion $\theta_\gl^{(\Eg)}$
$$\zeta_\gl^{(\Eg)} := \inf\{t\ge 0: \theta_\gl^{(\Eg)}(t) = \theta_\gl^{(\Eg)}(0) + \pi\}\;.$$
By the Markov property, the law of $\zeta_\gl^{(\Eg)}$ is actually independent of the initial condition $\theta_\gl^{(\Eg)}(0)$ so that it is relevant to set
$$ m_\gl^{(\Eg)} := \E[\zeta_\gl^{(\Eg)}]\;.$$

Note that $m_\gl$ is nothing but $1/N(\gl)$ the inverse of the integrated density of states introduced in \eqref{Eq:NE}. Indeed, by the law of large numbers, the number of rotations of $\theta_\gl$ on the interval $[-L/2,L/2]$ is of order $L/m_\gl$, so that the Sturm-Liouville property recalled above implies that the number of eigenvalues of $\cH_L$ that lie below $\gl$ is of the same order.\\

For any $E\ge 1$, the mere definition of our distorted coordinates imply that $\zeta_\gl^{(E)}$ is equal in law to $\zeta_\gl / E$, so that $m_\gl^{(E)} = m_\gl / E$. Moreover, we have the following integral expression for $m_\gl^{(\Eg)}$ (see~\cite{AllezDumazTW} for more details):
\begin{align}
 m_\gl^{(\Eg)}= \frac{\sqrt{2 \pi}}{\Eg} \int_0^{\infty} \frac{1}{\sqrt{u}} \exp(-2 \lambda u - \frac{u^3}{6}) du\; \label{expressionmgl}.
\end{align}
From this expression, we deduce that $m_\gl$ is bounded from above and below uniformly over all $\gl$ in compact sets, and that
\begin{equation}\label{Eq:Asympmgl} m_\gl \sim \frac{\pi}{\sqrt\gl}\;,\quad \gl \to \infty\;.\end{equation}
This allows to recover the asymptotic of $N(\gl)$ stated in the introduction. Let us finally mention that some moment estimates on $\zeta_\gl^{(\Eg)}$ are presented in Appendix \ref{Subsec:Moments}.

\subsection{Forward and backward diffusions}\label{subsec:defforwardbackward}

We introduce in this paragraph the so-called forward and backward diffusions as their concatenation will be instrumental in our study. Let us consider the original coordinates first.

The eigenproblem is symmetric in law under the map $t\mapsto -t$ since $B'(\cdot)$ and $B'(-\cdot)$ have the same law. To take advantage of this symmetry, we consider the solutions $(r_\gl^\pm =e^{\frac12 \rho_\gl^\pm},\theta_\gl^\pm)$ of
\begin{align}
d\theta_\gl^\pm(t) &= \Big(1 + (\gl-1) \sin^2 \theta_\gl^\pm + \sin^3\theta_\gl^\pm \cos\theta_\gl^\pm\Big) dt - \sin^2 \theta_\gl^\pm dB^\pm(t)\;,\\
d\rho_\gl^\pm(t) &=\big(-(\gl-1) \sin 2\theta_\gl^\pm - \frac12 \sin^2 2\theta_\gl^\pm + \sin^2 \theta_\gl^\pm \big)dt+ \sin 2\theta_\gl^\pm dB^\pm(t)\;,\label{rho+-}
\end{align}
for $t\in [-L/2,L/2]$ where
$$ B^+(t) := B(t)\;,\quad B^-(t) := B(L/2)-B(-t)\;.$$

The processes $(r_\gl^+,\theta_\gl^+)$ will be called the \emph{forward} diffusions, while $(r_\gl^-,\theta_\gl^-)$ will be called the \emph{backward} diffusions. We also introduce $(y^\pm_\gl)' + i y^\pm_\gl := r_\gl^{\pm}\exp(i \theta^{\pm}_\gl)$. Of course the forward diffusions coincide with the original diffusions $(r_\gl,\theta_\gl)$.

\smallskip

Take $\theta_\gl^+(-L/2) = \theta_\gl^-(-L/2) = 0$. We have already seen that $\gl$ is an eigenvalue of $\cH_L$ if and only if $\{\theta_\gl^+(L/2)\}_\pi = 0$. From the symmetry of the eigenvalue problem, we can also read off the set of eigenvalues out of the backward diffusions: namely, $\gl$ is an eigenvalue if and only if $\{\theta_\gl^-(L/2)\}_\pi = 0$. 

It will actually be convenient to combine these two criteria in the following way: one follows the forward diffusion up to some given time $u \in [-L/2,L/2]$, and the backward diffusion up to time $-u$. The set of eigenvalues can then be read off using the following fact (whose proof is postponed below):
\begin{lemma}\label{Lemma:Concat}
$\gl$ is an eigenvalue of $\cH_L$ if and only if $\{\theta_\gl^+(u) + \theta_\gl^-(-u)\}_\pi = 0$.
\end{lemma}
It is therefore natural to consider concatenations of the forward and backward diffusions. For any time $u\in [-L/2,L/2]$, we set:
\begin{align}
r_\gl^+(u) = r_\gl^-(-u) = 1\;, \label{bc_r+-}
\end{align}
and we define
\begin{equation}\label{Eq:Concat} (\ru_\gl(t),\thetau_\gl(t)) := \begin{cases} (r_\gl^+(t),\theta_\gl^+(t)),\quad &t\in [-L/2,u]\;,\\
(r_\gl^-(-t),k\pi - \theta_\gl^-(-t)),\quad &t\in (u,L/2]\;,
\end{cases}
\end{equation}
where $k\pi := \lfloor\theta_\gl^+(u)+\theta_\gl^-(-u) \rfloor_\pi$. Note that $\ru_\gl,\thetau_\gl$ depend on $u$, but for notational convenience we omit writing explicitly this dependence.

\smallskip

\noindent We also define the process $\yu_\gl$ by setting
\begin{align}
\yu_\gl(t) := \ru_\gl(t) \sin \thetau_\gl(t)\;.\label{def:concaty}
\end{align}
Note that for all $t\in [-L/2,L/2] \backslash \{u\}$ we have
$$ (\yu_\gl)'(t) =  \ru_\gl(t) \cos \thetau_\gl(t)\;,$$
and that this identity remains true at $u_+$ and $u_-$, with possibly a discontinuity there. When $\gl_i$ is an eigenvalue, denoting $\|\cdot\|_2$ the $L^2$-norm, we have the identity (valid for all $u \in [-L/2,L/2]$):
\begin{align*}
 \varphi_{i}(t) = \frac{\yu_{\gl_i}(t)}{\| \yu_{\gl_i} \|_2},\qquad  t \in [-L/2,L/2]\,.
\end{align*}


\begin{proof}[Proof of Lemma \ref{Lemma:Concat}]
If $\gl$ is an eigenvalue of $\cH_L$, and provided $\{\theta_\gl^\pm(-L/2)\}_\pi = 0$, the functions $y^+_\gl$ and $y^-_\gl(-\cdot)$ coincide up to a multiplicative factor (which is related to the values $r_\gl^\pm(-L/2)$). Consequently if $\gl$ is an eigenvalue then we must have the following equality in $\{-\infty\}\cup\R$
$$ \lim_{\epsilon\downarrow 0} \frac{(y^+_\gl)'(u-\epsilon)}{y^+_\gl(u - \epsilon)} = -\lim_{\epsilon\downarrow 0} \frac{(y^-_\gl)'(-u+\epsilon)}{y^-_\gl(-u+\epsilon)}\;.$$
This is equivalent to $\cotan( \theta^+_\gl(u_-)) = \cotan(\pi-\theta^-_\gl((-u)_+))$, which is itself equivalent to $\{\theta_\gl^+(u_-)\}_\pi = \{\pi-\theta_\gl^-((-u)_+)\}_\pi$. Note that $\{\theta_\gl^+(u_-)\}_\pi = \{\theta_\gl^+(u)\}_\pi$, and similarly $\{\pi-\theta_\gl^-((-u)_+)\}_\pi = \{\pi-\theta_\gl^-(-u)\}_\pi$. We end up with $\{\theta_\gl^+(u)\}_\pi = \{\pi-\theta_\gl^-((-u))\}_\pi$ which is equivalent to $\{\theta_\gl^+(u) + \theta_\gl^-(-u)\}_\pi =0$.\\
On the other hand, if $\{\theta_\gl^+(u) + \theta_\gl^-(-u)\}_\pi = 0$ then the concatenation $\yu_\gl$ is continuously differentiable at $u$ (recall that $r_\gl^{\pm}(\pm u)=1$), satisfies \eqref{ODEH_L} at all points $t\in [-L/2,L/2]$ and satisfies the Dirichlet b.c. Consequently $\gl$ is an eigenvalue.
\end{proof}

With distorted coordinates, all the above quantities find naturally their counterparts. For $u\in [-L/(2E),L/(2E)]$, we denote by $\yu^{(E)}_\gl$, $\ru^{(E)}_\gl$, $\thetau^{(E)}_\gl$ the concatenation of the respective backward $\theta^{-,(E)}_\gl$ and forward diffusions $\theta^{+,(E)}_\gl$ on $[-L/(2E),L/(2E)]$. In particular, $\yu_\gl^{(E)} = \ru_\gl^{(E)} \sin \thetau^{(E)}_\gl$, 
$(\yu_\gl^{(E)})'/E^{3/2} = \ru_\gl^{(E)} \cos \thetau^{(E)}_\gl$ and the link between $\yu^{(E)}_{\gl_i}$ and $\varphi_i$ becomes:
\begin{align*}
\varphi_i(t) = \frac{\yu^{(E)}_{\gl_i}(t/E)}{\sqrt{E} \| \yu^{(E)}_{\gl_i}\|_2},\qquad  t \in [-L/2,L/2]\,.
\end{align*}

For both sets of coordinates and for any $u\in [-L/(2\Eg),L/(2\Eg)]$ we will denote by $\P^{(u)}_{\theta,\theta'}$ the product law $\P^+_{(-L/(2\Eg),0)\rightarrow(u,\theta)} \otimes \P^-_{(-L/(2\Eg),0)\rightarrow(-u,\theta')}$ under which $\theta^{\pm,(\Eg)}_\gl$ are two independent bridges of diffusion between $(-L/(2\Eg),0)$ and $(u,\theta)$ for $\theta_\gl^{+,(\Eg)}$, and between $(-L/(2\Eg),0)$ and $(-u,\pi-\theta)$ for $\theta_\gl^{-,(\Eg)}$. Then, the processes $\ru_\gl^{(\Eg)}$, $\thetau_\gl^{(\Eg)}$ and $\yu_\gl^{(\Eg)}$ are defined under $\P^{(u)}_{\theta,\theta'}$ according to \eqref{Eq:Concat} and \eqref{def:concaty} with the original coordinates, and according to similar equations with distorted coordinates.

\section{Strategy of proof}\label{sec:strategy}

\subsection{Convergence to equilibrium of the phase}\label{subsec:CVequilibrium}
Our proofs rely on the following exponential convergence of the transition probability of the phase $\theta_\gl^{(\Eg)}$ toward its equilibrium measure $\mu_{\gl}^{(\Eg)}$.

\begin{theorem}[Exponential convergence to the invariant measure]\label{Th:CVDensities}
In the original coordinates, fix a compact set of energies $\Delta$. Then there exist $c,C>0$ such that for all $t\ge 1$
$$ \sup_{\gl \in \Delta}\sup_{\theta_0,\theta\in [0,\pi]} |p_{\gl,t}(\theta_0,\theta) - \mu_{\gl}(\theta)| \le ce^{-C \, t}\;.$$
In the distorted coordinates, fix $h>0$ and set $\Delta = [E-\frac{h}{n(E)E},E+\frac{h}{n(E)E}]$. There exist $c,C > 0$ such that uniformly over all $E>1$ and for all $t\ge 1$
$$ \sup_{\gl \in \Delta}\sup_{\theta_0,\theta\in [0,\pi]} |p_{\gl,t}^{(E)}(\theta_0,\theta) - \mu_{\gl}^{(E)}(\theta)| \le ce^{-C \, t}\;.$$
\end{theorem}
The proof of this estimate is delicate, especially in the distorted coordinates since the bound is \emph{uniform} over all $E\in [1,\infty)$. The proof relies on tools from Malliavin calculus and the theory of hypocoercivity, it is an important technical step in our article. This result will be crucial for deriving the exponential decay of the eigenfunctions and evaluating the expectation of the number of eigenvalues in small intervals.

\subsection{Goldscheid-Molchanov-Pastur (GMP) formula}

We have seen that the Sturm-Liouville property allows to extract a lot of spectral information from the phase function. A major observation on which this article relies is that we can extract even more information through a beautiful formula, originally obtained by Goldscheid, Molchanov and Pastur~\cite{GMP1977} in a similar but smoother context, and that we name \emph{GMP formula} from now on. This formula expresses the intensity of the point process
$$ \sum_{i\ge 1} \delta_{(\gl_i, \varphi_i, \varphi'_i)}$$
in terms of \emph{concatenations} of the forward and backward diffusions introduced earlier. Roughly speaking, it shows that in average the eigenfunctions are concatenations of the forward and backward diffusions at uniform points $u \in (-L/2,L/2)$. Below $\|\cdot\|_2$ denotes the $L^2$-norm.

\begin{proposition}[GMP formula]\label{Prop:GMP}
Let $\cD$ be the Skorohod space of c\`adl\`ag functions on $[-L/2, L/2]$.
For any measurable map $G$ from $\R\times\cD \times \cD$ into $\R_+$, with the original coordinates we have
\begin{align*}
\E\big[\sum_{i\ge 1} G(\gl_i,\varphi_i, \varphi'_i)\big] &= \int_{u=-\frac{L}{2}}^{\frac{L}{2}} \int_{\gl\in\R} \int_{\theta=0}^\pi p_{\gl,\frac{L}{2}+u}(\theta) p_{\gl,\frac{L}{2}-u}(\pi-\theta)\sin^2 \theta\\
&\qquad\qquad\qquad\times\E^{(u)}_{\theta,\pi-\theta}\Big[G\Big(\gl,\frac{\yu_\gl}{\|\yu_\gl\|_2}, 
\frac{(\yu_\gl)'}{\|\yu_\gl\|_2}\Big)\Big] \, d\theta d\gl du\,,
\end{align*}
and with the distorted coordinates this becomes
\begin{align*}
\E\big[\sum_{i\ge 1} G\big(\gl_i, \varphi_i ,  \varphi'_i \big)\big] 
&= \sqrt E \int_{u=-\frac{L}{2E}}^{\frac{L}{2E}} \int_{\gl \in\R} \int_{\theta=0}^\pi p_{\gl,\frac{L}{2E}+u}^{(E)}(\theta) p_{\gl,\frac{L}{2E}-u}^{(E)}(\pi-\theta)\sin^2 \theta\\
&\qquad\qquad\times\E^{(u)}_{\theta,\pi-\theta}\Big[G\Big(\gl,\frac{\yu^{(E)}_\gl(\cdot/E)}{\sqrt{E}\|\yu_\gl^{(E)}\|_2},\frac{(\yu^{(E)}_\gl)'(\cdot/E)}{E^{3/2}\|\yu_\gl^{(E)}\|_2 }\Big)\Big] \, d\theta d\gl du\;.
\end{align*}
\end{proposition}

To exploit this formula, one needs to analyze both the transition probabilities of the phase and the concatenation $\hat y^{(\Eg)}_\gl$ for $\gl$ in the regime of energies under consideration. Regarding the transition probabilities, Theorem \ref{Th:CVDensities} allows to replace them, up to an error that vanishes as $L\to\infty$, by the equilibrium measure whose expression is explicit. On the other hand, the concatenation of the diffusions can be thoroughly studied since the SDEs at stake are tractable.\\
Taking $G = 1_\Delta(\lambda)$, we will derive a \emph{Wegner estimate}, that is an estimate on the number of eigenvalues in small (microscopic) intervals: this will be an important ingredient for the convergence to a Poisson point process, see Subsection \ref{Subsec:StrategyPoisson}. Studying carefully the concatenation $\yu^{(\Eg)}_\gl$, we will prove the exponential decay of the eigenfunctions and we will derive their asymptotic behavior, see the next paragraph.

\subsection{Exponential decay}\label{subsec:strategyexpodecay}

Let us fix again $h >0$ and let $\Delta := [E-h/Ln(E),E+h/Ln(E)]$. Introduce
\begin{align}\label{def:nuEgras}
\boldsymbol{\nu_{E}} := \begin{cases} \nu_E & \mbox{ Bulk regime}\;,\\ 1&\mbox{ Crossover regime}\;,\end{cases}
\end{align}
where $\nu_E$ was defined in \eqref{Eq:nuE}. As we will see in Subsection \ref{subsec:Lyapounov}, this quantity is related to the Lyapounov exponent of our diffusions, that is, the rate of linear growth in time of $\ln r_\gl^{(\Eg)}(t)$.

\smallskip
The main technical step towards the exponential decay is the following result:
\begin{proposition}[Exponential decay of the eigenfunctions]\label{Prop:ExpoDecayCrossover}
For any $\eps > 0$ small enough, there exists $q=q(\eps) > 0$ such that
$$ \sup_{L>1} \E\bigg[\sum_{\gl_i \in \Delta} \Big(\inf_{u\in [-L/2,L/2]} G_u(\gl_i,\varphi_{i},\varphi_{i}') \Big)^q \bigg] < \infty\;,$$
where
$$G_u(\gl,\varphi,\psi) := \sup_{t\in [-\frac{L}{2},\frac{L}{2}]} \Big(\varphi^2(t) + \frac{\psi^2(t)}{\Eg}\Big)^{1/2} \sqrt \Eg\, e^{\frac12(\boldsymbol{\nu_{E}} - \eps) \frac{|t-u|}{\Eg}}\;.$$
\end{proposition}

With this result at hand, we can present the proof of the exponential decay of the eigenfunctions.

\begin{proof}[Proof of Theorem \ref{Th:Loc}]
Fix $\eps >0$ small enough. From the last proposition, we deduce that for every eigenvalue $\lambda_i \in \Delta$ there exists $\tilde U_i$ (depending on $\eps$) such that
$$ \Big(\varphi_i(t)^2 + \frac{\varphi'_i(t)^2}{\Eg}\Big)^{1/2} \le \frac{\tilde{c_i}}{\sqrt \Eg} \exp\Big(-\frac{(\boldsymbol{\nu_{E}}-\eps)}{2}\frac{|t-\tilde U_i|}{\Eg}\Big)\;,\quad \forall t\in [-L/2,L/2]\;,$$
where $\tilde c_i := \inf_{u\in [-L/2,L/2]} G_u(\gl_i,\varphi_{\gl_i},\varphi_{\gl_i}')$. Recentering the exponential term at $U_i$, we obtain
$$ \Big(\varphi_i(t)^2 + \frac{\varphi'_i(t)^2}{\Eg}\Big)^{1/2} \le \frac{{c_i}}{\sqrt \Eg} \exp\Big(-\frac{(\boldsymbol{\nu_{E}}-\eps)}{2}\frac{|t- U_i|}{\Eg}\Big)\;,\quad \forall t\in [-L/2,L/2]\;,$$
with $c_i = \tilde c_i \exp\Big(\frac{(\boldsymbol{\nu_{E}}-\eps)}{2}\frac{|\tilde U_i- U_i|}{\Eg}\Big)$. We need some control on the distance $\tilde U_i - U_i$ to conclude the proof. From the definition of $U_i$ and since $\varphi_i^2(t) dt$ is a probability measure, we have
$$ \tilde U_i - U_i = \int (\tilde U_i -t) \varphi_i^2(t) dt\;.$$
By Jensen's inequality, we thus have for any $a \in (0, \boldsymbol{\nu_{E}}-\eps)$
\begin{align*}
\exp\Big( a \frac{|\tilde U_i - U_i|}{\Eg} \Big) &\le \int \exp\Big(a \frac{|\tilde U_i -t|}{\Eg}\Big) \varphi_i^2(t) dt\\
&\le \frac{\tilde{c_i}^2}{\Eg} \int \exp\Big((a-(\boldsymbol{\nu_{E}}-\eps)) \frac{|\tilde U_i -t|}{\Eg} \Big) dt\\
&\le 2\frac{\tilde{c_i}^2}{\boldsymbol{\nu_{E}}-\eps-a}\;.
\end{align*}
From the bound of the proposition, we already know that
$$ \sup_{L>1} \E\Big[ \sum_{\gl_i \in \Delta} \tilde c_i^q \Big] < \infty\;.$$
Using the previous bound, this remains true with $c_i$ instead of $\tilde c_i$, up to decreasing $q$.
\end{proof}

The main ideas of the proof of Proposition \ref{Prop:ExpoDecayCrossover} are simple. First of all, we apply the GMP formula to rephrase our statement on the eigenfunctions in terms of the concatenation of the forward/backward processes. Second, we establish precise moment bounds on the exponential growth of $r^{(\Eg)}_\gl$. Third, we transfer these estimates to the concatenation $\yu^{(\Eg)}_\gl$. We refer to Section \ref{Sec:ExpoDecay} for the details.

\subsection{Poisson statistics}\label{Subsec:StrategyPoisson}

Obviously, Theorem \ref{Th:Shape} implies Theorem \ref{Th:Poisson} so we concentrate on the former statement. The argument is twofold. First, we introduce an approximation $\bar{\cN}_L$ of the random measure $\cN_L$ that possesses more independence, and we prove that it converges to the Poisson random measure of the statement. Second we show that $\cN_L-\bar{\cN}_L$ goes to $0$.

\smallskip
There are some topological difficulties arising in the spaces at stake, which will be explained in more details in the proof of Theorem \ref{Th:Shape}, see below. For the time being, we view the r.v.~$\cN_L$ and $\bar{\cN}_L$ as random Radon measures on $\R\times[-1/2,1/2]\times\cM(\bar\R)$, where $\bar\cM := \cM(\bar{\R})$, the set of probability measures on $\bar\R$ endowed with the topology of weak convergence, and where $\bar\R$ is the compactification of $\R$.

\smallskip

Let us present the approximation scheme that leads to the definition of $\bar\cN_L$. We subdivide the interval $(-L/2,L/2)$ into $k$ (microscopic) disjoint boxes $(t_{j-1},t_{j})$ where $t_j = -L/2+ j L/k$ and where $k=k(L)$ is a quantity that goes to $+\infty$ at a sufficiently small speed\footnote{Restrictions on the speed of $k$ are collected in Sections \ref{Sec:TwoPoints} and \ref{Sec:Poisson}.}. We consider the Anderson Hamiltonian $\cH^{(j)}_L = -\partial^2_x + \xi$ restricted to $(t_{j-1},t_{j})$ with Dirichlet b.c., and we denote by $\lambda^{(j)}_i$ its eigenvalues, $\varphi^{(j)}_i$ its eigenfunctions, ${U_i}^{(j)}$ its centers of mass and $w_i^{(j)}(dt) = \Eg \varphi^{(j)}_i({U_i}^{(j)} + \Eg\,t)^2 dt$ the associated probability measure (after recentering at $U_i^{(j)}$).
We then define
$$ {\cN}_L^{(j)} := \sum_{i} \delta_{(L\, n(E)(\lambda_i^{(j)} - E), U_i^{(j)}/L, w_i^{(j)})}\;,$$
as well as
$$ \bar{\cN}_L := \sum_{j=1}^k {\cN}_L^{(j)}\;.$$

\begin{proposition}\label{Prop:Approx}
The random measure $\cN_L - \bar{\cN}_L$ converges in law to the null measure as $L\to\infty$.
\end{proposition}

The proof of this proposition is presented in Subsection \ref{Subsec:Approx} and relies on two inputs: (1) the eigenfunctions of $\cH_L$ and $\cH^{(j)}_L$ are exponentially localized, (2) $\bar\cN_L$ converges to a Poisson random measure. Point (1) is the content of the previous proposition (the exponential localization of the eigenfunctions of $\cH^{(j)}_L$ holds for exactly the same reasons). Point (2) is proven below. Given these two inputs, the proof of Proposition \ref{Prop:Approx} consists in building a one-to-one correspondence between the atoms of $\cN_L$ and $\bar\cN_L$, when these measures are restricted to some arbitrary set $[-h,h]\times[-1/2,1/2]\times\bar\cM$, and to show that the corresponding pairs of atoms are close in the topology at stake.

\medskip

Let us now explain the main steps towards the convergence of $\bar{\cN}_L$ to a Poisson random measure. We fix a constant $h >0$ (independent of $L$) and define the interval of energies
\begin{align}
\Delta:= \Big[E-\frac{h}{Ln(E)}, E+ \frac{h}{Ln(E)}\Big]\;. \label{def:Delta}
\end{align}
For convenience, we set
\begin{equation}\label{Eq:NLDelta} N_L(\Delta) := \#\{\lambda_i \in \Delta\} = \int \un_{[-h,h]}(\gl) \cN_L(d\gl,dx,dw)\;,\end{equation}
as well as
\begin{equation}\label{Eq:NLDeltaj} N_L^{(j)}(\Delta) :=  \#\{\lambda_i^{(j)} \in \Delta\} = \int \un_{[-h,h]}(\gl) \cN_L^{(j)}(d\gl,dx,dw)\;.\end{equation}

\medskip

First, we control the second moments of $N_L^{(j)}(\Delta)$. Note that the $N_L^{(j)}(\Delta)$ are i.i.d., so it suffices to consider $j=1$. (We also state a similar estimate on $N_L(\Delta)$ for later convenience).

\begin{proposition}\label{Prop:Moment}
We have
$$\sup_{L>1} \E[N_L(\Delta)^2] < \infty\;,\quad \sup_{L>1} k\; \E[N^{(1)}_L(\Delta)^2] < \infty\;.$$
\end{proposition}

Second, we show that, with large probability, every box $(t_{j-1},t_{j})$ contains at most one eigenvalue that lies in the interval of energy $\Delta$.
\begin{proposition}[Minami estimate]\label{Prop:OneEigen}
$$\lim_{L\to\infty} k\P(N_L^{(1)}(\Delta) \ge 2) = 0\;.$$
\end{proposition}

Our proof of this \emph{two points estimate}, which is usually named after Minami \cite{Minami}, is technically involved and appears as one of the main achievements of this paper. As mentioned in the introduction, it relies on probabilistic tools, and can be viewed as an alternative approach to the strategy of proof of Molchanov~\cite{molcanov1980} (in a smoother context). The proofs of Propositions \ref{Prop:Moment} and \ref{Prop:OneEigen} are given in Section \ref{Sec:TwoPoints}.
 
\medskip
Third, we identify the limit of the intensity measure of $\cN_L^{(j)}$. Let 
$\boldsymbol{\sigma_E}$ be the measure $\sigma_E$ (introduced in \eqref{def:sigmaE}) in the Bulk regime, and $\sigma_\infty$ (defined in \eqref{def:sigmainfty}) in the Crossover regime. 
Note that the definition in the Bulk regime will be given in Section \ref{Sec:Poisson}, while the definition in the Crossover regime was presented above Theorem \ref{Th:Shape}.
\begin{proposition}[Intensity of $\cN_L^{(j)}$]\label{Prop:Intensity}
For any compactly supported, continuous function $f:\R\times [-1/2,1/2]\times \bar\cM \to \R$, we have uniformly over all $j$ and all $L$ large enough
\begin{align*}
\E\Big[ \int f d\cN_L^{(j)} \Big] &=   \int_{\R\times [-\frac12 + \frac{j-1}{k},-\frac12 + \frac{j}{k}] \times \bar\cM} f\; d\gl \otimes dx \otimes \boldsymbol{\sigma_E} + o\big(\frac1k\big)\;.
 \end{align*}
\end{proposition}
The proof of this proposition is presented in Subsection \ref{Subsec:Intensity}: it is notationally heavy since it relies on many different arguments and definitions introduced earlier in the article, however there is no major difficulty therein.

\begin{proof}[Proofs of Theorems \ref{Th:Poisson} and \ref{Th:Shape}]
As already explained, it suffices to concentrate on the stronger statement given by Theorem \ref{Th:Shape}. Let us now explain the topological difficulties at stake. The space $\cM(\R)$ is not locally compact so that the tightness of $(\cN_L)_{L\ge 1}$ or $(\bar{\cN}_L)_{L\ge 1}$ is not elementary in $\R\times[-1/2,1/2]\times\cM(\R)$. One option would have been to prove directly this tightness (using the exponential decay of the eigenfunctions), but we prefer a simpler point of view. Namely, we deal with $\bar\cM := \cM(\bar{\R})$, the set of probability measures on $\bar\R$ endowed with the topology of weak convergence, where $\bar\R$ is the compactification of $\R$. Note that $\cM(\bar{\R})$ is compact. We then view the r.v.~$\cN_L$ and $\bar{\cN}_L$ as random Radon measures on $\R\times[-1/2,1/2]\times\cM(\bar\R)$, so that we are in a more common setting to prove convergences. It can be checked that if the convergence of Theorem \ref{Th:Shape} holds in $\R\times[-1/2,1/2]\times\cM(\bar\R)$, then it also holds in $\R\times[-1/2,1/2]\times\cM(\R)$.

Given Proposition \ref{Prop:Approx}, we only need to show that $\bar{\cN}_L$ converges to a Poisson random measure of intensity $d\gl\otimes dx \otimes \boldsymbol{\sigma_E}$. By standard criteria, see for instance~\cite[Th 14.16]{Kallenberg}, the convergence follows if we can show that for any given compactly supported, continuous function $f:\R\times [-1/2,1/2]\times \bar\cM \to \R$ we have
$$ \E\Big[\exp\Big(i \int f d\bar{\cN}_L\Big)\Big] \to \exp\Big(\int (e^{if} -1) d\gl \otimes dx \otimes \boldsymbol{\sigma_E}\Big)\;.$$
We fix such a function and concentrate on the proof of this convergence. From the independence of the $\cN_L^{(j)}$ we have
\begin{equation}\label{Eq:PoissonIndep} \E\Big[\exp\Big(i \int f d\bar{\cN}_L\Big)\Big] = \prod_{j=1}^k \E\Big[\exp\Big(i \int f d{\cN}_L^{(j)}\Big)\Big]\;.\end{equation}
Choose $h>0$ large enough such that $f(\gl,\cdot,\cdot)= 0$ whenever $\gl \notin [-h,h]$, and set $\Delta$ as in \eqref{def:Delta}. The main observation is that on the event $\{N_L^{(j)}(\Delta)=1\}$, the random measure ${\cN}_L^{(j)}$ has at most one atom on the support of $f$ and therefore on this event
$$ \exp\Big(i \int f d{\cN}_L^{(j)}\Big) - 1 = \int (e^{i  f}-1) d{\cN}_L^{(j)}\;.$$
We can therefore write
\begin{align*}
\exp\Big(i \int f d{\cN}_L^{(j)}\Big) &= \un_{\{N_L^{(j)}(\Delta) =0\}} + \un_{\{N_L^{(j)}(\Delta) \ge 1\}}  \exp\Big(i \int f d{\cN}_L^{(j)}\Big)\\
&= 1 + \un_{\{N_L^{(j)}(\Delta) \ge 1\}}  \Big( \exp\Big(i \int f d{\cN}_L^{(j)}\Big) -1\Big)\\
&= 1 + \un_{\{N_L^{(j)}(\Delta) = 1\}} \int (e^{i  f}-1) d{\cN}_L^{(j)} + \un_{\{N_L^{(j)}(\Delta) \ge 2\}} \Big(\exp\Big(i \int f d{\cN}_L^{(j)}\Big)-1\Big)\;.
\end{align*}
The third term on the r.h.s.~is bounded by a constant (independent of $j$ and $L$) times $\un_{\{N_L^{(j)}(\Delta) \ge 2\}}$, so that by Proposition \ref{Prop:OneEigen} its expectation is negligible compared to $1/k$. Define $D_j := \R\times [-\frac12 + \frac{j-1}{k},-\frac12 + \frac{j}{k}] \times \bar\cM$. Since $f(\gl,\cdot,\cdot)=0$ whenever $\gl\notin [-h,h]$, $e^{i f}-1$ is compactly supported. By Proposition \ref{Prop:Intensity} we get for the second term
\begin{align*}
\E\Big[\un_{\{N_L^{(j)}(\Delta) = 1\}} \int ( e^{i f}-1) d{\cN}_L^{(j)}\Big] &= \int_{D_j} (e^{if}-1) d\gl \otimes dx \otimes \boldsymbol{\sigma_E} + o(1/k) \\
&\quad- \E\Big[\un_{\{N_L^{(j)}(\Delta) \ge 2\}} \int ( e^{i f}-1) d{\cN}_L^{(j)}\Big]\;.
\end{align*}
Set $C:=\|e^{i f}-1\|_\infty < \infty$. We have
$$ \Big| \int_{D_j} (e^{if}-1) d\gl \otimes dx \otimes \boldsymbol{\sigma_E}\Big| \le C \frac{2h}{k}\;.$$
Moreover
\begin{align*}
\Big| \E\Big[\un_{\{N_L^{(j)}(\Delta) \ge 2\}} \int (e^{i f}-1) d{\cN}_L^{(j)}\Big] \Big| &\le C \,\E[\un_{\{N_L^{(j)}(\Delta) \ge 2\}} N_L^{(j)}(\Delta)]\\
&\le C\, \P(N_L^{(j)}(\Delta) \ge 2)^{1/2}\, \E[N_L^{(j)}(\Delta)^2]^{1/2}\;.
\end{align*}
By Propositions \ref{Prop:Moment} and \ref{Prop:OneEigen}, this last term is negligible compared to $1/k$, uniformly over all $j$ and $L$ large enough. Putting everything together, we obtain
\begin{align*}
\E\Big[\exp(i \int f d{\cN}_L^{(j)})\Big] &= 1 + \int_{D_j} (e^{if}-1) \;d\gl \otimes dx \otimes \boldsymbol{\sigma_E} +o(1/k)\\
&= (1+o(1/k)) \exp\Big(\int_{D_j} (e^{if}-1)\; d\gl \otimes dx \otimes \boldsymbol{\sigma_E}\Big) \;,
\end{align*}
uniformly over all $j$ and all $L$ large enough. Plugging this identity into \eqref{Eq:PoissonIndep}, we get the desired limit.
\end{proof}

The remaining sections are organized as follows. In Section \ref{Sec:Expo}, we prove the convergence to equilibrium stated in Theorem \ref{Th:CVDensities}. In Section \ref{Sec:GMP}, we establish the GMP formula of Proposition \ref{Prop:GMP} and collect some corollaries. In Section \ref{Sec:ExpoDecay} we prove Proposition \ref{Prop:ExpoDecayCrossover} on the exponential decay of the eigenfunctions. In Section \ref{Sec:TwoPoints} we establish the estimates on $N_L^{(j)}(\Delta)$ stated in Propositions \ref{Prop:Moment} and \ref{Prop:OneEigen}. These four sections can be read independently of each other.\\
On the other hand, Section \ref{Sec:Poisson}, where the proofs of Proposition \ref{Prop:Approx} and \ref{Prop:Intensity} are presented, relies extensively on definitions and intermediate results collected in Section \ref{Sec:ExpoDecay}.\\
In order not to interrupt the flow of arguments, we have postponed to Appendix \ref{Appendix} some technical (but elementary) results used along the article.

\section{Convergence to equilibrium}\label{Sec:Expo}

The goal of this section is to prove Theorem \ref{Th:CVDensities}. This result is delicate for three reasons. First we stated $L^\infty$ bounds on the density thus requiring much finer control than the more usual total-variation bounds. Second, the (generator of the) diffusion that we are considering is not uniformly elliptic, but simply hypoelliptic, thus making both regularization and convergence estimates delicate. Third for $E\to \infty$, the drift term of the diffusion $\theta_\gl^{(E)}$ is unbounded and makes the process rotate very fast on the circle: it is then a priori unclear whether one can obtain bounds on the density that are uniform over $E>1$.

\medskip

The proof consists of two distinct steps. First, we show quantitative regularization estimates on the density of the diffusion. The existence and the smoothness of the density at any time $t>0$ is due to the hypoellipticity of the associated generator and follows from H\"ormander's Theorem~\cite{Hormander}. However this result does not provide any quantitative estimate on this density: this is problematic in particular in the case where $E\to\infty$. We establish a quantitative estimate using Malliavin calculus (which was originally introduced to give a probabilistic proof of H\"ormander's Theorem).

\smallskip

Second, we show exponential convergence to equilibrium in $H^1$: by Sobolev embedding, this readily implies exponential convergence in $L^\infty$. From the first step, we know that at a time of order $1$, the $H^1$-norm of the density is itself of order $1$. To establish an exponential convergence to equilibrium, one would use coercivity in $H^1$ of the (adjoint in $L^2(\mu_\gl^{(\Eg)})$ of the) generator of the diffusion: however the lack of ellipticity prevents one from getting this coercivity. We thus rely on hypocoercivity techniques following Villani's monograph~\cite{Villani}: we identify a ``twisted'' $H^1$-norm, equivalent to the original one, in which the (adjoint in $L^2(\mu_\gl^{(\Eg)})$ of the) generator of the diffusion is coercive. In particular, when working with distorted coordinates, we obtain a control on the coercivity constant which is uniform over $E > 1$.

\smallskip

From now on, all the functions are defined on the circle $\R/\pi\Z$ and take values in $\R$. We will denote by $\cC^k$ the space of $k$-times continuously differentiable functions on the circle. Furthermore ``uniformly over all parameters'' will mean uniformly over all $\gl \in \Delta$ when working with the original coordinates and uniformly over all $E >1$ and $\gl\in \Delta$ with the distorted coordinates.

\begin{remark}
Let us mention that the result of Theorem \ref{Th:CVDensities} controls the exponential decay in $L^\infty$ while a control in $L^2$ would have sufficed for our purpose, however the proof would have been only marginally simpler if we had worked in $L^2$ instead of $L^\infty$.
\end{remark}

\subsection{Hypoellipticity - regularization step}\label{subsec:hypo_regularization}
We apply Malliavin calculus to the diffusion $\theta_\gl^{(\Eg)}$ following Norris~\cite{Norris} and Hairer~\cite{HairerNotes}. We drop the superscript $(\Eg)$ but we argue simultaneously in both sets of coordinates. We also drop the subscript $\gl$. We let $\P_{\theta_0}$ denote the law of the diffusion starting from $\theta_0$ at time $0$ and $p_t(\theta_0,\cdot)$ the density of this diffusion. The goal of this step is to show the following estimate: for any $k\ge 1$, there exists $C_k>0$ such that uniformly over all parameters
\begin{equation}\label{Eq:Regularization}
\sup_{\theta_0}\|p_{1}(\theta_0,\cdot)\|_{\cC^k} < C_k\;.
\end{equation}

We write
$$ d\theta(t) = V_0(\theta(t)) dt + V_1(\theta(t))dB(t)\;,$$
with
$$V_0(x) = \Eg^{3/2} + \sqrt{\Eg} (\gl-\Eg) \sin^2 x + \sin^3 x \cos x\;,\quad V_1(x) = -\sin^2 x\;.$$

\begin{remark}\label{Rk:Hormander}
The diffusion is not elliptic since $V_1$ vanishes at $x=0$. However, it satisfies the so-called H\"ormander's Bracket Condition: there exists $k\ge 0$ (here $k=2$ works) such that
$$ \inf_{x\in [0,\pi)} \max_{F\in \ccV_k} |F|(x) > 0\;,$$
where $\ccV_0 := \{V_1\}$ and $\ccV_{n+1} := \{[F,\tilde{V}_0], [F,V_1]: F\in \ccV_n\}$ where $\tilde{V}_0 = V_0 + (1/2) V'_1 V_1$.
\end{remark}

We rely on the process $J_{0,t}$ which is the derivative of the flow associated to the SDE satisfied by $\theta$:
$$ dJ_{0,t} = \partial_x V_0(\theta(t)) J_{0,t} dt + \partial_x V_1(\theta(t)) J_{0,t} dB(t)\;,\quad J_{0,0} = 1\;.$$
We will also need the inverse $J_{0,t}^{-1}$ that satisfies
$$ dJ_{0,t}^{-1}(t) = -\Big( \partial_x V_0(\theta(t)) - (\partial_x V_1(\theta(t))^2 \Big)J_{0,t}^{-1} dt - \partial_x V_1(\theta(t)) J_{0,t}^{-1} dB(t)\;,\quad J_{0,0}^{-1} = 1\;.$$
We also set $J_{s,t} := J_{0,t} J_{0,s}^{-1}$.
\begin{lemma}\label{Lemma:BoundJ}
For every $p\ge 1$, there exists $c_p > 0$ such that uniformly over all parameters
\begin{equation}\label{Eq:BoundJ}
\sup_{\theta_0}\E_{\theta_0}\Big[\sup_{s\in [0,1]} |J_{0,s}|^p\vee |J_{0,s}|^{-p}\Big]^{1/p} < c_p\;.
\end{equation}
\end{lemma}
\begin{proof}
The proof is virtually the same for $J$ and $J^{-1}$. If we let $U$ be the logarithm of either of these processes, then it solves
$$ dU(t) = \alpha(t) dt + \beta(t) dB(t)\;,$$
where $\alpha$ and $\beta$ are adapted processes. The crucial observation is that $|\alpha(t)|$ and $|\beta(t)|$ are bounded by some deterministic constant $C>0$ uniformly over all parameters: indeed, the unbounded term $E^{3/2}$ that appears in $V_0$ (when working with the distorted coordinates) is ``killed'' upon differentiation. Lemma \ref{Lemma:Mgale} then suffices to conclude.
\end{proof}

Define the so-called Malliavin covariance matrix (which is a scalar here since our diffusion is one-dimensional):
$$ \ccC_t := \int_0^t J_{0,s}^{-2} V_1^2(\theta(s)) ds\;,\quad t\ge 0\;.$$

The following is a standard result of Malliavin calculus: the only specificity is that our estimates are uniform over all parameters (in particular, w.r.t.~$E>1$ in the distorted coordinates) and this requires some extra care in the proof.
\begin{proposition}\label{Prop:Malliavin}
Assume that for every $p\ge 1$ there exists $K_p>0$ such that $\sup_{\theta_0}\E_{\theta_0}[\ccC_t^{-p}] < K_p$. Then for every $k\ge 1$ there exists $c_k > 0$ such that
$$ \sup_{\theta_0}\|p_t(\theta_0,\cdot)\|_{\cC^k} < c_k\;.$$
\end{proposition}
\begin{proof}
By standard functional analysis arguments, see for instance~\cite[Th. 0.1]{Norris}, the bound of the statement is granted if one can show that for every $k\ge 1$ there exists $C_k > 0$ such that for all $\cC^\infty$ function $G$
$$ \sup_{\theta_0}\E_{\theta_0}[\partial^k G(\theta(t))] \le C_k \|G\|_\infty\;.$$

At this point we rely on the notion of Malliavin derivative that we do not recall, see for instance~\cite[Sec.~3 and 5]{HairerNotes}. Let $Y$ be a Malliavin differentiable r.v.~and denote by $\ccD_s Y$ the Malliavin derivative at time $s$ of $Y$. We recall the following properties of the Malliavin derivative~\cite[Sec. 3]{HairerNotes}:
$$\ccD_s f(X) = f'(X) \ccD_s X\;,\quad \ccD_s (XY) = X \ccD_s Y + Y \ccD_s X\;.$$

The key tool in Malliavin Calculus is the so-called integration by parts formula which writes
$$ \E_{\theta_0}\big[\langle \ccD Y, u \rangle_{L^2([0,t])}\big] = \E_{\theta_0}\Big[Y \int_0^t u(s) dB(s)\Big]\;,$$
for any (regular enough) adapted process $u$.\\
It turns out that the Malliavin derivative of the solution to an SDE (with regular enough coefficients) admits a nice expression in terms of the derivative of the flow of the SDE. In the present case, we have (see for instance~\cite[Eq (5.6)]{HairerNotes})
$$ \ccD_s \theta(t) = J_{s,t} V_1(\theta(s)) = J_{0,t} u(s)\;,\quad u(s) := J_{0,s}^{-1} V_1(\theta(s))\;.$$
Note that $\langle\ccD \theta(t), u\rangle_{L^2([0,t])} = J_{0,t} \ccC_t =: \ccN_t$. For every smooth function $G$ and any Malliavin differentiable r.v.~$Z$, we apply the integration by parts formula to $Y=G(\theta(t)) (J_{0,t} \ccC_t)^{-1} Z$ and get
$$ \E_{\theta_0}[G'(\theta(t)) Z] = \E_{\theta_0}[G(\theta(t)) \tilde Z]\;,$$
where $\tilde Z$ is given by
$$ \tilde Z := \ccN_t^{-1} \Big(Z\int_0^t u(s) dB(s) -\langle \ccD Z, u\rangle_{L^2([0,t])} \Big) + Z \ccN_t^{-2} \langle \ccD \ccN_t, u\rangle_{L^2([0,t])}\;.$$
Applying iteratively this identity we obtain for some r.v.~$Z_k$
$$ \E_{\theta_0}[\partial^k G(\theta(t))] = \E_{\theta_0}[G(\theta(t)) Z_k]\;,$$
thus yielding
$$\sup_{\theta_0}\E_{\theta_0}[\partial^k G(\theta(t))] \le \|G\|_\infty \sup_{\theta_0}\E_{\theta_0}[|Z_k|]\;.$$

To conclude the proof, it suffices to check that the supremum on the r.h.s.~is bounded. Let us analyze the structure of $Z_k$. One can check that $Z_k$ is a polynomial in $\ccQ_t^{-1}$ and in the successive Malliavin derivatives of $\theta(t)$ and $R(t) := \int_0^t u(s) dB(s)$ in the direction $u$. The assumption of the statement together with \eqref{Eq:BoundJ} allow to bound the moments of $\ccN_t^{-1} = (J_{0,t} \ccC_t)^{-1}$. We are left with controlling the moments of the Malliavin derivatives of $\theta(t)$ and $R(t)$ in the direction $u$: we claim that these quantities are solutions of SDEs with smooth coefficients so that they admit moments of any order uniformly over the parameters (in particular, the term $E^{3/2}$ does not appear in the coefficients of these SDEs). This follows from a recursion as in~\cite[Sec.3]{Norris}, let us provide the details on the first step.\\
Note that $dR(t) = J_{0,t}^{-1} V_1(\theta(t)) dB(t)$. The triplet $(\theta(t),J_{0,t}^{-1}, R(t))$ is the solution of an SDE in $\R^3$. The derivative of the flow associated to this SDE is then a $3\times 3$ lower triangular matrix that we denote $\tilde{J}_{0,t}$:
$$ d\tilde{J}_{0,t} = \begin{pmatrix} \partial_x V_0(\theta(t)) & 0 & 0\\ A_{2,1}(t) & A_{2,2}(t) & 0\\ 0 & 0 & 0\end{pmatrix} \tilde{J}_{0,t} dt +  \begin{pmatrix} \partial_x V_1(\theta(t)) & 0 & 0\\ \partial_x^2 V_1(\theta(t)) & 0 & 0\\ J_{0,t}^{-1} \partial_x V_1(\theta(t)) & 0 & 0\end{pmatrix} \tilde{J}_{0,t} dB(t)\;,\quad \tilde{J}_{0,0} = I\;,$$
with $A_{2,1}(t) = -\Big( \partial_x^2 V_0(\theta(t)) - 2\partial_x V_1(\theta(t)) \partial_x^2 V_1(\theta(t))\Big)J_{0,t}^{-1}$ and $A_{2,2}(t) = -\Big( \partial_x V_0(\theta(t)) - (\partial_x V_1(\theta(t))^2 \Big)$. Note that all the entries are smooth and ``bounded'' (the term $E^{3/2}$ does not appear therein) so that a generalization of Lemma \ref{Lemma:BoundJ} allows to bound the moments of $\tilde{J}_{0,t}$. Moreover we have, see for instance~\cite[Eq.(5.6)]{HairerNotes}:
$$ \begin{pmatrix} \ccD_s \theta(t)\\ \ccD_s J_{0,t}^{-1} \\ \ccD_s R(t)\end{pmatrix} = \tilde{J}_{s,t} \begin{pmatrix} V_1(\theta(s))\\ \partial_x V_1(\theta(s)) \\ J_{0,s}^{-1}V_1(\theta(s))\end{pmatrix}\;,\quad s\in [0,t]\;,$$
so that the moments of $\int_0^t \ccD_s \theta(t) u(s) ds$ and $\int_0^t \ccD_s R(t) u(s) ds$ are all bounded uniformly over all parameters. One then iterates the above arguments by considering the SDE solved by
$$(\theta(t),J_{0,t}^{-1}, R(t), \int_0^t \ccD_s \theta(t) u(s) ds, \int_0^t \ccD_s J_{0,t}^{-1}(t) u(s) ds, \int_0^t \ccD_s R(t) u(s) ds)\;.$$
\end{proof}
It remains to check the assumption of Proposition \ref{Prop:Malliavin} at time $t=1$ (this is arbitrary). In the classical proof of H\"ormander's Theorem with Malliavin Calculus, this is where one uses the so-called H\"ormander's Bracket Condition (see Remark \ref{Rk:Hormander}) via repeated applications of It\^o's formula involving the process $\ccC_t$, see for instance~\cite[Proof of Theorem 6.3]{HairerNotes}. This would work with the original coordinates, but with the distorted coordinates this does not seem to work out (easily). We proceed differently and write
\begin{align*}
\E_{\theta_0}[\ccC_1^{-p}] &= \E_{\theta_0}\Big[\Big(\int_0^1 J_{0,s}^{-2} V_1(\theta(s))^2 ds\Big)^{-p}\Big]\\
&\le \E_{\theta_0}\Big[\Big(\int_0^1 V_1(\theta(s))^2 ds\Big)^{-2p}\Big]^{1/2}\E_{\theta_0}\Big[\sup_{t\in [0,1]} J_{0,t}^{4p}\Big]^{1/2}\,.
\end{align*}
Lemma \ref{Lemma:BoundJ} allows to bound the second term. To bound the first term, we need some control on the time spent near $\pi \Z$ by the diffusion $\theta$: this is provided by Lemma \ref{Lemma:ExpoBoundTheta}, that relies on a comparison of $\theta$ with the solution of the deterministic part of its SDE. This completes the proof of \eqref{Eq:Regularization}.

\subsection{Hypocoercivity - convergence step}\label{Subsec:Hypocoercivity}

In this subsection, we argue simultaneously in the two sets of coordinates. We abbreviate $\mu_\gl$ and $\mu_\gl^{(E)}$ to $\mu$. All the $L^p$ spaces will be taken w.r.t.~the measure $\mu$ (and not the Lebesgue measure). We write $\|\cdot\|$ and $\langle \cdot,\cdot\rangle$ for the $L^2$ norm and inner product.\\
By Lemma \ref{Lemma:InvMeas}, uniformly over all parameters the measure $\mu$ is equivalent to Lebesgue measure. This implies that the corresponding $L^p$ norms are equivalent, and that we have the following Poincar\'e inequality
\begin{equation}\label{Eq:Poincare} \Big\| f - \int f d\mu \Big\|^2 \lesssim \|f'\|^2\;,\end{equation}
uniformly over all parameters.\\
For some constant $c>0$ (to be adjusted later on) we define the following $H^1$ norm (w.r.t.~$\mu$)
$$ \| f\|_{H^1}^2 := \|f\|^2 + c \|f'\|^2\;.$$

Let $\cL$ be the generator of our diffusion
$$ \cL f = \sigma^2 f'' + b f'\;,$$ 
where
$$ \sigma = \frac{\sin^2 x}{\sqrt 2}\;,\quad b = \Eg^{3/2} + \sqrt \Eg(\gl-\Eg) \sin^2 x + \sin^3 x \cos x\;.$$
Let $\cLa$ be its adjoint in $L^2$. The unique decomposition of $-\cLa$ into a symmetric (actually self-adjoint) and an anti-symmetric part is given by $-\cLa= A^*A + B$ where
$$ A f = \sigma f'\;,\quad Bf = (b - 2\sigma \sigma' - \sigma^2 \frac{\mu'}{\mu})f'\;,$$
and $A^*$ is the adjoint in $L^2$ of $A$
$$ A^* f = - \frac{(\sigma \mu f)'}{\mu}\;.$$

It is a standard fact that the centered density $q_t = (p_t / \mu - 1)$ w.r.t.~the invariant measure of a diffusion satisfies the PDE
$$ \partial_t q_t = \cLa q_t\;.$$
In the previous section we dealt with regularity issues and showed that $p_1$ (and therefore $q_1$ since $\mu$ is bounded from above and below) is smooth: more precisely, its $\cC^1$-norm is bounded uniformly over all parameters. Our goal is to prove an exponential decay of the $L^\infty$ norm of $p_t$. By Sobolev Embedding it suffices to prove an exponential decay in $H^1$.\\

The natural route to such an estimate is to prove some coercivity: unfortunately this fails. Indeed, we have
$$ \langle f, \cLa f\rangle = - \| A f\|^2\;,$$
and since $\sigma$ vanishes at $0$, $\| Af\|$ is \emph{not} equivalent to $\| f'\|$ and so we cannot use Poincar\'e inequality to get a bound in terms of $\|f\|$. Actually it is possible to check that the operator $A^*A$ does not have spectral gap so that one cannot get exponential decay of the $L^2$-norm of $q_t$ from the previous computation. In $H^1$ the situation is similar. Generally speaking, if we work with ``standard'' norms then we do not have enough control on the derivative(s) of the function at stake near $0$.

\bigskip

These computations do not take advantage of the anti-symmetric part $B$ of $\cLa$. The general idea of hypocoercivity~\cite{Villani} consists in exploiting the successive Lie brackets between $A$ and $B$ to recover some coercivity. One can check that $[A,B] f := AB f-BA f$ contains the term $-\Eg^{3/2} \sigma' f'$ and $[[A,B],B] f$ contains the term $\Eg^3 \sigma'' f'$: since $\sigma''$ does not vanish anymore at $x=0$, these terms should offer the required control on the derivative at $0$. To implement this idea, one constructs a twisted $H^1$-norm, denoted $\mathfrak{H}^1$, that contains some successive Lie brackets of $A$ and $B$, see below for the precise expression, and that satisfies the following properties.

\begin{proposition}\label{Prop:Hypocoercivity}
There exists a norm $\|\cdot \|_{\mathfrak{H}^1}$, derived from an inner product $\langle \cdot,\cdot\rangle_{\mathfrak{H}^1}$ such that:\begin{enumerate}
\item There exists $\kappa>0$ such that uniformly over all parameters we have
$$ \|f\|_{H^1} \le \| f \|_{\mathfrak{H}^1} \le (1+\kappa) \|f\|_{H^1}\;.$$
\item There exists $K>0$ such that uniformly over all parameters we have
$$\langle f, \cLa f\rangle_{\mathfrak{H}^1} \le -K \|f'\|^2\;.$$
\end{enumerate}
\end{proposition}
With this proposition at hand, we can easily conclude the proof of the main result of this section.
\begin{proof}[Proof of Theorem \ref{Th:CVDensities}]
In the previous subsection, we showed that there exists $C_1 > 0$ such that uniformly over all parameters
\begin{equation}\label{Eq:ControlC1}
\sup_{\theta_0}\|p_1(\theta_0,\cdot)\|_{\cC^1} < C_1\;.
\end{equation}
Set $q_t := \frac{p_t}{\mu} - 1$ and recall that $\partial_t q_t = \cLa q_t$. The second property of the proposition then yields
$$ \partial_t \|q_t\|_{\mathfrak{H}^1}^2 \le - 2K \|q_t'\|^2\;.$$
Since $\int q_t d\mu = 0$, Poincar\'e inequality \eqref{Eq:Poincare} together with the first property of the proposition show that for some constant $K'>0$
 $$ \partial_t \|q_t\|_{\mathfrak{H}^1}^2 \le - K'  \|q_t\|^2_{\mathfrak{H}^1}\;.$$
Applying the first property again, we deduce that uniformly over all the parameters
$$ \|q_t\|_{H^1} \le  c'\|q_1\|_{H^1} e^{-C'(t-1)}\;,$$
for some constants $c',C'>0$. Combining this bound with \eqref{Eq:ControlC1}, applying the Sobolev embedding $H^1(dx) \subset L^\infty(dx)$ and the fact that $\mu$ is equivalent to the Lebesgue measure, we obtain the desired result.
\end{proof}
The rest of this subsection is devoted to the proof of Proposition \ref{Prop:Hypocoercivity}. We start by introducing successive Lie brackets: actually we identify within the successive Lie brackets between $A$ and $B$ the terms $\Eg^{3/2} C_k$ that will allow us to gain coercivity and we denote the remainder $R_k$. Introduce $C_0 f := A f$, and recursively for $k=1,2$
$$ C_k f := (-1)^k \sigma^{(k)} f'\;,\quad R_k f:= [C_{k-1}, B] f - \Eg^{3/2} C_k f\;,$$
and finally
$$ C_3 f := 0\;,\quad R_3 f := [C_2,B] f\;.$$

We now introduce a family of coefficients indexed by $\delta \in (0,1/4)$. Set $b_{-1} := \delta$ and for every $k\ge 0$
$$ a_k := \delta^{6-2k} b_{k-1}\;,\quad b_k := \delta^{5-2k}a_k\;.$$
One can check that for all $0\le k \le 2$
\begin{equation}\label{Eq:akbk1}
b_k \le \delta a_k\;,\quad a_{k} \le \delta b_{k-1}\;,\quad  a_k^2 = {\delta b_{k-1} b_k}\;,\quad b_k^2 = {\delta a_k a_{k+1}}\;.
\end{equation}

We then introduce 
$$ \| f \|_{\mathfrak{H}^1}^2 :=\|f\|_{H^1}^2 + \frac1{\Eg^{3/2}}\sum_{k=0}^2 \Big( a_k\|C_k f\|^2 + 2b_k \langle C_k f, C_{k+1} f\rangle\Big)\;,$$
as well as $\langle f,g \rangle_{\mathfrak{H}^1} = \langle f,g \rangle_{{H}^1} +  \llrrbracket{f,g}$ where
$$ \llrrbracket{f,g} := \frac1{\Eg^{3/2}}\sum_{k=0}^2 \Big( a_k \langle C_k f, C_k g\rangle + b_k (\langle C_k f, C_{k+1} g\rangle + \langle C_k g, C_{k+1} f \rangle) \Big)\;.$$
This norm depends on two parameters, $c$ and $\delta$, that will be adjusted later on.

\medskip

The bound $b_k^2 \le \delta a_k a_{k+1}$ implies
$$ \Big|2b_k \langle C_k f, C_{k+1} f\rangle\Big| \le \delta a_k \|C_k f\|^2 + \delta a_{k+1} \|C_{k+1} f\|^2\;.$$
Note that for every $k\in\{0,1,2,3\}$, we have $\|C_k f\| \le 2\|f'\|$. Since $\Eg \ge 1$, we easily deduce the first property of Proposition \ref{Prop:Hypocoercivity}.

%
\medskip

The proof of the second property is carried out in two steps. First, we show that there exists a constant $C>0$ such that uniformly over all parameters,
\begin{equation}\label{Eq:BoundNabla}
\langle f',(\cLa f)'\rangle \le C \| f'\|^2\;.
\end{equation}
Second we show that there exists $\delta > 0$ and $K'>0$ such that uniformly over all parameters
\begin{equation}\label{Eq:BoundVillani}
\langle f,\cLa f\rangle + \llrrbracket{f,\cLa f} \le -K' \| f'\|^2\;.
\end{equation}
With these two bounds at hand, we get
$$\langle f, \cLa f\rangle_{\mathfrak{H}^1}\le(c\,C - K') \|f'\|^2 \;.$$
It suffices to set $c:= K' (2C)^{-1}$ and $K=K'/2$ in order to deduce the second property of the proposition.

\begin{remark}
Our proof is close to~\cite[Proof of Theorem 24]{Villani} and would be essentially the same if we had taken $\|f\|_2^2$ instead of $\|f\|_{H^1}^2$ in the definition of the $\mathfrak{H}_1$ norm. Actually, with the original coordinates, the proof would carry through. With the distorted coordinates however, the first property of Proposition \ref{Prop:Hypocoercivity} would fail.
\end{remark}

We proceed with the proof of \eqref{Eq:BoundNabla}. For convenience, we define the operator $Df=f'$. We then have
$$ \langle f',(\cLa f)'\rangle = -\langle Df, DA^*A f \rangle - \langle Df, DBf\rangle\;.$$
Since $B$ is anti-symmetric we find $\langle Df, DBf\rangle = \langle Df, [D,B] f\rangle$. Note that $[D,B] f = (b - 2\sigma\sigma' - \sigma^2 \frac{\mu'}{\mu})' f'$. The key point here is that the only unbounded factor ($\Eg^{3/2}$ which appears in $b$) is killed upon differentiation. A simple computation shows that there exists $C_1$ such that
\begin{equation}\label{Eq:Nabla1}
\big|\langle Df, [D,B] f\rangle\big| \le C_1 \| f'\|^2\;.
\end{equation}
On the other hand, we have
\begin{align*}
\langle Df, DA^*A f \rangle &= -\langle D^2 f, A^*A f \rangle - \langle Df, \frac{\mu'}{\mu} A^*A f\rangle\\
&= \langle \sigma D^2 f, \sigma D^2 f\rangle + \langle D^2 f, \frac{(\sigma^2 \mu)'}{\mu} Df\rangle - \langle Df, \frac{\mu'}{\mu} A^*A f\rangle\;.
\end{align*}
Recall the bounds on $\mu$ stated in Lemma \ref{Lemma:InvMeas}. There exists $C_2>0$ such that
$$ \big|\langle D^2 f, \frac{(\sigma^2 \mu)'}{\mu} Df\rangle\big| \le C_2 \| \sigma D^2 f\| \| Df\|\;,$$
and
$$ \big|\langle Df, \frac{\mu'}{\mu} A^*A f\rangle\big| \le C_2 (\|\sigma Df \| \|\sigma D^2f\| + \| Df\|^2)\;.$$
By Young's inequality, we deduce that there exists $C_3 > 0$ such that
\begin{equation}\label{Eq:Nabla2}
\langle Df, DA^*A f \rangle \ge \frac12 \langle \sigma D^2 f, \sigma D^2 f\rangle - C_3 \| Df\|^2\;.
\end{equation}
Combining \eqref{Eq:Nabla1} and \eqref{Eq:Nabla2} we obtain \eqref{Eq:BoundNabla}.

\medskip

We turn to the proof of \eqref{Eq:BoundVillani}, which is very close to~\cite[Proof of Theorem 24]{Villani}. The main difference lies in the unboundedness of the coefficients of the operator at stake, namely the term $E^{3/2}$ in $B$ in the distorted coordinates, that requires some additional care.

First of all, we say that an operator $Q$ is bounded relatively to some operators $(E_j)_j$ if
$$ \| Q f\|^2 \lesssim \sum_{j} \| E_j f\|^2\;,$$
uniformly over all parameters.

It is tedious but elementary to check that:\begin{enumerate}[label=(\roman*)]
\item For every $0\le k \le 2$, $[A,C_k]$ is bounded relatively to $\{C_j\}_{0\le j \le k}$ and $\{C_j A\}_{0\le j \le k-1}$,
\item For every $0\le k \le 2$, $[C_k,A^*]$ is bounded relatively to $I$ and $\{C_j\}_{0\le j \le k}$,
\item For every $1\le k \le 2$, $R_k$ is bounded relatively to $\{C_j\}_{0\le j \le k-1}$ and $\{C_j A\}_{0\le j \le k-1}$,
\item $R_3$ is bounded relatively to $\Eg^{3/2} C_1$ and $\Eg^{3/2} C_2$.
\end{enumerate}

Recall that $-\cLa=A^*A + B$. With self-evident notations, we have
\begin{equation}\label{Eq:cLaf}
\llrrbracket{f,-\cLa f} = \frac1{\Eg^{3/2}} \sum_{k=0}^2 \Big(a_k \big[(I)^k_A + (I)^k_B] + b_k[(II)^k_A + (II)^k_B\big]\Big)\;.
\end{equation}
Linear algebra manipulations and the Cauchy-Schwarz inequality yield (see~\cite[Proof of Theorem 24, pp.36-37]{Villani} for details, one simply needs to $\lambda_k=\Lambda_k=\Eg^{3/2}$ therein):
\begin{align*}
(I)^k_A \ge\,& \|C_k A f\|^2 - \|C_kAf\| \|[A,C_k]f\| - \|C_kf\|\|[C_k,A^*]Af\|\;,\\
(I)^k_B \ge\,& -\Eg^{3/2}\|C_kf\|\|C_{k+1}f\| - \|C_k f\| \|R_{k+1} f\|\;,\\
(II)^k_A \ge\,& - \|C_kf\| \|[C_{k+1},A^*]Af\| - \|C_kAf\| \|C_{k+1} Af\| - \|C_{k+1} Af\| \|[A,C_k] f\|\\
& - \|C_k Af\| \|C_{k+1} Af\| - \|C_k Af\|\|[A,C_{k+1}] f\| - \|C_{k+1} f\| \|[C_{k},A^*]Af\|\;,\\
(II)^k_B \ge\,& \Eg^{3/2} \|C_{k+1} f\|^2 - \|C_{k+1} f\| \|R_{k+1} f\| - \Eg^{3/2} \|C_k f\| \| C_{k+2} f\| - \|C_k f\| \|R_{k+2} f\|\;.
\end{align*}
At this point, the proof differs from~\cite[Proof of Theorem 24, pp.38]{Villani}: indeed therein the parameters $\lambda_k$, $\Lambda_k$ are of order $1$, while in our case they are taken equal to $\Eg^{3/2}$ and therefore need extra care.\\

Take $\epsilon =\sqrt\delta$. By Young's inequality we have $XY \le \epsilon X^2 + (4\epsilon)^{-1} Y^2$. We thus get for any $k\in\{0,1,2\}$
\begin{align}
a_k \big[(I)^k_A + (I)^k_B] + b_k[(II)^k_A + (II)^k_B\big] &\ge a_k \| C_k A f\|^2 + b_k \Eg^{3/2} \|C_{k+1} f\|^2\\
\label{5:8}&-\epsilon \Eg^{3/2} b_{k-1} \|C_k f\|^2 - \Eg^{3/2} \frac{a_k^2}{4\epsilon b_{k-1}} \|C_{k+1} f\|^2\\
&-\epsilon  b_{k-1} \|C_k f\|^2 - \frac{a_k^2}{4\epsilon b_{k-1}} \|R_{k+1} f\|^2\\
&-\epsilon a_{k} \|C_k A f\|^2 - \frac{a_k}{4\epsilon} \|[A,C_{k}] f\|^2\\
&-\epsilon b_{k-1} \|C_k f\|^2 - \frac{a_k^2}{4\epsilon b_{k-1}} \|[C_{k},A^*]A f\|^2\\
\label{5:12}&-\epsilon b_{k} \|C_{k+1} f\|^2 - \frac{b_k}{4\epsilon} \|R_{k+1} f\|^2\\
\label{5:13}&-\epsilon \Eg^{3/2} b_{k-1} \|C_k f\|^2 - \Eg^{3/2} \frac{b_k^2}{4\epsilon b_{k-1}} \|C_{k+2} f\|^2\\
\label{5:14}&-\epsilon b_{k-1} \|C_k f\|^2 - \frac{b_k^2}{4\epsilon b_{k-1}} \|R_{k+2} f\|^2\\
&-\epsilon b_{k-1} \|C_k f\|^2 - \frac{b_k^2}{4\epsilon b_{k-1}} \|[C_{k+1},A^*]A f\|^2\\
&-\epsilon a_{k} \|C_k A f\|^2 - \frac{b_k^2}{4\epsilon a_k} \|C_{k+1}A f\|^2\\
&-\epsilon a_{k+1} \|C_{k+1} A f\|^2 - \frac{b_k^2}{4\epsilon a_{k+1}} \|[A,C_k] f\|^2\\
&-\epsilon a_{k} \|C_k A f\|^2 - \frac{b_k^2}{4\epsilon a_k} \|C_{k+1}A f\|^2\\
&-\epsilon a_{k} \|C_k A f\|^2 - \frac{b_k^2}{4\epsilon a_k} \|[A,C_{k+1}] f\|^2\\
\label{5:20}&-\epsilon b_{k} \|C_{k+1} f\|^2 - \frac{b_k^2}{4\epsilon a_k} \|[C_{k},A^*]A f\|^2\;,
\end{align}
with the further condition that \eqref{5:12} to \eqref{5:20}, as well as \eqref{5:8}, are not present for $k=2$.\\
The inequalities \eqref{Eq:akbk1} ensure that for $\delta$ small enough (uniformly over all parameters) the sum over $k$ of all the terms from \eqref{5:8} to \eqref{5:20}, except the first terms of \eqref{5:8} and \eqref{5:13} for $k=0$,  is larger than
$$-\frac14 \|A f\|^2 - \frac12 \sum_{k=0}^{2} \Big(a_k \| C_k A f\|^2 + b_k \Eg^{3/2} \|C_{k+1} f\|^2\Big)\;.$$
For more details see~\cite[Proof of Theorem 24, pp.38-39]{Villani}. We now control the contribution to \eqref{Eq:cLaf} of the first terms of \eqref{5:8} and \eqref{5:13} for $k=0$. This contribution is given by (recall that $\delta < 1/4$):
$$ -2\epsilon b_{-1} \Eg^{3/2} \|A f\|^2 = -2\delta^{3/2} \Eg^{3/2} \|Af\|^2 \ge -\frac14 \Eg^{3/2}\|Af\|^2\;.$$
Consequently
$$ \llrrbracket{f,-\cLa f}  \ge -\frac14(1+\frac1{\Eg^{3/2}}) \|Af\|^2 + \frac1{2\Eg^{3/2}} \sum_{k=0}^2 \Big(a_k \| C_k A f\|^2 + b_k \Eg^{3/2} \|C_{k+1} f\|^2\Big)\;.$$
On the other hand, $\langle f, -\cLa f\rangle = \|Af\|^2$. Putting everything together, we thus showed that
\begin{align*}
\langle f, -\cLa f\rangle + \llrrbracket{f,-\cLa f}  &\ge \frac12 \| Af \|^2  + \frac1{2\Eg^{3/2}} \sum_{k=0}^2 \Big(a_k \| C_k A f\|^2 + b_k \Eg^{3/2} \|C_{k+1} f\|^2\Big)\\
&\ge \frac12 \| Af \|^2  + \frac12 \sum_{k=0}^1 b_k \|C_{k+1} f\|^2\\
&\ge \frac14 \int_{[0,\pi)} \big(\sin^4 x + b_0 \sin^2 (2x) + 4 b_1 \cos^2 (2x)\big) f'(x)^2 \mu(x)dx\\
&\ge K \|f'\|^2\;,
\end{align*}
for some $K$ only depending on $\delta$. This completes the proof.

\section{GMP formula}\label{Sec:GMP}

In this section, we prove the GMP formula stated in Proposition \ref{Prop:GMP} and we deduce some simple facts of this formula. Actually, we will prove a slightly stronger statement for later convenience:

\begin{proposition}\label{Prop:GMPalternative}
Fix $u \in (-L/(2 \Eg),L/(2 \Eg))$ and $a,b\in \R$. For any bounded and measurable map $G$ from $\R\times\cD\times\cD$ into $\R_+$, we have
\begin{align*}
&\E\bigg[\sum_{i\ge 1} \big(a  \varphi_i( \Eg u)^2 + b\; \varphi_i'(\Eg u)^2\big)\;G(\gl_i,\;\varphi_i,\varphi_i')\bigg] \\
&= \frac1{\sqrt \Eg}\int_{\gl\in\R} \int_{\theta=0}^\pi p^{(\Eg)}_{\gl,\frac{L}{2 \Eg}+u}(\theta) p^{(\Eg)}_{\gl,\frac{L}{2 \Eg}-u}(\pi-\theta) (a\sin^2\theta + b\,\Eg \cos^2 \theta)\\
&\qquad \qquad \qquad \qquad \times \E^{(u)}_{\theta,\pi-\theta}\bigg[G\Big(\gl,\frac{\yu^{(\Eg)}_\gl(\cdot/\Eg)}{\sqrt{\Eg}\|\yu^{(\Eg)}_\gl\|},\frac{(\yu^{(\Eg)}_\gl)'(\cdot/\Eg)}{\Eg^{3/2} \|\yu^{(\Eg)}_\gl\|}\Big)\bigg] \, d\theta d\gl\\
\end{align*}
\end{proposition}

\begin{remark}
There's an ambiguity in the statement of this proposition (and of Proposition \ref{Prop:GMP}): the eigenfunctions $\varphi_i$ are uniquely defined up to a sign change while the functions $\yu_\gl$ are positive at $(-L/2)_+$. For the proposition to be totally rigorous one needs to assume further that $\varphi_i$ is positive at $(-L/2)_+$. Actually the maps $G$ that we will consider in the sequel are invariant under changing the sign of $\varphi_i$ so this has no impact.
\end{remark}

Given this proposition, the proof of the GMP formula is simple.
\begin{proof}[Proof of Proposition \ref{Prop:GMP}]
Recall that the eigenfunctions $\varphi_i$ are normalized in $L^2$. It suffices to take $a=1$, $b=0$ in the previous proposition, to integrate w.r.t.~$u$ and to apply Fubini's Theorem.
\end{proof}

Before we proceed to the proof of Proposition \ref{Prop:GMPalternative}, note that $\gl \mapsto \theta_\gl^{(\Eg)}$ is differentiable and let us introduce $z_\gl^{(\Eg)} = \partial_\gl \theta_\gl^{(\Eg)}$, the derivative with respect to $\gl$ of the angle $\theta_\gl^{(\Eg)}$. It satisfies the SDE
\begin{align*}
d z_\gl^{(\Eg)}(t) &= \sqrt\Eg \sin^2 \theta_\gl^{(\Eg)} dt - z_\gl^{(\Eg)} \Big[ d \rho_\gl^{(\Eg)}(t) - \frac12 d \langle \rho_\gl^{(\Eg)} \rangle_t \Big]\;.
\end{align*}
We thus obtain the following integral expression for $z_\gl$:
\begin{align} \label{Eq:zgl}
z_\gl^{(\Eg)}(t) = \frac{1}{(r_\gl^{(\Eg)}(t))^2}\int_{-\frac{L}{2 \Eg}}^t \sqrt\Eg \,\sin^2 \theta_\gl^{(\Eg)} (s) (r_\gl^{(\Eg)}(s))^2 ds\;.
\end{align}

\begin{proof}[Proof of Proposition \ref{Prop:GMP}]
It suffices to prove the statement with the original coordinates, since the statement with the distorted coordinates follows from \eqref{Eq:CoVp} and the change of variable $\theta\mapsto \arccotan(\frac1{\sqrt E} \cotan \theta)$ whose Jacobian is given by $(\sqrt E (\sin^2 \theta + \frac1{E} \cos^2 \theta))^{-1}$.

By the monotone convergence theorem, we can assume that $G$ is a continuous map from $\R\times\cD \times \cD$ into $\R_+$ such that $G(\gl,\varphi,\psi) = 0$ whenever $\gl \notin [-A,A]$ or $\|\varphi\|_{L^\infty} > A$ or  $\|\psi\|_{L^\infty} > A$, for some given $A>0$. Fix $u \in (-L/2,L/2)$. Recall from Lemma \ref{Lemma:Concat} that the set of eigenvalues $(\lambda_i)_{i\ge 1}$ coincides with the set of $\gl \in\R$ such that
$$ \big\{ \theta_\gl^{+}(u)+\theta_\gl^{-}(-u) \big\}_\pi = 0\;,$$
and the eigenfunction $\varphi_\gl$ associated to the eigenvalue $\gl$ coincides with $\yu_\gl/\|\yu_\gl\|$. Therefore,
\begin{align*}
&\sum_{i\ge 1} G(\gl_i,\varphi_i, \varphi'_i)  \big(a \; \varphi_i(u)^2 + b\; \varphi_i'(u)^2\big)\\
&= \sum_{\gl \;:\;\big\{ \theta_\gl^{+}(u)+\theta_\gl^{-}(-u) \big\}_\pi = 0} 
G\Big(\gl,\frac{\yu_\gl}{\|\yu_\gl\|},\frac{(\yu_\gl)'}{\|\yu_\gl\|}\Big)\; \frac{a (\yu_\gl(u))^2 + b ({{\yu_{\gl}}'(u)})^2}{\|\yu_\gl\|^2}\,.
\end{align*}
Almost surely the map $\gl \mapsto \theta^{+}_\gl(u) + \theta^{-}_\gl(-u)$ is a diffeomorphism from $\R$ to $\R$. We denote its inverse by $\theta\mapsto \gl(\theta)$. We can rewrite the last sum as:
\begin{align*}
\sum_{\theta \in \pi \Z} G\Big(\gl(\theta),\frac{\yu_{\gl(\theta)}}{\|\yu_{\gl(\theta)}\|},\frac{(\yu_{\gl(\theta)})'}{\|\yu_{\gl(\theta)}\|}\Big)\; \frac{a( \yu_{\gl(\theta)}(u))^2 + b ({{\yu_{\gl(\theta)}}'(u)})^2}{\|\yu_{\gl(\theta)}\|^2 } \,.
\end{align*}
Set for any $\eps \in (0,\pi)$
\begin{equation}\label{Eq:RVeps}
{G}_{\eps}(u) := \frac1{2\eps} \int_{\theta \in \pi\Z + [-\eps,\eps]} G\Big(\gl(\theta),\frac{\yu_{\gl(\theta)}}{\|\yu_{\gl(\theta)}\|}, \frac{(\yu_{\gl(\theta)})'}{\|\yu_{\gl(\theta)}\|}\Big) \frac{a( \yu_{\gl(\theta)}(u))^2 + b ({{\yu_{\gl(\theta)}}'(u)})^2}{\|\yu_{\gl(\theta)}\|^2} \;d\theta\;.
\end{equation}
Almost surely ${G}_{\eps}(u)$ is bounded by $\|G\|_\infty(|a|+|b|) A^2 \big(\#\{\lambda_i \in [-A,A]\} +2\big)$. From Lemma \ref{Eq:BdExpoDS} in the Appendix, this r.v.~has a finite expectation. Furthermore, by continuity ${G}_{\eps}(u)$ converges a.s.~as $\eps\downarrow 0$ to
$$\sum_{i\ge 1} G(\gl_i,\varphi_i, \varphi'_i)\, \big(a\varphi_i(u)^2 + b ({\varphi_i'(u)})^2\big)\;.$$
By the Dominated Convergence Theorem, we deduce that,
\begin{align*}
\E\Big[\sum_{i\ge 1} G(\gl_i,\varphi_i, \varphi'_i)\, \big(a\varphi_i(u)^2 + b ({\varphi_i'(u)})^2\big) \Big]= \lim_{\eps\downarrow 0} \E\big[{G}_{\eps}(u) \big]\;.
\end{align*}
We now compute the expectation of ${G}_{\eps}(u)$. Note that $\partial_\gl(\theta_\gl^{+}(u)+\theta_\gl^{-}(-u)) = z_\gl^+(u) + z_\gl^-(-u)$. Using \eqref{Eq:zgl} and given the boundary condition imposed on $r_\gl^\pm$, this equals $\|\yu_\gl\|^2$. We then apply the change of variable $\theta \mapsto \gl(\theta)$ and obtain
\begin{align*}
{G}_{\eps}(u)= \frac1{2\eps}&\int_{\gl\in\R} \un_{\{\pi\Z + [-\eps,\eps]\}}(\thetau_\gl(u_-)+\thetau_\gl(u_+)) G\Big(\gl,\frac{\yu_\gl}{\|\yu_\gl\|}, \frac{(\yu_\gl)'}{\|\yu_\gl\|}\Big) \\
&\qquad\qquad\times\big( a \sin^2 \theta_\gl^+(u) + b \cos^2 \theta_\gl^+(u)\big)d\gl\;.
\end{align*}

We then take expectation, use Fubini's Theorem, and integrate with respect to the value of the forward diffusion at time $u$:
\begin{align*}
\E[{G}_{\eps}(u)] &=\frac1{2\eps}\int_{\gl\in\R} \int_{\theta =0}^{\pi} \int_{\theta' = -\eps}^{\eps} p_{\gl,u}(\theta) p_{\gl,L-u}(\pi-\theta + \theta') (a\sin^2 \theta + b\cos^2\theta) \\
&\times\E^{(u)}_{\theta,\pi - \theta - \theta'}\Big[G\Big(\gl,\frac{\yu_\gl}{\|\yu_\gl\|}, \frac{(\yu_\gl)'}{\|\yu_\gl\|}\Big) \Big]d\theta' d\theta d\gl\;.
\end{align*}
By the Dominated Convergence Theorem, the boundedness of the transition probabilities and the continuity w.r.t.~$\theta'$, we can permute the integrals with respect to $\gl,\theta$ and the limit as $\eps\downarrow 0$ to find
\begin{align*}
\lim_{\eps\downarrow 0} \E[{G}_{\eps}(u)] =& \int_{\gl\in\R} \int_{\theta=0}^\pi p_{\gl,u}(\theta) p_{\gl,L-u}(\pi-\theta) (a\sin^2\theta + b\cos^2\theta)\\
&\qquad \times\E^{(u)}_{\theta,\pi - \theta}\Big[G\Big(\gl,\frac{\yu_\gl}{\|\yu_\gl\|}, \frac{(\yu_\gl)'}{\|\yu_\gl\|}\Big) \Big] \, d\theta d\gl\;,
\end{align*}
thus concluding the proof.
\end{proof}

As a first application of the GMP formula, one can derive an expression of the density of states in terms of the stationary measure of the process $\{\theta_\gl\}_\pi$. There is another (actually simpler) formula for this density of states given by $n(\gl) = \partial_\gl (1/m_\gl)$ where $m_\gl$ is introduced in Subsection \ref{Subsec:Moments}, see~\cite{Fukushima}.

\begin{corollary}\label{cor:DOSformula}
The density of states writes
$$ n(\gl) = \int_{0}^\pi \mu_{\gl}(\theta) \mu_{\gl}(\pi-\theta)\sin^2 \theta d\theta\;,\quad \gl\in\R\;.$$
We also have for any $E\ge 1$
$$ n(\gl) = \frac1{\sqrt E} \int_{0}^\pi \mu_{\gl}^{(E)}(\theta) \mu_{\gl}^{(E)}(\pi-\theta)\sin^2 \theta d\theta\;,\quad \gl\in\R\;.$$
\end{corollary}
\begin{proof}
Fix some interval $\Delta \subset \R$ and set $N_L(\Delta):= \#\{\lambda_i: \lambda_i \in \Delta\}$. By the uniform integrability stated in Lemma \ref{Eq:BdExpoDS}, $\int_\Delta n(\gl) d\gl$ is not only the a.s.~limit of $N_L(\Delta)/L$ but also its limit in $L^1$. The GMP formula allows to compute $\E[N_L(\Delta)]$ by simply choosing $G(\gl)=\un_\Delta(\gl)$. By Theorem \ref{Th:CVDensities} one can then replace the product of the densities of the forward/backward diffusions $p_{\gl,\frac{L}{2}+u}(\theta) p_{\gl,\frac{L}{2}-u}(\pi-\theta)$ by the product of the densities of the invariant measure $\mu_{\gl}(\theta) \mu_{\gl}(\pi-\theta)$, and derive the first expression of the statement. The formula involving distorted coordinates follows from a change of variables.
\end{proof}

Recall the r.v.~defined in \eqref{Eq:NLDelta} and \eqref{Eq:NLDeltaj}.

\begin{proposition}[Wegner estimates]\label{Prop:MicroEstimates} Fix $h>0$ and set $\Delta := [E- h/(L n(E)),E+h/(Ln(E))]$. In the Bulk and the Crossover regimes, we have as $L\to\infty$:
\begin{enumerate}
\item $\E[N_L(\Delta)] = 2h (1+o(1))$,
\smallskip
\item $\E[N_L^{(1)}(\Delta)] = \frac{2h}{k}(1+o(1))$.
\end{enumerate}
\end{proposition}
\begin{proof}
By Proposition \ref{Prop:GMP}, we have:
\begin{align*}
\E\big[N_L(\Delta)\big]&= \sqrt \Eg \int_{-\frac{L}{2\Eg}}^{\frac{L}{2\Eg}} \int_{\gl\in\Delta} \int_{\theta=0}^\pi p_{\gl,\frac{L}{2\Eg}+u}^{(\Eg)}(\theta) p_{\gl,\frac{L}{2\Eg}-u}^{(\Eg)}(\pi-\theta)\sin^2 \theta d\theta d\gl du
\end{align*}
If one replaces the transition probabilities by the equilibrium densities, then one gets
$$  \frac{L}{\sqrt \Eg} \int_{\gl\in\Delta} \int_{\theta=0}^\pi \mu_\gl^{(\Eg)}(\theta) \mu_\gl^{(\Eg)}(\pi - \theta) \sin^2 \theta d\theta d\gl\;,$$
which goes to $2h$ by Corollary \ref{cor:DOSformula}. The error made upon this replacement is of order $\Eg /L$ by Theorem \ref{Th:CVDensities} and Lemma \ref{Lemma:InvMeas} and therefore vanishes in the limit $L\to\infty$.
Consequently we get the first estimate.\\
The second estimate is derived from the same argument, one simply replaces the interval $[-L/\Eg,L/\Eg]$ by an interval of length $L/(\Eg k)$ and uses that $L/(\Eg k) \gg 1$.
\end{proof}

\section{Exponential decay}\label{Sec:ExpoDecay}

Fix $h >0$ and set $\Delta := [E - \frac{h}{n(E) L},E + \frac{h}{n(E) L}]$ for all this section. Recall that in the Bulk regime, $E\in \R$ is fixed while in the Crossover regime $E=E(L) \to \infty$ as $L\to\infty$.\\

The goal of this section is to prove the exponential decay of the eigenfunctions stated in Proposition \ref{Prop:ExpoDecayCrossover}. We first compute the Lyapounov exponent associated to the diffusions in Subsection \ref{subsec:Lyapounov}, then we introduce in Subsection \ref{subsec:adjointandmain} the adjoint diffusions as they will naturally arise in the estimate of the exponential decay. Finally in Subsection \ref{Subsec:ProofExpoDecay} we present the proof of Proposition \ref{Prop:ExpoDecayCrossover}.

\subsection{Lyapounov exponent}\label{subsec:Lyapounov}

Recall the definition of the SDEs \eqref{SDEthetaE} and \eqref{SDErhoE}. Take $\rho_\gl^{(\Eg)}(0) := 0$. The Lyapounov exponent is usually defined as the almost sure limit of
$$\frac{\ln r_\gl^{(\Eg)}(t)}{t}\;,\quad t\to\infty\;.$$
This deterministic quantity is intimately related to the rate of exponential decay of the eigenfunctions. For technical convenience, we manipulate the almost sure limit of
$$ 2\frac{\ln r_\gl^{(\Eg)}(t)}{t} = \frac{\rho_\gl^{(\Eg)}(t)}{t} \;,\quad t\to\infty\;,$$
that we denote by $\nu_\gl^{(\Eg)}$. In other words, $\nu_\gl^{(\Eg)}$ is twice the Lyapounov exponent. Our next proposition gives an explicit expression of this exponent.

\begin{proposition}[Lyapounov exponent]\label{propo:Lyapounov}
The random variables $\zeta_\gl^{(\Eg)}$ and $\rho_\gl^{(\Eg)}(\zeta_\gl^{(\Eg)})$ are integrable, and we have
\begin{align*}
 \nu_\gl^{(\Eg)} = \frac{\E[\rho_\gl^{(\Eg)}(\zeta_\gl^{(\Eg)})]}{\E[\zeta_\gl^{(\Eg)}]}\;.
\end{align*}
Moreover for any $\gl\in\R$
\begin{equation}\label{Eq:nugl} \nu_\gl^{(\Eg)} = \Eg \; \frac{ \int_0^{+\infty} \sqrt{u} \exp(-2 \lambda u - \frac{u^3}{6} ) du}{\int_0^{\infty} \frac{1}{\sqrt{u}} \exp(-2 \lambda u - \frac{u^3}{6}) du} > 0\;.\end{equation}
\end{proposition}
We have a remarkably simple expression for the quantity $\nu_\gl^{(\Eg)}$, from which we can deduce important properties. In particular, a crucial point for our proof is that this expression is \emph{positive} for all $\gl$ and $E$. The proof of this proposition follows from elementary (but not straightforward) computations on the SDEs, and is deferred to the Appendix \ref{Subsec:Lyapounov}.

\smallskip

It is easy to check that in both regimes, uniformly over all $\gl \in \Delta$ we have as $L\to\infty$
$$\nu_\gl^{(\Eg)}/\nu_E^{(\Eg)} \to 1\;.$$
Moreover in the Crossover regime we have
$$ \nu_E^{(E)} \to 1, \quad \mbox{as }E \to \infty\;.$$
As a consequence, we can safely approximate $\nu_\gl$ by $\nu_E$ in the Bulk regime, and $\nu_\gl^{(E)}$ by $1$ in the Crossover regime. A posteriori, this explains the definition of $\boldsymbol{\nu_{E}}$ in \eqref{def:nuEgras}.

Finally, let us mention that uniformly over all $\gl\in\Delta$ and all $L>1$
\begin{equation}\label{Eq:mgl}
m_\gl^{(\Eg)} \asymp \begin{cases}
1&\mbox{ in the Bulk regime}\;,\\
 \Eg^{-3/2}&\mbox{ in the Crossover regime}\;.
\end{cases}
\end{equation}

\subsection{Adjoint diffusions and the main growth estimate}\label{subsec:adjointandmain}

We introduce the processes $\bar\theta_\gl^{(\Eg)}$ and $\bar\rho_\gl^{(\Eg)}$ as the solutions of the following SDEs driven by a Brownian motion $\bar{B}^{(\Eg)}$
\begin{equation}\label{Eq:bartheta}\begin{split}
d\bar{\theta}_\gl^{(\Eg)}(t) &= \big(- \Eg^{3/2} - \sqrt \Eg (\gl-\Eg) \sin^2 \bar{\theta}_\gl^{(\Eg)} + 3 \sin^3(\bar{\theta}_\gl^{(\Eg)})\cos(\bar{\theta}_\gl^{(\Eg)}) +\sin^4 \bar{\theta}_\gl^{(\Eg)} \frac{\partial_\theta \mu_\gl^{(\Eg)}(\bar{\theta}_\gl^{(\Eg)})}{\mu_\gl^{(\Eg)}(\bar{\theta}_\gl^{(\Eg)})} \big)dt\\
&\quad- \sin^2 \bar{\theta}_\gl^{(\Eg)} d\bar{B}^{(\Eg)}(t)\;,
\end{split}\end{equation}
and
\begin{equation}\label{Eq:barrho}\begin{split}
d\bar\rho_\gl^{(\Eg)} &= \Big(\sqrt \Eg (\lambda -\Eg)\sin 2\bar{\theta}_\gl^{(\Eg)} + \frac12 \sin^2 2\bar{\theta}_\gl^{(\Eg)} + \sin^2 \bar{\theta}_\gl^{(\Eg)}\\
&\quad - 8 \sin^2 \bar\theta^{(\Eg)}_\gl \cos^2 \bar\theta^{(\Eg)}_\gl - 2 \sin^3 \bar\theta^{(\Eg)}_\gl \cos \bar\theta^{(\Eg)}_\gl \frac{\partial_\theta \mu_\gl^{(\Eg)}(\bar{\theta}_\gl^{(\Eg)})}{\mu_\gl^{(\Eg)}(\bar{\theta}_\gl^{(\Eg)})}\Big)dt\\
&\quad+ \sin 2\bar{\theta}_\gl^{(\Eg)} d\bar{B}^{(\Eg)}(t)\;.
\end{split}\end{equation}
By coherence with previous notations, we denote by $\bar{\P}_{(s,\theta_0)}$ the law of the process $\bar\theta_\gl^{(\Eg)}$ starting from $\theta_0$ at time $s$, and by $\bar{\P}_{(s,\theta_0)\to (t,\theta_1)}$ the law of the bridge from $(s,\theta_0)$ to $(t,\theta_1)$.

\smallskip

The process $\bar\theta_\gl^{(\Eg)}$ is the adjoint diffusion of $\theta_\gl^{(\Eg)}$ with respect to the invariant measure $\mu_\gl^{(\Eg)}$. In other words, its generator is the operator $\cLa$ which is the adjoint in $L^2(\mu_\gl^{(\Eg)})$ of the generator $\cL$ of $\theta_\gl^{(\Eg)}$ (see the beginning of Subsection \ref{Subsec:Hypocoercivity} for more details on the generators involved). Let us recall that this adjunction implies that the law of the process $\theta_\gl^{(\Eg)}$ starting from its invariant measure is the same as the law of the process $\bar{\theta}_\gl^{(\Eg)}$, read backward-in-time, and starting from its invariant measure. In mathematical terms: for any measurable set $A\subset \cC([0,t],\R)$, we have
$$ \int_{\theta_0=0}^\pi \mu_\gl^{(\Eg)}(\theta_0)\P_{(0,\theta_0)}(A)d\theta_0 = \int_{\theta_1=0}^\pi \mu_\gl^{(\Eg)}(\theta_1)\bar{\P}_{(0,\theta_1)}(\bar{A})d\theta_1\;,$$
where $\bar{A} := \{f: f(t-\cdot) \in A\}$ is the image of $A$ upon reversing time. Desintegrating this last expression, we further get
\begin{equation}\label{Eq:Adjoint}
\mu_\gl^{(\Eg)}(\theta_0) p_{\gl,t}^{(\Eg)}(\theta_0,\theta_1) \P_{(0,\theta_0) \to (t,\theta_1)}(A) = \mu_\gl^{(\Eg)}(\theta_1)\bar{p}_{\gl,t}^{(\Eg)}(\theta_1,\theta_0)\bar\P_{(0,\theta_1)\to (t,\theta_0)}(\bar{A})\;,
\end{equation}

\smallskip

Note that the process $\rho_\gl^{(\Eg)}$, resp.~$\bar{\rho}_\gl^{(\Eg)}$, is actually measurable w.r.t.~the process $\theta_\gl^{(\Eg)}$, resp.~$\bar{\theta}_\gl^{(\Eg)}$. (The stochastic term $\sin 2\theta_\gl^{(\Eg)} dB^{(\Eg)}$ can be expressed via It\^o's formula applied to $\ln \sin \theta^{(\Eg)}_\gl$). It can then be checked, using the adjunction relation above, that the process $(\rho_\gl^{(\Eg)}(s)-\rho_\gl^{(\Eg)}(0), s\in [0,t])$ starting from the invariant measure has the same law as the process $(\bar{\rho}_\gl^{(\Eg)}(t-s) - \bar{\rho}_\gl^{(\Eg)}(t), s\in [0,t])$ starting from the invariant measure.

\bigskip

The main estimate needed for the proof of Proposition \ref{Prop:ExpoDecayCrossover} is presented in the following lemma. From now on, we work simultaneously in the Bulk and the Crossover regimes.

\begin{lemma}[Linear growth of the diffusions $\rho_\gl^{(\Eg)}$ and $\overline{\rho}_\gl^{(\Eg)}$ ]\label{Lemma:Zgl}
Set $Z_\gl^{(\Eg)}(t) := \rho_\gl^{(\Eg)}(t) - \nu_\gl^{(\Eg)} t$ and $\bar{Z}_\gl^{(\Eg)}(t) := \bar{\rho}_\gl^{(\Eg)}(t) + \nu_\gl^{(\Eg)} t$. For any $\eps > 0$, there exists $q>0$ such that
$$ \sup_{L>1} \sup_{\gl \in \Delta} \sup_{\theta\in [0,\pi)} \E_{(0,\theta)}\Big[ \sup_{t\ge 0} e^{q |Z_\gl^{(\Eg)}(t)|} e^{-qt\eps}\Big] < \infty\;,$$
and
$$ \sup_{L>1} \sup_{\gl \in \Delta} \sup_{\theta\in [0,\pi)} \bar{\E}_{(0,\theta)}\Big[ \sup_{t\ge 0} e^{q |\bar{Z}_\gl^{(\Eg)}(t)|} e^{-qt\eps}\Big] < \infty\;,$$
\end{lemma}
\begin{proof}
We drop the superscript $(\Eg)$ to alleviate notations. The proof is virtually the same for the two processes so we concentrate on $Z_\gl$. Let $n := \lfloor \Eg^{3/2} \rfloor$. Let $0 =:T_0 < T_1 <\ldots$ be the stopping times defined by
$$ T_k := \inf\{t>T_{k-1}: \theta_\gl(t) = \theta_\gl(T_{k-1})+n\pi\}\;.$$
The r.v.~$T_k$ is equal in law to the sum of $kn$ i.i.d.~r.v.~distributed as $\zeta_\gl$.
We then define the i.i.d.~r.v.
$$ G_k : = Z_\gl(T_k) - Z_\gl(T_{k-1})\;,\quad k\ge 1\;.$$

Decomposing $Z_\gl(T_k) - Z_{\gl}(T_{k-1})$ into the sum of the increments of $Z_\gl$ in between the successive hitting times of $\pi \Z$ by the phase $\theta_\gl$, we see that $\E[G_k] = 0$ by Proposition \ref{propo:Lyapounov}.
Introduce also
$$Y_k := \sup_{t\in [T_{k-1},T_{k}]} |Z_\gl(t)-Z_\gl(T_{k-1})|\;,\quad k\ge 1\;.$$
We have for any $q\ge 0$
\begin{align*}
\sup_{t\ge 0} e^{q {Z}_\gl(t)} e^{-qt\eps} &\le \sum_{k\ge 0} e^{q(G_1+\ldots+G_k)} e^{q Y_{k+1}} e^{-q T_k \eps}\\
\sup_{t\ge 0} e^{-q {Z}_\gl(t)} e^{-qt\eps} &\le \sum_{k\ge 0} e^{-q(G_1+\ldots+G_k)} e^{q Y_{k+1}} e^{-q T_k \eps}\;.
\end{align*}
By symmetry of the arguments, we only present the details on the bound of the first term. Using the Cauchy-Schwarz inequality twice and the fact that the $Y_k$'s, the $G_k$'s and the $(T_k-T_{k-1})$'s are i.i.d., we have for every $k\ge 0$
\begin{equation}\label{Eq:GYtau} \E[e^{q(G_1+\ldots+G_k)} e^{q Y_{k+1}} e^{-q T_k \eps}] \le \E[e^{4qG_1}]^{k/4} \E[e^{-4q T_1 \eps}]^{k/4} \E[e^{2q Y_1}]^{1/2}\;.\end{equation}
Assume that there exist $C_0 > 0$ and $q_0 > 0$ such that for all $L>1$, all $\lambda \in \Delta$ and all $q \in [-q_0,q_0]$ we have
$$ \E[e^{q Y_1}] < C_0\;,\quad \E[e^{q T_1}] < C_0\;.$$
Note that $|G_1| \le Y_1$ almost surely. Note also that $\E[T_1]=n m_\gl$ and recall $\E[G_1] = 0$. By Lemma \ref{Lemma:ExpoTaylor} we deduce that there exist $C_1,q_1 > 0$ such that the r.h.s.~of \eqref{Eq:GYtau} is bounded by
$$ (1+ 16C_1 q^2)^{k/4} (1-4q\eps n m_\gl + 16C_1 q^2 \eps^2)^{k/4}(1+ C_0 + C_1 q^2)^{1/2}\;,$$
for all $q \in (0,q_1)$. Recall from \eqref{Eq:mgl} that $m_\gl \asymp \Eg^{-3/2}$ so that $nm_\gl$ is bounded from below by a positive constant uniformly over all parameters. Recall also that $\eps$ is fixed. We deduce that by choosing $q$ small enough, the last term is bounded by $C \eta^k$ for some constants $C>0$ and $\eta \in (0,1)$. Summing this term over $k\ge 0$, we get the desired upper bound.

It remains to prove the exponential bounds on the non-negative r.v.~$Y_1$ and $T_1$. Regarding $T_1$, we use Lemma \ref{Lemma:Moments} to deduce that for $q > 0$ small enough we have
$$\E[e^{q T_1}] = \E[e^{q \zeta_1}]^n \le (\frac1{1-qm_\gl})^n=e^{-n\log(1-qm_\gl)}\;.$$
Since $nm_\gl$ is of order $1$, this suffices to conclude. We turn to $Y_1$. For any $x,T>0$ we have
$$ \P(Y_1 > x) \le \P(T_1 > T) + \P(\sup_{t\in [0,T]} |Z_\gl(t)| > x)\;.$$
The first term on the r.h.s~decays exponentially in $T$ by the exponential bound already obtained. The second term decays exponentially in $x^2/T$ by Lemma \ref{Lemma:Mgale} as soon as $x$ is large enough. Adjusting $T$ and $x$, we easily obtain an exponential decay in $x$ uniformly over all parameters, thus concluding the proof.
\end{proof}

\subsection{Proof of Proposition \ref{Prop:ExpoDecayCrossover}}\label{Subsec:ProofExpoDecay}

By analogy with $G_u$, let us introduce the observable:
$$ H_u(\gl,\varphi,\psi) := \frac{1}{(\int \varphi^2(s) ds)^{1/2}}\sup_{t\in [-\frac{L}{2\Eg},\frac{L}{2\Eg}]} \Big(\varphi^2(t) + \frac{\psi^2(t)}{\Eg^3}\Big)^{1/2} e^{\frac12(\nu_{\gl}^{(\Eg)} - \eps) |t-u|}\;.$$

Set $H = \inf_{u\in [-\frac{L}{2\Eg},\frac{L}{2\Eg}]} H_u$. By the GMP formula given in Proposition \ref{Prop:GMP}, it suffices to bound
\begin{align*}
&\sqrt \Eg \int_{u=-\frac{L}{2\Eg}}^{\frac{L}{2\Eg}} \int_{\gl \in\Delta} \int_{\theta=0}^\pi p_{\gl,u+\frac{L}{2\Eg}}^{(\Eg)}(\theta) p_{\gl,\frac{L}{2\Eg}-u}^{(\Eg)}(\pi-\theta)\sin^2 \theta\; \E^{(u)}_{\theta,\pi-\theta}\Big[\Big(H\big(\gl,\yu_\gl^{(\Eg)}, (\yu_\gl^{(\Eg)})'\big)\Big)^q\Big] \, d\theta d\gl du\\
\le&\sqrt \Eg \int_{u=-\frac{L}{2\Eg}}^{\frac{L}{2\Eg}} \int_{\gl \in\Delta} \int_{\theta=0}^\pi p_{\gl,u+\frac{L}{2\Eg}}^{(\Eg)}(\theta) p_{\gl,\frac{L}{2\Eg}-u}^{(\Eg)}(\pi-\theta)\sin^2 \theta \;\E^{(u)}_{\theta,\pi-\theta}\Big[\Big(H_u\big(\gl,\yu_\gl^{(\Eg)}, (\yu_\gl^{(\Eg)})'\big)\Big)^q\Big] \, d\theta d\gl du\;.
\end{align*}
Note that at the second line, we have bounded $H$ by $H_u$ where $u$ the concatenation time.

Under $\P^{(u)}_{\theta,\pi-\theta}$ we have
\begin{align}
 H_u(\gl,\yu_\gl^{(\Eg)}, (\yu_\gl^{(\Eg)})') = \|\yu_\gl^{(\Eg)}\|_2^{-1} \sup_{t\in [-\frac{L}{2\Eg},\frac{L}{2\Eg}]} \ru_\gl^{(\Eg)}(t) e^{\frac12(\nu_{\gl}^{(\Eg)} - \eps) |t-u|}\;. \label{def:Hu}
\end{align}
The Lebesgue measure of $\Delta$ is $2h/(n(E)L) \asymp \sqrt\Eg/L$. Consequently, it suffices to show that 
\begin{align}\label{Eq:GMPH_u}
 \int_{\theta=0}^\pi p_{\gl,u+\frac{L}{2\Eg}}^{(\Eg)}(\theta) p_{\gl,\frac{L}{2\Eg}-u}^{(\Eg)}(\pi-\theta)\sin^2 \theta \;\E^{(u)}_{\theta,\pi-\theta}\Big[\Big(H_u\big(\gl,\yu_\gl^{(\Eg)}, (\yu_\gl^{(\Eg)})'\big)\Big)^q\Big] \, d\theta 
\end{align}
is bounded by some constant uniformly over all $u\in [-L/(2\Eg),L/2\Eg]$ and all $\gl\in \Delta$.

\smallskip

By the Cauchy-Schwarz inequality and the bound $\sqrt{x} \le 1+x$ for any $x\ge 0$, we have
\begin{align*}
&\E^{(u)}_{\theta,\pi-\theta}\Big[\Big(H_u\big(\gl,\yu_\gl^{(\Eg)}, (\yu_\gl^{(\Eg)})'\big)\Big)^q\Big]\\
&\le \Big(1+\E^{(u)}_{\theta,\pi-\theta}\Big[\sup_{t\in [-\frac{L}{2\Eg},\frac{L}{2\Eg}]} \ru_\gl^{(\Eg)}(t)^{2q} \exp\big(q(\nu_{\gl}^{(\Eg)} - \eps) |t-u|\big)\Big]\Big) \,\E^{(u)}_{\theta,\pi-\theta}\Big[\|\yu_\gl^{(\Eg)}\|_2^{-2q}\Big]^{1/2}\;.
\end{align*}
By Theorem \ref{Th:CVDensities}, the term
$$\int_{\theta=0}^\pi p_{\gl,u+\frac{L}{2\Eg}}^{(\Eg)}(\theta) p_{\gl,\frac{L}{2\Eg}-u}^{(\Eg)}(\pi-\theta)\sin^2 \theta d\theta\;,$$
is bounded by some constant uniformly over all parameters. To conclude, it therefore suffices to prove that (for a different choice of $q$)
\begin{align}\label{Eq:IntExpoDecay}
\int_{0}^\pi p_{\gl,u+\frac{L}{2\Eg}}^{(\Eg)}(\theta) p_{\gl,\frac{L}{2\Eg}-u}^{(\Eg)}(\pi-\theta)\sin^2 \theta \;\E^{(u)}_{\theta,\pi-\theta}\Big[\sup_{t\in [-\frac{L}{2\Eg},\frac{L}{2\Eg}]} \ru_\gl^{(\Eg)}(t)^q \exp\big(\frac{q}{2}(\nu_{\gl}^{(\Eg)} - \eps) |t-u|\big)\Big] \, d\theta\,,
\end{align}
and
\begin{align}\label{Eq:MassExpoDecay}
\sup_{\theta} \E^{(u)}_{\theta,\pi-\theta}\Big[\|\yu_\gl^{(\Eg)}\|_2^{-q}\Big]\;,
\end{align}
are bounded by some constant uniformly over all $u$ and $\gl$.\\

We aim at applying the estimates of Lemma \ref{Lemma:Zgl}. The difficulty is twofold. First, the estimates in the lemma concern unconditioned processes while in \eqref{Eq:IntExpoDecay} we have (a concatenation of) conditioned processes. Second, in \eqref{Eq:IntExpoDecay} the process $\ru^{(\Eg)}_\gl$ is set to $1$ at $u$ so we have to be careful with this normalisation in the bounds.\\

To deal with conditioned processes, we use our estimates on the convergence to equilibrium of the diffusions to prove the following.
\begin{lemma}[Absolute continuity of the bridges]\label{Lemma:CondToUncond}
There exists $t_0 \ge 1$ and a constant $C>0$ such that for all $L$ large enough, for all $\gl \in \Delta$, $t>0$, $\theta\in [0,\pi)$ and for all events $A$ that depend on $\theta^{(\Eg)}_\gl(s), s\in [0,t]$ we have
$$ \sup_{\theta'}\P_{(0,\theta) \to (t+t_0,\theta')}(A) \le C\,\P_{(0,\theta)}(A)\;,$$
and
$$ \sup_{\theta'}\bar{\P}_{(0,\theta) \to (t+t_0,\theta')}(A) \le C\,\bar{\P}_{(0,\theta)}(A)\;.$$
In addition, as $t_0\to\infty$ we have
$$ \P_{(0,\theta) \to (t+t_0,\theta')}(A) = (1+o(1))\P_{(0,\theta)}(A)\;,\quad \bar{\P}_{(0,\theta) \to (t+t_0,\theta')}(A) = (1+o(1)) \bar{\P}_{(0,\theta)}(A)\;,$$
uniformly over all $\theta,\theta' \in [0,\pi]$ and all events $A$ as above.
\end{lemma}
\begin{proof}
The proof is identical for $\theta_\gl^{(\Eg)}$ and $\bar{\theta}_\gl^{(\Eg)}$ so we restrict to the former. For any given $t>0$ and for any event $A$ that only depends on the trajectory of $\theta_\gl^{(\Eg)}(s)$ for $s\in [0,t]$, by the Markov property we have
\begin{align*}
\P_{(0,\theta) \to (t+t_0,\theta')}(A) &= \lim_{\delta \downarrow 0} \frac{\E_{(0,\theta)}\big[\un_A \P_{(0,\theta_\gl^{(\Eg)}(t))}(\theta_\gl^{(\Eg)}(t_0) \in [\theta'-\delta,\theta'+\delta])\big]}{\P_{(0,\theta)}(\theta_\gl^{(\Eg)}(t+t_0) \in [\theta'-\delta,\theta'+\delta])}\\
&= \lim_{\delta \downarrow 0}\frac{ \E_{(0,\theta)}\big[\un_A \Big( 1 + \frac{\P_{(0,\theta_\gl^{(\Eg)}(t))}(\theta_\gl^{(\Eg)}(t_0) \in [\theta'-\delta,\theta'+\delta])-\mu_\gl^{(\Eg)}(\theta')}{\mu_\gl^{(\Eg)}(\theta')}\Big)\big]}{\Big( 1 + \frac{\P_{(0,\theta)}(\theta_\gl^{(\Eg)}(t+t_0) \in [\theta'-\delta,\theta'+\delta])-\mu_\gl^{(\Eg)}(\theta')}{\mu_\gl^{(\Eg)}(\theta')}\Big)}\;.
\end{align*}
By Theorem \ref{Th:CVDensities} and Lemma \ref{Lemma:InvMeas} as $t\to\infty$ we have
\begin{equation}\label{Eq:BdDens} \sup_{\gl \in \Delta}\sup_{\theta_0,\theta\in [0,\pi]} \frac{|p_{\gl,t}^{(\Eg)}(\theta_0,\theta) - \mu_{\gl}^{(\Eg)}(\theta)|}{\mu_{\gl}^{(\Eg)}(\theta)} \to 0\;.\end{equation}
This suffices to conclude.
\end{proof}

Let us now bound \eqref{Eq:MassExpoDecay}. By symmetry, we restrict to $u\in [-\frac{L}{2\Eg},0]$. As $u+1 < L/(2\Eg)$, we write
\begin{align*}
\E^{(u)}_{\theta,\pi-\theta}[\|\yu_\gl^{(\Eg)}\|_2^{-q}] &\le \E^{(u)}_{\theta,\pi-\theta}\Big[\Big(\int_u^{u+1} (\yu_\gl^{(\Eg)}(t))^2 dt\Big)^{-q/2}\Big]\;.
\end{align*}
This expression only involves the backward diffusion. Using \eqref{Eq:Adjoint} this last term equals
$$ \frac{\mu_\gl^{(\Eg)}(\pi-\theta)\bar{p}_{\gl,\frac{L}{2\Eg} -u}^{(\Eg)}(\pi-\theta,0)}{\mu_\gl^{(\Eg)}(0){p}_{\gl,\frac{L}{2\Eg} -u}^{(\Eg)}(0,\pi-\theta)}\; \bar{\E}_{(0,\pi-\theta)\to (\frac{L}{2\Eg} -u,0)}\Big[\Big(\int_0^{1} (\bar{y}_\gl^{(\Eg)}(t))^2 dt\Big)^{-q/2}\Big]\;,$$
with $\bar{r}_\gl^{(\Eg)}(0)=1$ (recall that $\ru_\gl^{(\Eg)}(u)=1$). By Theorem \ref{Th:CVDensities} and Lemma \ref{Lemma:InvMeas}, the prefactor is bounded by a constant uniformly over all parameters. By Lemma \ref{Lemma:CondToUncond} the expectation can be bounded by a constant times
\begin{align*}
\bar{\E}_{\pi-\theta}\Big[\Big(\inf_{t\in [0,1]} \bar{r}_\gl^{(\Eg)}(t)\Big)^{-2q}\Big]^{1/2}\; \bar{\E}_{\pi-\theta}\Big[\Big(\int_0^{1} \sin^2 \bar{\theta}_\gl^{(\Eg)}(t) dt\Big)^{-q}\Big]^{1/2}\;.
\end{align*}
To bound the first term we apply Lemma \ref{Lemma:Zgl}. To bound the second term we use Lemma \ref{Lemma:ExpoBoundTheta}.\\

We turn to \eqref{Eq:IntExpoDecay}, which is more involved. Let us denote by $\tilde u := u+L/(2\Eg)$ the distance of $u$ to the left boundary of the interval and $\tilde v := -u + L/(2\Eg)$ the distance to the right boundary. By symmetry, it suffices to bound
\begin{align*}
&\int_{0}^\pi p_{\gl,\tilde u}^{(\Eg)}(\theta) p_{\gl,\tilde v}^{(\Eg)}(\pi-\theta)\sin^2 \theta \;\E^{(u)}_{\theta,\pi-\theta}\Big[\sup_{t\in [-\frac{L}{2\Eg},u]} \ru_\gl^{(\Eg)}(t)^q \exp\big(\frac{q}{2}(\nu_{\gl}^{(\Eg)} - \eps) |t-u|\big)\Big] \, d\theta\;.
\end{align*}
This expression only involves the forward diffusion. By shifting time appropriately, it rewrites
\begin{align*}
\int_{0}^\pi p_{\gl,\tilde u}^{(\Eg)}(\theta) p_{\gl,\tilde v}^{(\Eg)}(\pi-\theta)\sin^2 \theta \;\E_{(0,0)\to (\tilde{u},\theta)}\Big[\sup_{t\in [0,\tilde{u}]} \Big(\frac{r_\gl^{(\Eg)}(t)}{r_\gl^{(\Eg)}(\tilde{u})}\Big)^q \exp\big(\frac{q}{2}(\nu_{\gl}^{(\Eg)} - \eps) |t-\tilde{u}|\big)\Big] \, d\theta\;,
\end{align*}
with $r_\gl^{(\Eg)}(0)=1$. Let us take $t_0$ from Lemma \ref{Lemma:CondToUncond}. We distinguish two cases according to the relative values of $\tilde{u}$ and $3t_0$.

Assume $\tilde u \leq 3t_0$. Theorem \ref{Th:CVDensities} and Lemma \ref{Lemma:InvMeas} allow to bound $p_{\gl,\tilde v}^{(\Eg)}(\pi-\theta)$ by a constant. Bounding $\sin^2\theta$ by one, and integrating in $\theta$, we see that it suffices to bound:
\begin{align*}
\E_{(0,0)}\Big[\sup_{t\in [0,\tilde u]} \Big(\frac{r_\gl^{(\Eg)}(t)}{r_\gl^{(\Eg)}(\tilde u)}\Big)^q \exp\big(\frac{q}{2}(\nu_{\gl}^{(\Eg)} - \eps) |t-\tilde u|\big)\Big]\,,
\end{align*}
which is itself bounded thanks to Lemma \ref{Lemma:Zgl}.\\

Assume $\tilde u > 3t_0$. It suffices to bound (uniformly over all the parameters):
\begin{align}\label{Eq:BdMomentTildeu}
 \E_{(0,0) \to (\tilde u,\theta)}\Big[\sup_{t\in [0,\tilde u]} \Big(\frac{r_\gl^{(\Eg)}(t)}{r_\gl^{(\Eg)}(\tilde u)}e^{\frac12(\nu_\gl^{(\Eg)}-\eps)|t-\tilde u|}\Big)^q\Big]\,.
\end{align}
Indeed, either $\tilde v > 3t_0$ and then we bound $p_{\gl,\tilde u}^{(\Eg)}(\theta) p_{\gl,\tilde v}^{(\Eg)}(\pi-\theta)$ using Theorem \ref{Th:CVDensities} and Lemma \ref{Lemma:InvMeas}; or $\tilde v \leq 3 t_0$ and then we apply the same arguments to bound $p_{\gl,\tilde u}(\theta)$ while we integrate over $\theta$ the term $p_{\gl,\tilde v}(\pi - \theta)$. To bound \eqref{Eq:BdMomentTildeu}, we split the supremum into two parts:
\begin{itemize}
\item
For the supremum over $t\in [2t_0, \tilde u]$, we use the adjoint diffusion and the identity \eqref{Eq:Adjoint} so that it suffices to bound
$$ \bar\E_{(0,\theta) \to (\tilde u,0)}\Big[\sup_{t\in [0,\tilde u-2t_0]} \Big(\bar{r}_\gl^{(\Eg)}(t) e^{\frac12(\nu_\gl^{(\Eg)}-\eps)t}\Big)^q\Big]\;,$$
with $\bar{r}_\gl^{(\Eg)}(0)=1$. Using Lemma \ref{Lemma:CondToUncond}, the latter is bounded from above by a constant times
$$ \bar\E_{(0,\theta)}\Big[\sup_{t\in [0,\tilde u-2t_0]} \Big(\bar{r}_\gl^{(\Eg)}(t) e^{\frac12(\nu_\gl^{(\Eg)}-\eps)t}\Big)^q\Big]\;,$$
and we can then apply Lemma \ref{Lemma:Zgl}.\\
\item For the supremum over $t\in [0,2t_0]$, we write
\begin{align*}
\frac{r_\gl^{(\Eg)}(t)}{r_\gl^{(\Eg)}(\tilde u)}e^{\frac12(\nu_\gl^{(\Eg)}-\eps)(\tilde u-t)} = e^{\frac12(Z_\gl^{(\Eg)}(t) - Z_\gl^{(\Eg)}(2t_0) -\eps(2t_0-t))} e^{\frac12(Z_\gl^{(\Eg)}(2t_0) - Z_\gl^{(\Eg)}(\tilde u) -\eps(\tilde u-2t_0))}\;.
\end{align*}
Using the Cauchy-Schwarz inequality we thus find
\begin{align*}
&\E_{(0,0) \to (\tilde u,\theta)}\bigg[\sup_{t\in [0,2t_0]} \Big(\frac{r_\gl^{(\Eg)}(t)}{r_\gl^{(\Eg)}(\tilde u)}e^{\frac12(\nu_\gl^{(\Eg)}-\eps)|t-\tilde u|}\Big)^q\bigg]\\
\le\;&
\E_{(0,0) \to (\tilde u,\theta)}\Big[\sup_{t\in [0,2t_0]} e^{q(Z_\gl^{(\Eg)}(t) - Z_\gl^{(\Eg)}(2t_0) -\eps(2t_0-t))}\Big]^{1/2}\\
&\qquad \times\E_{(0,0) \to (\tilde u,\theta)}\Big[e^{q(Z_\gl^{(\Eg)}(2t_0) - Z_\gl^{(\Eg)}(\tilde u) -\eps(\tilde u-2t_0))}\Big]^{1/2}\;.
\end{align*}
The second expectation on the r.h.s.~can be bounded using the adjoint diffusion as above. By Lemma \ref{Lemma:CondToUncond}, the first expectation is bounded by a constant times
$$\E_{(0,0)}\Big[\sup_{t\in [0,2t_0]}e^{q(Z_\gl^{(\Eg)}(t) - Z_\gl^{(\Eg)}(2t_0) -\eps(2t_0-t))}\Big]\;,$$
which is itself bounded by a constant by Lemma \ref{Lemma:Zgl}.
\end{itemize}
This concludes the proof of Proposition \ref{Prop:ExpoDecayCrossover}.

\section{Fine estimates on the diffusion}\label{Sec:TwoPoints}

Set $\Delta := [E- h/(Ln(E)),E+h/(Ln(E))]$. In this section, we concentrate on the r.v.~$N_L(\Delta)$ and $N_L^{(j)}(\Delta)$ and establish Propositions \ref{Prop:Moment} and \ref{Prop:OneEigen}. To simplify the notation, we work on the interval $[0,L/\Eg]$ instead of $[-L/(2 \Eg),L/(2\Eg)]$ since our arguments will only rely on diffusions going forward in time. As a consequence, in this whole section, the diffusion $\theta_\gl^{(\Eg)}$ starts from $0$ at time $0$ and live on the interval of time $[0,L/\Eg]$.

\emph{Furthermore to alleviate the notations, we drop the superscript $(\Eg)$ from the diffusions although the proof is carried out simultaneously in both regimes.}

\subsection{A thorough study of a joint diffusion}\label{Subsec:OneEigen}

The whole discussion revolves around a thorough study of the joint diffusion $(\theta_\gl,\theta_\mu)$ where $[\gl,\mu] := \Delta$. Actually, it is more convenient to deal with $(\theta_\gl,\alpha)$ where
$$ \alpha(t) := \theta_\mu(t)-\theta_\gl(t)\;,\quad t\ge 0\;.$$
From the expression of the diffusions $\theta_\mu$ and $\theta_\gl$ in \eqref{SDEthetaE}, it is easy to see that $(\{\alpha\}_\pi, \{\theta_\gl\}_\pi)$ is a Markov process.
Moreover, the value of $\theta_\gl(0)$ being fixed, the process $\alpha$ remains non-negative and is stochastically non-decreasing w.r.t.~its initial condition, we will use these properties below without further mention.

Since we shifted the interval $[-L/(2\Eg),L/(2\Eg)]$ to $[0,L/\Eg]$, the number of eigenvalues of $\mathcal{H}_L$ in the interval $[\gl,\mu]$ is given by $N_L(\Delta) = \lfloor \theta_\mu(L/\Eg) \rfloor_\pi - \lfloor \theta_\gl(L/\Eg)\rfloor_\pi$. As this quantity is not very tractable, we instead look at $\lfloor \alpha(L/\Eg)\rfloor_\pi$ but we need to argue that it is a faithful approximation of the former.\\
Since $x = \lfloor x \rfloor_\pi \pi +\{x\}_\pi$, we have the identity
\begin{equation}\label{Eq:alphaN}
\alpha(L/\Eg) = N_L(\Delta)\pi + \{\theta_\mu(L/\Eg)\}_\pi - \{\theta_\gl(L/\Eg)\}_\pi\;,
\end{equation}
from which one deduces the following simple inequalities
\begin{equation}\label{Eq:ObviousBound}
N_L(\Delta) - 1 \le \lfloor {\alpha(L/\Eg)} \rfloor_\pi \le N_L(\Delta)\;.
\end{equation}

Recall that $\theta_\gl(0) = \theta_\mu(0) = 0$. Define the durations of the successive excursions of $\{\theta_\gl\}_\pi$ defined as
\begin{equation}\label{Eq:Defzetai}\zeta^{(0)} := 0,\quad \zeta^{(i)} := \inf\{t\ge 0: \theta_\gl(\zeta^{(1)} + \ldots + \zeta^{(i-1)}+t) = i\pi\} \mbox{ for }i \geq 1\;.\end{equation}
We introduce the analogous stopping times for $\alpha$, namely
$$\tau^{(0)} := 0,\quad \tau^{(i)} := \inf\{t\ge 0: \alpha(\tau^{(1)} + \ldots + \tau^{(i-1)}+t) = i\pi\}  \mbox{ for }i \geq 1\;.$$
At this point, let us observe that the $\zeta^{(i)}$ are typically of order $\Eg^{-3/2}$ since their expectation equals $m_\gl$. On the other hand, the $\tau^{(i)}$ should typically be of order $L/\Eg$ since the number of eigenvalues in $\Delta$ is of order $1$. This illustrates that the two processes $\theta_\gl$ and $\alpha$ do not evolve at all on the same time-scale. It is also important to note that the $\tau^{(i)}$'s are not i.i.d.~since they are coupled by the values of the diffusion $\theta_\gl$ at the times $\sum_{j=1}^{i-1} \tau^{(j)}$'s.\\
To circumvent this difficulty, we show that the $\tau^{(i)}$'s are stochastically larger than a sequence of i.i.d.~r.v.~whose law is \emph{almost} the law of $\tau^{(1)}$: this is the content of the next lemma, on which the rest of our arguments rely.
\begin{lemma}[Stochastic lower bound of the hitting times $\tau^{(i)}$]\label{Lemma:KeyCoupling}
Let $(\tilde{\tau}^{(i)},\tilde{\zeta}^{(i)})_{i\ge 1}$ be a sequence of i.i.d.~r.v.~such that each $(\tilde{\tau}^{(i)},\tilde{\zeta}^{(i)})$ has the law of $(\tau^{(1)},\zeta^{(1)})$. Then the sequence $(\tau^{(i)})_{i\ge 1}$ is stochastically larger than the sequence $((\tilde{\tau}^{(i)} - \tilde{\zeta}^{(i)})_+)_{i\ge 1}$, that is, for any $n\ge 1$ and any bounded and non-decreasing function $f:\R^n\to\R$, we have
$$ \E[f(\tau^{(1)},\ldots,\tau^{(n)})] \ge \E[f((\tilde{\tau}^{(1)} - \tilde{\zeta}^{(1)})_+,\ldots,(\tilde{\tau}^{(n)} - \tilde{\zeta}^{(n)})_+)]\;.$$
\end{lemma}
\begin{proof}
Define $X_0 = 0$ and for every $i\ge 1$
$$ X_{i}:= \{\theta_\gl(\tau^{(1)}+\ldots+\tau^{(i)})\}_\pi\;.$$
The Markov property implies that the conditional law of $\tau^{(i)}$ given $\cF_{\tau^{(1)} + \ldots +\tau^{(i-1)}}$ is $\nu_{X_{i-1}}$, where $\nu_x$ is the law of $\tau^{(1)}$ given $(\theta_\gl(0),\alpha(0)) = (x,0)$. For any bounded and non-decreasing function $f:\R^n\to\R$ we thus have
$$ \E\big[f(\tau^{(1)},\ldots,\tau^{(n)}) \,|\, \cF_{\tau^{(1)} + \ldots +\tau^{(n-1)}}\big] = F(\tau^{(1)},\ldots,\tau^{(n-1)}, X_{n-1})\;,$$
where
$$ F(v_1,\ldots,v_{n-1},x_{n-1}) = \int_y f(v_1,\ldots,v_{n-1},y) \nu_{x_{n-1}}(dy)\;.$$
Let $\tilde{\nu}$ be the law of $(\tilde{\tau}^{(1)} - \tilde{\zeta}^{(1)})_+$. We claim that for every $x\in [0,\pi)$, $\nu_{x}$ is larger\footnote{We say $\nu$ is larger than $\mu$ if for any non-decreasing and bounded function $f$ we have $\int f d\nu \ge \int f d\mu$.} than $\nu_x$. We deduce that
$$ F(v_1,\ldots,v_{n-1},x_{n-1}) \ge G(v_1,\ldots,v_{n-1}) := \int_y f(v_1,\ldots,v_{n-1},y) \tilde{\nu}(dy)\;.$$
Note that $G$ is bounded and non-decreasing so that a simple recursion yields
$$ \E\big[f(\tau^{(1)},\ldots,\tau^{(n)})] \ge \int f(y_1,\ldots,y_n) \tilde{\nu}(dy_1)\ldots \tilde{\nu}(dy_n)\;.$$
To prove the claim, we fix $x\in [0,\pi)$ and we consider the process $(\theta_\gl,\alpha)$ starting from $(x,0)$ and driven by the Brownian motion $B$: the law of the associated r.v.~$\tau^{(1)}$ is therefore $\nu_x$. We consider an independent Brownian motion $\tilde{B}$ and we build a diffusion $(\tilde{\theta}_\gl,\tilde{\alpha})$ starting from $(0,0)$ as follows: up to the stopping time $S_x:= \inf\{t\ge 0: \tilde{\theta}_\gl(t) = x\}$, the diffusion $(\tilde{\theta}_\gl,\tilde{\alpha})$ is driven by $\tilde{B}$; at any time $t\ge S_x$, the diffusion is driven by $B(t-S_x)$. Consequently $\tilde{\theta}_\gl(S_x+t) = \theta_\gl(t)$ for all $t\ge 0$. By monotonicity property of $\alpha$, we have almost surely
$$(\tilde{\tau}^{(1)} - S_\pi)_+ \le (\tilde{\tau}^{(1)} - S_x)_+ \le {\tau}^{(1)}\;.$$
Since $S_\pi=\tilde{\zeta}^{(1)}$, we deduce that $(\tilde{\tau}^{(1)} - \tilde{\zeta}^{(1)})_+$ is stochastically lower than $\tau^{(1)}$ (see Figure \ref{Fig:coupling} for an illustration of this coupling). This concludes the proof.
\end{proof}
\begin{figure}
\includegraphics[width=7cm]{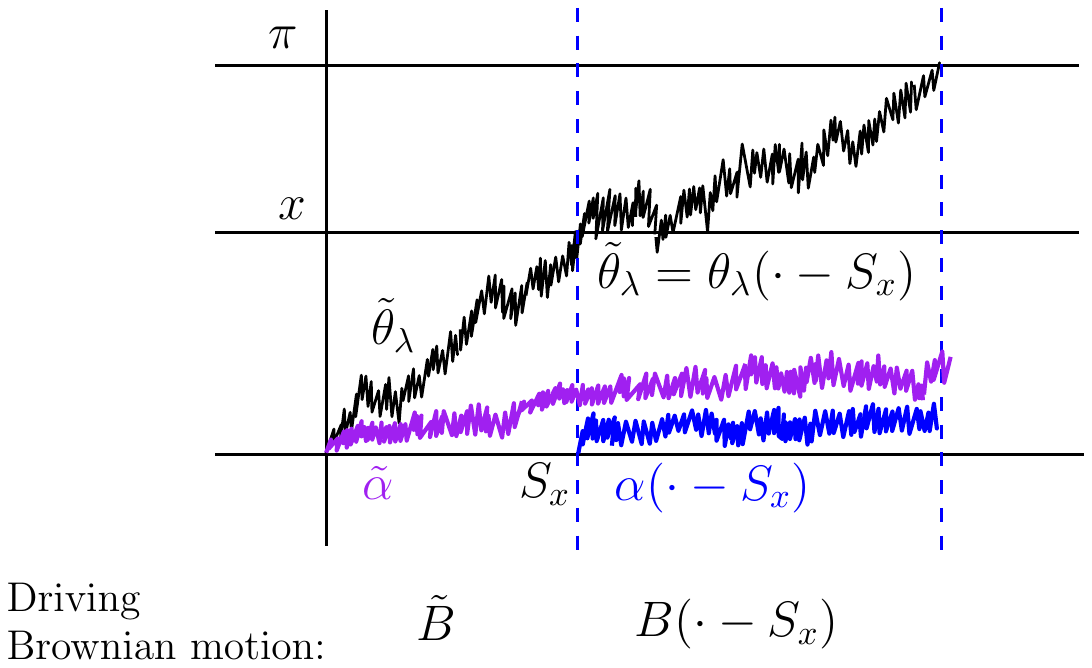}
\caption{Coupling of $(\theta_\gl,\alpha)$ and $(\tilde\theta_\gl,\tilde\alpha)$.}\label{Fig:coupling}
\end{figure}

We can now provide the proof of Proposition \ref{Prop:Moment}.
\begin{proof}[Proof of Proposition \ref{Prop:Moment}]
We start with the bound of $N_L(\Delta)^2$. Without loss of generality, we can assume that $h$ is small. Indeed, since $N_L(\Delta\cup \Delta') = N_L(\Delta) + N_L(\Delta')$ for any two disjoint intervals $\Delta$ and $\Delta'$, the bound of the second moment for a small interval propagates to larger intervals.\\
By \eqref{Eq:ObviousBound}, it suffices to prove that
$$ \sup_{L > 1} \E\big[ \lfloor \alpha(L/\Eg) \rfloor_\pi^2\big] < \infty\;.$$
By Lemma \ref{Lemma:KeyCoupling}, we have
\begin{align*}
\P(\tau_1+\ldots+\tau_n \le L/\Eg) &\le \P\big(\sum_{i=1}^n (\tilde{\tau}^{(i)} - \tilde{\zeta}^{(i)})_+ \le L/\Eg\big) \le \P\big((\tilde{\tau}^{(1)} - \tilde{\zeta}^{(1)})_+ \le L/\Eg\big)^n \\
&\le \big(\P(\tilde{\tau}^{(1)} \le 2L/\Eg) + \P(\tilde{\zeta}^{(1)} > L/\Eg)\big)^n\;.
\end{align*}
Observe that
$$ \P(\tilde{\tau}^{(1)} \le 2L/\Eg) = \P(\lfloor {\alpha(2L/\Eg)} \rfloor_\pi \ge 1) \le \P(N_{2L}(\Delta) \ge 1) \le \E[N_{2L}(\Delta)]\;,$$
which, by Proposition \ref{Prop:MicroEstimates} and provided that $h$ is small enough, is smaller than $1/4$ for all $L$ large enough. Furthermore, recall that $m_\gl \asymp \Eg^{-3/2}$, and by Markov's inequality
$$ \P( \tilde{\zeta}^{(1)} > L/\Eg) \le m_\gl  \frac{\Eg}{L}\;,$$
which is also smaller than $1/4$ for all $L$ large enough. As a consequence
$$ \E\big[ \lfloor {\alpha(L/\Eg)}\rfloor_\pi^2\big] = \sum_{n\ge 1} (2n-1)\P(\tau_1+\ldots+\tau_n \le L/\Eg)\;,$$
is bounded uniformly over all $L$ large enough.

\smallskip

Regarding the bound of $N_L^{(j)}(\Delta)^2$, let us set $L'=L/(k\Eg)$. One needs to show that
$$ \sup_{L > 1} k\E\big[ \lfloor \alpha(L') \rfloor_\pi^2\big] < \infty\;.$$
Repeating the same steps, we see that
$$ \P(\tilde{\tau}^{(1)} \le 2L') \le \E[N_{2L}^{(j)}(\Delta)]\;,$$
which, by Proposition \ref{Prop:MicroEstimates} and provided that $h$ is small enough, is bounded by $1/(4k)$. Moreover
$$ \P( \tilde{\zeta}^{(1)} > L') \le m_\gl  / L'\;,$$
which is smaller than $1/(4k)$ provided\footnote{This is a restriction on the speed at which $k$ can go to $\infty$.} $k^2 \ll L\sqrt\Eg$. This suffices to conclude.
\end{proof}

Set $L' = L/(k\Eg)$. Before we present the proof of Proposition \ref{Prop:OneEigen}, we need a last estimate whose proof is involved and therefore postponed to the next subsection.
\begin{lemma}[Key bound on $\alpha$ and $\theta$]\label{Lemma:Eventalphatheta}
As $L\to\infty$ we have $k \P(\{\alpha(L')\}_\pi + \{\theta_\gl(L')\}_\pi \ge \pi) \to 0$.
\end{lemma}

\begin{proof}[Proof of Proposition \ref{Prop:OneEigen}]
Identities \eqref{Eq:alphaN} and \eqref{Eq:ObviousBound} remain true upon replacing $N_L(\Delta)$ by $N_L^{(1)}(\Delta)$ and the time $L/\Eg$, at which the processes are evaluated, by $L'$. 

Inequalities  \eqref{Eq:ObviousBound} imply that $N^{(1)}_{L}(\Delta) =  \lfloor {\alpha(L')} \rfloor_\pi$ or $N^{(1)}_L(\Delta) =  \lfloor {\alpha(L')} \rfloor_\pi + 1$. Therefore, if $N^{(1)}_L(\Delta) \ge 2$, then necessarily one of the two events $\{\lfloor {\alpha(L')} \rfloor_\pi \ge 2\}$ and $\{N^{(1)}_L(\Delta) \neq \lfloor {\alpha(L')} \rfloor_\pi\}$ must be satisfied. We will control the probabilities of the two events. For the last event, note that  \eqref{Eq:alphaN} implies:
$$ N_L^{(1)}(\Delta)  = \lfloor {\alpha(L')}\rfloor_\pi \Leftrightarrow \{\theta_\mu(L')\}_\pi - \{\theta_\gl(L')\}_\pi \ge 0\;.$$
Since we have
$$ \{\alpha(L')\}_\pi = \begin{cases} \{\theta_\mu(L')\}_\pi - \{\theta_\gl(L')\}_\pi\quad&\mbox{ if }\{\theta_\mu(L')\}_\pi - \{\theta_\gl(L')\}_\pi \ge 0\;,\\
\pi + \{\theta_\mu(L')\}_\pi - \{\theta_\gl(L')\}_\pi \quad&\mbox{ if } \{\theta_\mu(L')\}_\pi - \{\theta_\gl(L')\}_\pi < 0\;.
\end{cases}$$
we deduce that
$$ N_L^{(1)}(\Delta)  = \lfloor {\alpha(L')} \rfloor_\pi \Leftrightarrow \{\alpha(L')\}_\pi + \{\theta_\gl(L')\}_\pi < \pi\;.$$
Consequently
\begin{align*}
k\P(N_L^{(1)}(\Delta) \ge 2) &\le k\P(\lfloor {\alpha(L')} \rfloor_\pi \ge 2) + k\P(\{\alpha(L')\}_\pi + \{\theta_\gl(L')\}_\pi \ge \pi)\;.
\end{align*}
By Lemma \ref{Lemma:Eventalphatheta} the second term on the r.h.s.~goes to $0$ as $L\to\infty$. Regarding the first term, by Lemma \ref{Lemma:KeyCoupling} we have
\begin{align*}
\P\big(\lfloor {\alpha(L')} \rfloor_\pi \ge 2\big) &= \P(\tau^{(1)}+\tau^{(2)} \le L') \le \P((\tilde{\tau}^{(1)}-\tilde{\zeta}^{(1)})_+ +(\tilde{\tau}^{(2)}-\tilde{\zeta}^{(2)})_+ \le L')\\
&\le (\P((\tilde{\tau}^{(1)}-\tilde{\zeta}^{(1)})_+ \le L'))^2\\
&\le (\P(\tilde{\tau}^{(1)} \le 2L')) + \P(\tilde{\zeta}^{(1)} \ge L'))^2\;.
\end{align*}
Applying the same arguments as in the previous proof, we get
$$ k\P\big(\lfloor {\alpha(L')} \rfloor_\pi \ge 2\big) \le k\Big(\frac{m_\gl}{L'} + \frac{5h}{k}\Big)^2\;,$$
which goes to $0$ since $m_\gl \asymp \Eg^{-3/2}$ and provided\footnote{This is another restriction on the speed at which $k$ can go to $\infty$.} $k^3 \ll L^2 \Eg$.
\end{proof}

\subsection{Proof of Lemma \ref{Lemma:Eventalphatheta}}

From now on, we write $\alpha$ for $\{\alpha\}_\pi$ and $\theta_\gl$ for $\{\theta_\gl\}_\pi$. Recall that we drop the superscript $(\Eg)$ although the proof is carried out simultaneously in both regimes. For simplicity we define $L' := L/(k\Eg)$.\\
We set $\epsilon_L := 1/\ln\ln\ln L'$ in the Bulk regime and $\epsilon_L := 1/(\ln\ln\ln L' \wedge \ln E)$ in the Crossover regime. We note that this choice is relatively arbitrary. From the estimates on the invariant measure stated in Lemma \ref{Lemma:InvMeas} and the convergence of the densities of Theorem \ref{Th:CVDensities}, there exist two constant $c,C>0$ such that
$$ \P(\theta_\gl(L') > \pi-3\epsilon_L) \le c(\epsilon_L + e^{-CL'})\;,$$
for all $L$ large enough and therefore $k\P(\theta_\gl(L') > \pi-3\epsilon_L)$ goes to $0$ as $L\to\infty$ (provided $k$ goes to $\infty$ slowly enough). To establish the lemma it suffices to show that
\begin{equation}\label{Eq:alphaToprove}
\lim_{L\to\infty} k\P(\alpha(L') > \frac52 \epsilon_L) = 0\;.
\end{equation}
Recall that $(\theta_\gl,\alpha)$ is markovian. If we can show that for some $C>0$
\begin{equation}\label{Eq:alphaToProve2} \lim_{L\to\infty} k \sup_{\theta_0, \alpha_0 \in [0,\pi)} \P(\alpha(C \ln L') > \frac52 \epsilon_L \,|\, \theta_\gl(0) = \theta_0, \alpha(0) = \alpha_0) = 0\;,\end{equation}
then the Markov property applied at time $L'-C\ln L'$ yields \eqref{Eq:alphaToprove}. The order of magnitude $\ln L'$ will be justified by the discussion below.\\

The proof of this convergence relies on a thorough study of the process
$$ R := \log \tan (\alpha/2) \in [-\infty,\infty)\;,$$
which happens to behave very much like a diffusion in $\R$ within a potential represented on Figure \ref{Fig:R} (note that this potential is similar to the one studied in \cite{AllezDumazSine} for the small $\beta$ limit of the Sine$_\beta$ process). The main features of this potential are: $-\infty$ is an entrance point while $+\infty$ is an exit point ; the potential admits a well centered at $-\ln L'$ and an unstable equilibrium point near $0$ ; the drift generated by this potential is (roughly) a negative constant, resp.~a positive constant, on $(-\ln L', 0)$, resp.~on $(0,\ln L')$.\\
Actually, the process $\alpha$ (and therefore the process $R$) is not markovian as its evolution depends on the Markov process $\theta_\gl$: however, $\theta_\gl$ evolves at a smaller time-scale than $R$ so that after ``averaging'' over the evolution of $\theta_\gl$, the process $R$ can be seen as a diffusion. One difficulty in our proofs will then consist in controlling the error made upon this replacement.\\

To alleviate the notation, we write $\P_{\theta_0,R_0}$ for the law of $(\theta_\gl,R)$ starting from $(\theta_0,R_0)$ at time $0$. The convergence \eqref{Eq:alphaToProve2} is implied by
\begin{equation}\label{Eq:RToProve} \lim_{L\to\infty} k \sup_{\theta_0 \in [0,\pi), R_0 \in [-\infty,\infty)} \P_{\theta_0,R_0}(R(C \ln L') > -\ln \epsilon_L^{-1} ) = 0\;.\end{equation}
Set $u_L := \ln \epsilon_L^{-1}$. We divide the proof of \eqref{Eq:RToProve} according to the initial position of $R$. The first lemma shows that if $R$ starts in the interval $[-\infty,-2 u_L]$, that is, within the well of the potential of Figure \ref{Fig:R}, then it typically remains in $[-\infty,-u_L)$ up to time $C \ln L'$.
\begin{lemma}[Small initial values]\label{Lemma:Small}
For any constant $C>0$ we have
$$ \lim_{L\to\infty} k \sup_{\theta_0 \in [0,\pi), R_0 \in [-\infty,-2 u_L]} \P_{\theta_0,R_0}(\sup_{[0,C\ln L']} R > -u_L) = 0\;.$$
\end{lemma}
The second lemma shows that if $R$ starts in the interval $[-2 u_L,2 u_L]$, that is, near the unstable equilibrium point of the potential of Figure \ref{Fig:R}, then it typically escapes this interval by time $2\ln L'$.
\begin{lemma}[Intermediate initial values]\label{Lemma:Intermediate}
We have
$$ \lim_{L\to\infty} k \sup_{\theta_0 \in [0,\pi), R_0 \in [-2 u_L,2 u_L]} \P_{\theta_0,R_0}(|R| \mbox{ does not hit } 2 u_L \mbox{ by time } 2\ln L' ) = 0\;.$$
\end{lemma}
Finally we show that if $R$ starts in the interval $[2 u_L,\infty)$, that is, near the exit point $+\infty$ of the potential of Figure \ref{Fig:R}, then it typically explodes to $+\infty$ within a time of order $\ln L'$. As $R$ restarts from $-\infty$ when it hits $+\infty$, we are then back to the regime of the first lemma.
\begin{lemma}[Large initial values]\label{Lemma:Large}
There exists a constant $C>0$ such that
$$ \lim_{L\to\infty} k \sup_{\theta_0 \in [0,\pi), R_0 \in [2 u_L,\infty)} \P_{\theta_0,R_0}(R \mbox{ does not hit } +\infty \mbox{ by time }C\ln L' ) = 0\;.$$
\end{lemma}
\noindent It is straightforward to deduce \eqref{Eq:RToProve} from the three lemmas and the Markov property.

\medskip

Let us now write down the SDE solved by $R$
\begin{align*}
dR &= \Big(\sqrt \Eg (\mu-\gl) \sin^2(\alpha+\theta_\gl) \cosh R \;+\;\sqrt \Eg (\gl-\Eg) \sin(\alpha+2\theta_\gl) \;+\; \frac12 \cos(\alpha+2\theta_\gl)\\
&\quad\;+\; \frac{\tanh R}4(1+\cos (2 \alpha +4\theta_\gl))\Big) dt - \sin(\alpha+2\theta_\gl) dB(t)\;.
\end{align*}

Note that we intentionally left some occurrences of the process $\alpha$ in this expression. Note also that, depending on the range of values $\alpha$ (or equivalently, $R$) at which our analysis focuses, we will neglect some terms of this SDE.

\begin{figure}
\centering
\includegraphics[width = 6cm]{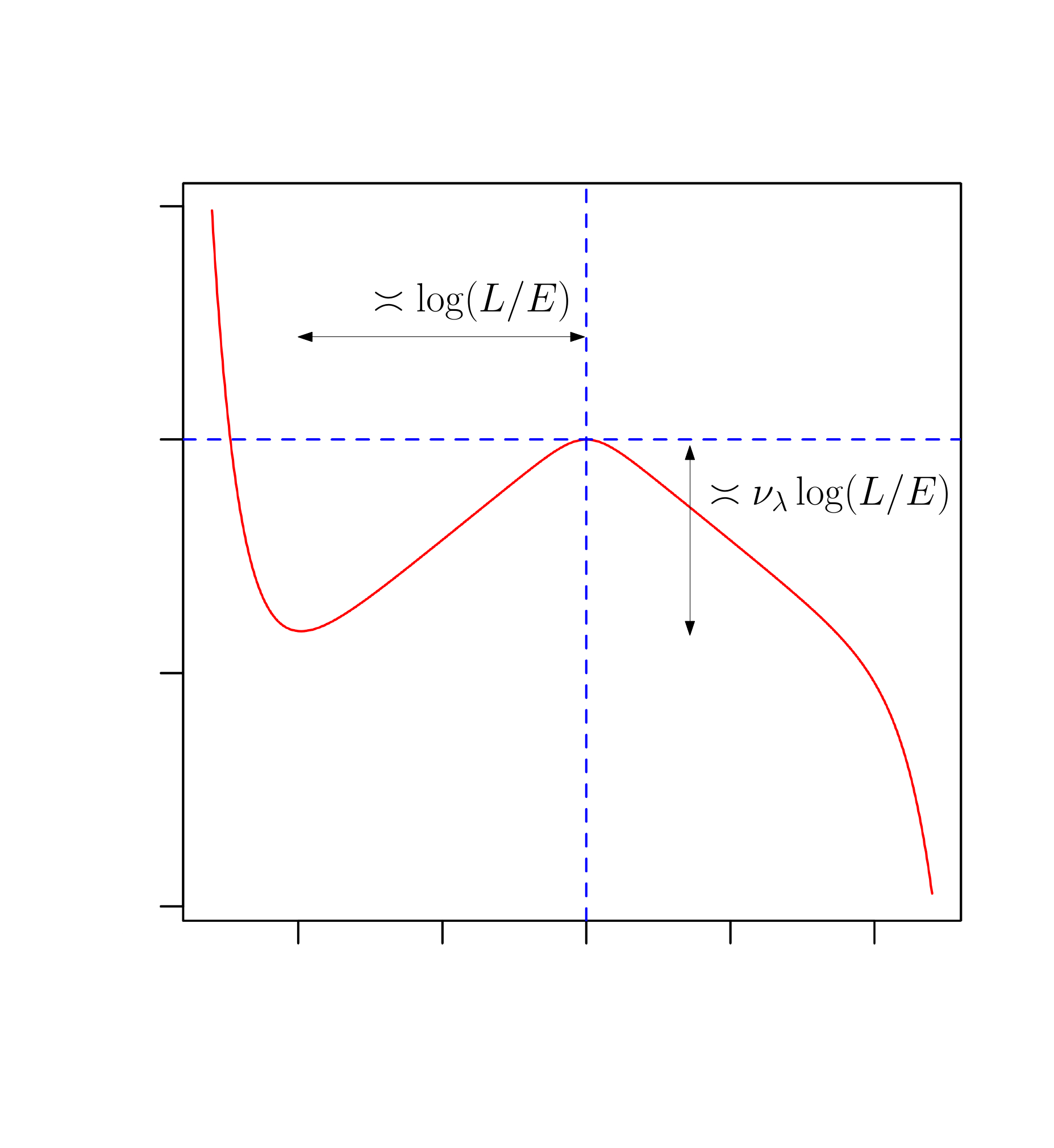}
\caption{Effective potential in which $R$ evolves.}\label{Fig:R}
\end{figure}

\begin{proof}[Proof of Lemma \ref{Lemma:Small}]
By monotonicity, it suffices to consider the case $R_0 = -2u_L$. Let $S := \inf\{t\ge 0: R(t) \in \{-(3/4) \ln L',-u_L\}\}$. We are going to prove that there exist $c,c'>0$ such that
\begin{equation}\label{Eq:alphatau}
\inf_{\theta_0\in [0,\pi)} \P_{\theta_0,R_0}(R(S) = -(3/4) \ln L' ; S > c' \ln L') > 1 - \epsilon_L^{c}\;.
\end{equation}
From this estimate and the Markov property, one deduces the statement of the lemma provided\footnote{This is a further constraint on $k$.} $k \ll \epsilon_L^{-c}$.\\
Let us write $dR(t) = A(t) dt + dM(t)$ where $M(t) := \int_0^t - \sin(\alpha+2\theta_\gl) dB(t)$ is the martingale part of the SDE. One can check that for all $t \le S$, the drift term satisfies 
\begin{align*}
A(t) = C(t) + \cO(\epsilon_L + (L')^{-1/4})\,,
\end{align*}
where
$$ C(t) = \sqrt \Eg (\gl-\Eg) \sin(2\theta_\gl) + \frac12 \cos(2\theta_\gl) - \frac14(1+\cos(4\theta_\gl))\;.$$
By \eqref{Eq:Expectlnr} we have $\E[\int_0^{\zeta_\gl} C(s) ds] = -\nu_\gl m_\gl$. Recall that the expectation of $\zeta_\gl$ is $m_\gl$. Therefore for $t$ large enough and as long as $t < S$, the process $R(t)-R_0$ should roughly behave like $-\nu_\gl t + M(t)$ so that it has negligible probability to reach large values.

To put that on firm ground, introduce $d\tilde{R}(t) := C(t) dt + dM(t)$ with $\tilde{R}(0) := R_0$. For any $\epsilon > 0$ there exists $q>0$ such that
$$ \sup_{L>1}\sup_{\theta_0}\E_{\theta_0,R_0}[\sup_{t\ge 0} e^{q |\tilde{R}(t) - R_0 + \nu_\gl t|} e^{-q\epsilon t}] < \infty\;.$$
Indeed for the integral $\int_0^t C(s)ds$, it follows from the very same arguments as in Lemma \ref{Lemma:Zgl} and we can deal with the martingale part using Lemma \ref{Lemma:Mgale}.

Using $\sup_{s \leq t }|R(s) - \tilde R(s)| \leq \cO(\epsilon_L + (L')^{-1/4}) \times t$ for all $t \leq S$, we therefore deduce that
$$ \sup_{L>1}\sup_{\theta_0}\E_{\theta_0,R_0}[\sup_{t\in [0,S]} e^{q |R(t) - R_0 + \nu_\gl t|} e^{-2q\epsilon t}] < \infty\;.$$
Call $K$ the expression on the l.h.s. Provided $2\epsilon < \nu_\gl$ we deduce that
$$ \P_{\theta_0,R_0}(R(S) = -u_L) \le K e^{-q u_L}\;.$$
A simple computation allows to deduce that there exist $c'>0$ and $C'>0$ such that
$$ \P_{\theta_0,R_0}\big(R(S) = -(3/4) \ln L' \;;\;S \le c' \ln L'\big) \le K e^{-q C' \ln L'}\;.$$
This concludes the proof of \eqref{Eq:alphatau}.
\end{proof}
For our next proof, recall the definitions of the hitting times \eqref{Eq:Defzetai} and set $T_k = \zeta^{(1)} +\ldots + \zeta^{(k)}$. We set $n := \lfloor \Eg^{3/2} \rfloor$ and $\tau_k := T_{nk}$ for every $k\ge 1$.

\begin{proof}[Proof of Lemma \ref{Lemma:Intermediate}]
It suffices to prove the statement of the lemma with $\theta_0 = 0$ and with $2\ln L'$ replaced by $\ln L'$. Indeed, if $\theta_0 \ne 0$, then by Lemma \ref{Lemma:Moments} there exists $c>0$ such that with probability at least $1-e^{-c \ln L' \Eg^{3/2}}$ the process $\{\theta_\gl\}_\pi$ hits $0$ by time $\ln L'$. Therefore, we now assume that $\theta_0 = 0$.\\
Set $N := C \ln L'$ and define the event $\cB := \{\tau_N < \ln L'\}$. If $C > 0$ is small enough then by Lemma \ref{Lemma:Moments} there exists $c>0$ such that for all $L$ large enough we have
$$\P(\cB^\cc) < e^{-c \Eg^{3/2} \ln L'}\;.$$
Introduce the stopping time
$$ S := \inf\{t\ge 0: R(t) \notin [-2 u_L,2  u_L]\}\;.$$
We claim that on any excursion of the diffusion $\theta_\gl$ from $nk\pi$ to $n(k+1)\pi$, $R$ has a small chance to escape from $[-2u_L,2u_L]$ i.e. for every $k\ge 0$, we have
\begin{equation}\label{Eq:Claimtauk} \P(\tau_k < S \le \tau_{k+1}) \ge \delta_L \times \P(S > \tau_k)\;,\end{equation}
where $\delta_L := \exp(-\epsilon_L^{-33})$ is small but large enough for our purposes.

We postpone the proof of the claim and proceed with the proof of the lemma. For every $k\ge 0$, we have
\begin{align*}
\P(S \le \tau_{k+1} \,|\, S > \tau_k) = \frac{\P(\tau_k< S \le \tau_{k+1})}{\P(S>\tau_k)} \ge \delta_L\;.
\end{align*}
Consequently
$$ \P(S \le \tau_{k+1}) \ge \P(S \le \tau_k) + \delta_L(1-\P(S \le \tau_k))\;,$$
and a simple recursion yields
$$ \P(S \le \tau_{N}) \ge 1 - (1-\delta_L)^N\;.$$
Therefore
\begin{align*}
\P(S \le \ln L') &\ge \P(\cB \cap \{S \le \tau_{N}\}) \ge \P(S \le \tau_{N}) - \P(\cB^\cc)\ge 1 - (1-\delta_L)^N - e^{-c \Eg^{3/2} \ln(L')}\;
\end{align*}
and
$$ k \P(S > \ln L') < k e^{-C \ln L' \delta_L} + k e^{-c \Eg^{3/2} \ln(L')}\;,$$
which goes to $0$ since $\epsilon_L^{-1} \le \ln\ln\ln L'$ and provided $k$ does not go too fast to $+\infty$ (for instance provided $k \ll \ln L'$).\\

We now prove \eqref{Eq:Claimtauk}. By the strong Markov property applied at time $\tau_k$, it suffices to prove that
\begin{equation}\label{Eq:BoundSalpha}
\inf_{R_0 \in [-2 u_L,2 u_L]} \P_{0,R_0}(S < \tau_1) \ge \delta_L\;.
\end{equation}
Let us write $dR(t) = A(t) dt + dM(t)$ where
$$ M(t) = \int_0^t - \sin(\alpha+2\theta_\gl) dB(s)\;.$$
Set $S' := \inf\{t\ge 0: R(t) \notin [-10 u_L,10 u_L]\}$. Set $\alpha_\pm := 2\arctan e^{\pm 10 u_L}$ and introduce the martingale
$$N_t := - \int_{0}^{t}  \sin(\alpha_- \vee \alpha \wedge \alpha_+ +2\theta_\gl) dB(s)\;,$$
which coincides with $M_t$ for all time $t \leq S'$.
%
%
There exists $C>0$ such that whenever $t\le S'$, the drift satisfies $|A(t)| < C$ almost surely, and therefore
\begin{align}\label{Eq:DoubleBdR} -Ct + N_t \leq R(t)-R(0) \leq Ct + N_t\;.\end{align}

We then note that
\begin{align*}
\P(S < \tau_1) &\ge \P(\exists t < \tau_1 \wedge S': N_t \ge Ct + 4 u_L)\\
&\ge \P(\exists t < (2nm_\gl)\wedge \tau_1 \wedge S': N_t \ge 5 u_L)\\
&\ge \P(\exists t < (2nm_\gl)\wedge \tau_1 \wedge S': \sup_{[0,t]} N = 5 u_L \;, \inf_{[0,t]} N \ge - u_L)\;.
\end{align*}
Note that, if at some time $t < (2nm_\gl)\wedge \tau_1$ we have $\sup_{[0,t]} N = 5 u_L$ and $\inf_{[0,t]} N \ge - u_L$ then by the inequalities \eqref{Eq:DoubleBdR} we have $t < S'$. Consequently
\begin{align*}
\P(S < \tau_1) &\ge \P(\exists t < (2nm_\gl)\wedge \tau_1: \sup_{[0,t]} N = 5 u_L \;, \inf_{[0,t]} N \ge - u_L)\;.
\end{align*}

Let $q_r := \inf\{s\ge 0: \langle N \rangle_s \ge r\}$, and define $\beta_r := N_{q_r}$. By Dubins-Schwarz's Theorem, $\beta$ is a Brownian motion. We thus get for any $r>0$
\begin{equation}\begin{split}
\P(S < \tau_1) &\ge \P(\beta_{r} \ge 5 u_L ;  \inf_{[0,r]} \beta \ge - u_L ; q_r <  (2nm_\gl) \wedge \tau_1)\\
&\ge  \P(\beta_{r} \ge 5 u_L ; \inf_{[0,r]} \beta \ge - u_L) - \P(q_r \ge (2nm_\gl) \wedge \tau_1)\;.\label{Eq:Stau1}
\end{split}\end{equation}
We now set $r := \epsilon_L^{32}$. By the Reflection Principle~\cite[Exercise III.3.14]{RevuzYor} and standard Gaussian estimates there exists $c>0$ such that
\begin{equation}\label{Eq:betar}\P(\beta_{r} > 5 u_L ; \inf_{[0,r]} \beta \ge - u_L) = \P(5 u_L < \beta_r < 7 u_L) \ge e^{-c \frac{u_L^2}{r}}\;.\end{equation}
To bound from above $\P(q_r \ge (2nm_\gl) \wedge \tau_1)$, we distinguish the Bulk and the Crossover regimes.\\
We start with the Crossover regime. We have
\begin{align*}
\P(q_r \ge (2nm_\gl) \wedge \tau_1) &= \P(\langle N\rangle_{(2n m_\gl) \wedge \tau_1} \le r)\le \P(\tau_1 > 2n m_\gl) + \P(\langle N\rangle_{\tau_1} \le r)\;.
\end{align*}
From Lemma \ref{Lemma:Moments}, we deduce the existence of $c'>0$ such that
$$ \P(\tau_1 > 2n m_\gl) \le e^{-c' E^{3/2}}\;.$$
Regarding the second term, we write
$$ \langle N\rangle_{\tau_1} = \int_0^{\tau_1} \sin^2(\alpha_-\vee \alpha \wedge \alpha_+ + 2\theta_\gl) ds \ge \frac12 \epsilon_L^{20} \int_0^{\tau_1} \un_{\{|\theta_\gl(s)| \le \epsilon_L^{11}\}} ds\;.$$
Note that the integral on the r.h.s.~is a sum of $n$ i.i.d.~r.v.~$X_k$ with
$$ X_1 = \int_0^{\zeta_\gl} \un_{\{|\theta_\gl(s)| \le \epsilon_L^{11}\}} ds\;.$$
It is not difficult to check that for $q=1/2$ (this is arbitrary) there exists $\delta > 0$ such that for all $L$ large enough
$$ \P(X_1 > \frac{\delta \epsilon_L^{11}}{n}) \ge q\;.$$
Indeed, the drift of the diffusion is of order $n=\lfloor E^{3/2}\rfloor$ and the quadratic variation of the martingale part is small until $\theta_\gl$ has hit $\epsilon_L^{11}$ so that Lemma \ref{Lemma:Mgale} allows to conclude.

Consequently optimizing the Laplace transform of the Binomial distribution one gets the existence of $c''>0$ such that
\begin{align*}
\P(\langle N\rangle_{\tau_1} \le r) \le \P\big(\sum_{k=1}^n \frac{\delta}{2n} \epsilon_L^{31}  \un_{\{X_k> \frac{\delta \epsilon_L^{11}}{n}\}} \le r\big) &\le \P(\mbox{Bin}(n,q) \le \frac{2n \epsilon_L}{\delta})\le e^{-c'' E^{3/2}}\;.
\end{align*}
We henceforth obtain (recalling that $u_L = \ln \epsilon_L^{-1}$)
$$ \P(S < \tau_1) \ge e^{-c \frac{\ln^2 \epsilon_L^{-1}}{\epsilon_L^{32}}} - e^{-c' E^{3/2}} - e^{-c'' E^{3/2}}\;.$$
Since $\epsilon_L^{-1} \le \ln E$, we obtain \eqref{Eq:BoundSalpha}.\\
Let us now consider the Bulk regime, for which $n=1$ and $\tau_1 = \zeta_\gl$. We have
\begin{align*}
\P(q_r \ge (2nm_\gl) \wedge \tau_1) &= \P(\langle N\rangle_{(2 m_\gl) \wedge \zeta_\gl} \le r)\;.
\end{align*}
Then we introduce $\sigma := \inf\{t\ge 0: \theta_\gl(t) = \epsilon_L^{11}\}$, which is smaller than $\zeta_\gl$, and we write
\begin{align*}
\langle N\rangle_{(2 m_\gl) \wedge \zeta_\gl} = \int_0^{(2 m_\gl) \wedge \zeta_\gl} \sin^2(\alpha_-\vee \alpha \wedge \alpha_+ + 2\theta_\gl) ds \ge \frac12 \epsilon_L^{20}\big( \sigma \wedge (2 m_\gl)\big)\;.
\end{align*}
Since $m_\gl$ is of order $1$, we thus have
$$ \P(\langle N\rangle_{(2 m_\gl) \wedge \zeta_\gl} \le r) \le \P( \sigma \wedge (2 m_\gl) \le 2 \epsilon_L^{12}) = \P(\sigma \le 2 \epsilon_L^{12})\;.$$
To bound this last term, we consider the process $\theta_\gl(\cdot \wedge \sigma)$. Its drift term is bounded by some constant $C_\alpha > 0$ (which is of order $1$) and the quadratic variation of its martingale term is bounded by $C_\beta = \epsilon_L^{44}$. By Lemma \ref{Lemma:Mgale}, we thus get
$$ \P(\sigma \le 2 \epsilon_L^{12}) \le \P(\sup_{s\in [0,2 \epsilon_L^{12}]} \theta_\gl(s\wedge \sigma) \ge \frac12 \epsilon_L^{11}) \le 2e^{-\frac{\epsilon_L^{-34}}{64}}\;.$$
Consequently
$$ \P(S < \tau_1) \ge e^{-c \frac{\ln^2 \epsilon_L^{-1}}{\epsilon_L^{32}}} - 2e^{-\frac{\epsilon_L^{-34}}{64}} \ge \delta_L\;.$$
\end{proof}

For large initial values, the arguments are similar to those presented in the case of small initial values so we do not provide a detailed proof.
\begin{proof}[Sketch of proof of Lemma \ref{Lemma:Large}]
By monotonicity, it suffices to consider the case $R(0) = 2 u_L$. Whenever $R(t) \in [u_L,\infty)$,
\begin{align*}
dR(t) &= -\Big(\sqrt\Eg(\gl-\Eg) \sin(2\theta_\gl) + \frac12 \cos(2\theta_\gl) - \frac14(1+\cos(4\theta_\gl))\Big) dt\\
&\quad + \sqrt \Eg (\mu-\gl) \sin^2\theta_\gl \frac{\sinh^2 R}{\cosh R} dt- \sin(\alpha+2\theta_\gl) dB(t) + \cO(\epsilon_L + L^{-1}) dt\;.
\end{align*}
Until $R$ is of order $\ln L'$, the dominant term in the drift is
$$ -\Big(\sqrt\Eg(\gl-\Eg) \sin(2\theta_\gl) + \frac12 \cos(2\theta_\gl) - \frac14(1+\cos(4\theta_\gl))\Big)\;,$$
which is the exact opposite of the drift in the regime of small initial values of Lemma \ref{Lemma:Small}: therefore, in average the process $R$ increases at speed $\nu_\gl$. When $R$ is of order at least $\ln L'$, then the term
$$ \sqrt \Eg (\mu-\gl) \sin^2\theta_\gl \frac{\sinh^2 R}{\cosh R}\;,$$
becomes predominant and makes the process $R$ explode in finite time with large probability.
\end{proof}

\section{Poisson statistics}\label{Sec:Poisson}

In this section, we identify the limiting intensity of $\cN_L^{(j)}$ stated in Proposition \ref{Prop:Intensity}, and we show that $\cN_L$ is well-approximated by $\bar{\cN}_L$, that is, we prove Proposition \ref{Prop:Approx}. We start with the latter.

\subsection{Controlling the approximation}\label{Subsec:Approx}

We fix $h>0$ and set $\Delta = [E-h/(Ln(E)),E+h/(Ln(E))]$.\\

Recall the definition of the points $(t_j)$ at the beginning of Subsection \ref{Subsec:StrategyPoisson}. Let us introduce for $n\ge 1$ and $j\in \{0,\ldots,k\}$, the following neighborhood of $t_j$:
$$ \cD_j(n) := [t_j-nt_L \Eg, t_j + nt_L \Eg] \cap [-L/2,L/2]\;,$$
where $t_L$ is such that\footnote{This implicitly forces $k$ to go to $\infty$ not too fast.} $\ln( L/\Eg) \ll t_L \ll L/(k\Eg)$. We also set $\cD(n) := \cup_j \cD_j(n)$.\\

Recall that we have already proven Theorem \ref{Th:Loc}. Fix $\eps > 0$ small enough. We introduce the event $\cB_L$ on which:\begin{enumerate}[label = (\roman*)]
\item for every $\lambda_i\in\Delta$, we have
$$ \Big(\varphi_i(t)^2 + \frac{\varphi_i'(t)^2}{\Eg}\Big)^{1/2} \le \frac{t_L}{\sqrt \Eg} e^{-\frac{1}{2}(\boldsymbol{\nu_E}-\eps)\frac{|t-U_i|}{\Eg}}\;,\quad \forall t\in [-L/2,L/2]\;,$$
\item for every $\lambda_i \in \Delta$, we have $U_i \notin \cD(3)$,
\item we have $\int_{\cD(2)} \Big(\sum_{\lambda_i \in \Delta} |\varphi_i| + \frac{|\varphi_i'|}{\sqrt \Eg}\Big)^2 ds < e^{-(\boldsymbol{\nu_E}-2\eps) t_L}$,
\end{enumerate}
\begin{lemma}\label{Lemma:BL}
We have $\P(\cB_L) \to 1$ as $L\to\infty$.
\end{lemma}
\begin{proof}
The proof requires the following preliminary estimate: there exists $c> 0$ such that uniformly over all $L$ large enough we have
\begin{equation}\label{Eq:4p}
\E\Big[\int_{\cD(4)} \sum_{\lambda_i \in \Delta} \Big(\varphi_i(v)^2 + \frac{\varphi_i'(v)^2}{\Eg}\Big) dv\Big] \le c \frac{kt_L \Eg}{L}\;.
\end{equation}
This bound easily follows from Proposition \ref{Prop:GMPalternative} (with $a=1$, $b=\Eg^{-1}$ and $G\equiv 1_{\Delta}$) and Theorem \ref{Th:CVDensities}.\\
Applying Markov's inequality in the estimates of Theorem \ref{Th:Loc}, we see that (i) holds true with a probability that goes to $1$ as $L\to\infty$. We now work on this event.\\
Assume that there exists $\lambda_i \in \Delta$ such that $U_i \in \cD(3)$. We find
$$ \int_{|s-U_i| \le t_L \Eg} \varphi_i^2(s)ds = 1 - \int_{|s-U_i| > t_L \Eg} \varphi_i^2(s)ds \ge 1 - \frac{2t_L^2}{\boldsymbol{\nu_E}-\eps} e^{-(\boldsymbol{\nu_E}-\eps)t_L}\;.$$
By assumption $\{s\in [0,L]: |s-U_i| \le t_L \Eg\} \subset \cD(4)$. Markov's inequality applied to \eqref{Eq:4p} and our assumption $t_L \ll L/(k\Eg)$ therefore show that the probability of (ii) goes to $1$.\\
We now prove (iii). By Cauchy-Schwarz's inequality we have
\begin{align*}
\Big(\sum_{\lambda_i \in \Delta} |\varphi_i| + \frac{|\varphi_i'|}{\sqrt \Eg} \Big)^2 &\le N_L(\Delta) \sum_{\lambda_i \in \Delta} \Big(|\varphi_i| + \frac{|\varphi_i'|}{\sqrt \Eg}\Big)^2\\
&\le 2 N_L(\Delta) \sum_{\lambda_i \in \Delta} \varphi_i^2 + \big(\frac{\varphi_i'}{\sqrt \Eg}\big)^2\;.
\end{align*}
We deduce that on the event where (i) and (ii) hold we have
\begin{align*}
\int_{\cD(2)} \Big(\sum_{\lambda_i \in \Delta} |\varphi_i| + \frac{|\varphi_i'|}{\sqrt \Eg} \Big)^2 ds &\le 4 N_L(\Delta)^2 \frac{t_L^2}{\boldsymbol{\nu_E}-\eps}e^{-(\boldsymbol{\nu_E}-\eps)t_L}\;.
\end{align*}
Markov's inequality combined with Proposition \ref{Prop:Moment} shows that $\P(N_L(\Delta)^2 < t_L) \to 1$. We thus deduce (iii) (we use a factor $\eps t_L$ in the exponential term to ``kill'' the prefactors).
\end{proof}

We also need estimates on the exponential decay of the $\varphi_i^{(j)}$. We let $\bar{\cB}_L$ be the event on which for every $\lambda_i^{(j)}\in\Delta$, we have
$$ \Big(\varphi_i^{(j)}(t)^2 + \frac{(\varphi_i^{(j)})'(t)^2}{\Eg}\Big)^{1/2} \le \frac{t_L}{\sqrt \Eg} e^{-\frac{1}{2}(\boldsymbol{\nu_E}-\eps)\frac{|t-U_i^{(j)}|}{\Eg}}\;,\quad \forall t\in (t_{j-1},t_j)\;.$$

\begin{lemma}
We have $\P(\bar{\cB}_L) \to 1$ as $L\to\infty$.
\end{lemma}
\begin{proof}
A simple adaptation of the proof of Proposition \ref{Prop:ExpoDecayCrossover} allows to show the counterpart of Theorem \ref{Th:Loc} for the operators $\cH_L^{(j)}$. Namely, there exist some r.v.~$c_i^{(j)}$ such that for every $\gl_i^{(j)}\in\Delta$
\begin{equation}\label{Eq:PSDecaybar}
\Big(\varphi_i^{(j)}(t)^2 + \frac{(\varphi_i^{(j)})'(t)^2}{\Eg}\Big)^{1/2} \le \frac{c_i^{(j)}}{\sqrt \Eg} e^{-\frac{1}{2}(\boldsymbol{\nu_E}-\eps)\frac{|t-U_i^{(j)}|}{\Eg}}\;,\end{equation}
and for some $q>0$ we have
$$ \limsup_{L\to\infty} \E\Big[\sum_{j=1}^{k}\sum_{\gl_i^{(j)} \in \Delta} (c_i^{(j)})^q\Big] < \infty\;.$$
This being given, the proof of the lemma follows from Markov's inequality.
\end{proof}

Finally, we need some control on the \emph{gaps} between the eigenvalues of $\cH^{(j)}_L$ and on the distance of these eigenvalues to the boundary of $\Delta$. Take some $\delta_L \to 0$ as $L\to\infty$. Let $\cG_L$ be the event on which for any $(i,j) \ne (i',j')$ we have
\begin{equation}\label{Eq:Gap} \lambda_i^{(j)}, \lambda_{i'}^{(j')} \in \Delta \Longrightarrow  | Ln(E)(\lambda_i^{(j)} - \lambda_{i'}^{(j')})| > \delta_L\;,\end{equation}
and
$$ \forall i,j\;,\quad \mbox{dist}(Ln(E)(\lambda_i^{(j)} - E), \partial \Delta) > \delta_L\;.$$
Since we already know from the arguments\footnote{The only missing ingredient is the proof of Proposition \ref{Prop:Intensity} which is presented in the next subsection and does not rely on the present arguments.} in Subsection \ref{Subsec:StrategyPoisson} that $\bar{\cN}_L$ converges to a Poisson point process, we deduce that
$$ \P(\cG_L) \to 1\;,\quad L\to\infty\;.$$
Note that this convergence holds for any given sequence $\delta_L$ that converges to $0$. In the proof below, we will need to impose some restriction on the speed at which $\delta_L$ goes to $0$.\\

Let us also state a simple fact of the theory of generalized Sturm-Liouville operators. The domain of $\cH_L$ is given by
\begin{align*}
\Big\{f \in L^2([0,L]):& \;f(0) = f(L) = 0,\; f \mbox{ A.C.}, \; f'- Bf  \mbox{ A.C.},\\
&\mbox{and } -(f'-Bf)' - Bf' \in L^2([0,L])\Big\}\,.
\end{align*}
As a consequence, if one multiplies an element of the domain by some smooth function, compactly supported in $(t_{j-1},t_j)$, then one gets an element of the domain of $\cH_L^{(j)}$. We will use this fact in the next proof.\\

Finally, let us recall the L\'evy-Prokhorov distance on $\cM=\cM(\R)$ that metrizes the weak convergence topology:
\begin{equation}\label{Eq:LevyPro}
d_\cM(w,w') := \inf\big\{\epsilon > 0: \forall B\in \cB(\R),\; w(B) \le w'(B^\epsilon) + \epsilon \mbox{ and } w'(B) \le w(B^\epsilon) + \epsilon\big\}\;,
\end{equation}
where $B^\epsilon$ is the $\epsilon$-neighborhood of $B$.\\
Recall also that we actually work on $\bar{\cM}=\cM(\bar\R)$ and that one can define $d_{\bar\cM}$ similarly as above: the only difference is that the $\epsilon$-neighborhoods need to be taken w.r.t.~a distance that metrizes $\bar\R$. If one chooses this distance in such a way that the $\epsilon$-neighborhoods on $\bar\cM$ always contain the $\epsilon$-neighborhoods on $\cM$, then one can check that for any measures $w,w' \in \cM(\R)$ we have
$$ d_{\bar{\cM}}(w,w') \le d_{\cM}(w,w')\;.$$
Consequently, since all the measures that we manipulate are actually elements of $\cM$, in the sequel we will only deal with $d_\cM$.

\begin{proof}[Proof of Proposition \ref{Prop:Approx}]
Fix $\Delta := [E-h/(Ln(E)),E+h/(Ln(E))]$. We will show that, if we restrict ourselves to $[-h,h] \times [-1/2,1/2]\times\bar\cM$, then with large probability there is a one-to-one correspondence between the atoms of $\cN_L$ and of $\bar{\cN}_L$, and that the distance between the corresponding pairs of atoms goes to $0$ as $L\to\infty$. The proof consists of two steps.\\

\emph{Step 1.} We argue deterministically on the event $\cB_L\cap\bar{\cB}_L \cap\cG_L$.
Let $(\lambda,\varphi)$ be an eigenvalue/eigenfunction of $\cH_L$ such that $\lambda \in \Delta$. Let $U$ be its center of mass. From the conditions stated in the event $\cB_L$, we know that there exists $j\in \{1,\ldots,k\}$ such
$$ U \in (t_{j-1}+3 t_L \Eg,t_j - 3 t_L \Eg)\;.$$
We consider a smooth function $\chi_j : [-L/2,L/2] \to [0,1]$ that equals $1$ on $[t_{j-1}+2t_L \Eg,t_j-2t_L \Eg]$ and $0$ on $(t_{j-1} + t_L \Eg,t_j-t_L \Eg)^\cc$, and such that
$$ \sup_t |\chi_j'(t)|\vee |\chi_j''(t)| \le \frac{2}{t_L\Eg}\;.$$
We then set
$$ \psi := \frac{\varphi \chi_j}{\|\varphi \chi_j\|_2}\;.$$
Note that $\psi$ is a compactly supported function in $(t_{j-1}, t_j)$ whose $L^2$ norm equals one. Furthermore, by (iii) of $\cB_L$
\begin{align}\label{diffpsivarphi}
 \| \psi - \varphi\|_2^2 \lesssim e^{-Kt_L}\;,
\end{align}
for some constant $K>0$. From now on, the constants $K$ will never depend on $L$ and $j$, but may change from line to line.\\
By construction, $\psi$ belongs to the domain of $\cH_L^{(j)}$ and we have
$$ (\cH_L^{(j)} - \gl) \psi = \begin{cases} 0 &\mbox{ on } (t_{j-1}+2t_L \Eg, t_j - 2t_L \Eg)\\
\|\varphi \chi_j \|^{-1}_2 \; \big(- \varphi \chi_j'' - 2 \varphi' \chi_j'\big) &\mbox{ on } (t_{j-1}, t_{j-1}+2t_L \Eg]\cup [t_j - 2t_L \Eg,t_j)\;.
\end{cases}$$
Henceforth by (iii) of $\cB_L$
\begin{equation}\label{Eq:psi} \|(\cH_L^{(j)} - \gl) \psi \|_{L^2(t_{j-1},t_j)}^2 \le \frac{1}{\Eg} e^{-K t_L}\;.\end{equation}

On the other hand, we can expand $\psi$ on the $L^2(t_{j-1},t_j)$-basis made of the eigenfunctions of $\cH_L^{(j)}$:
$$ \|(\cH_L^{(j)} - \gl) \psi \|_{L^2(t_{j-1},t_j)}^2 = \sum_i |\lambda_i^{(j)} - \gl|^2 \langle \psi,\varphi_i^{(j)}\rangle^2\;.$$
The r.h.s.~is a convex combination of the $|\lambda_i^{(j)} - \gl|^2$. Given the bound \eqref{Eq:psi}, there must exist some $\ell \ge 1$ such that
$$ |\lambda_\ell^{(j)} - \gl|^2 \le \frac{1}{\Eg} e^{-K t_L}\;.$$

Recall that $t_L \gg \ln (L/\Eg)$ so that, provided $\delta_L$ does not go too fast to $0$ (its speed can be fixed arbitrarily), we have
$$ \frac{1}{\Eg} e^{-K t_L} \ll \Big(\frac{\delta_L}{Ln(E)}\Big)^2\;.$$
Given the definition of the event $\cG_L$ we deduce that the integer $\ell$ above is unique, that $\lambda_\ell^{(j)} \in \Delta$, and that we have (recall that $t_L \gg \ln (L/\Eg)$)
$$ \sum_{i: i\ne \ell} \langle \psi,\varphi_i^{(j)}\rangle^2 \le \frac{(Ln(E))^2}{\delta_L^2 \Eg} e^{-K t_L} \le e^{-K t_L}\;,$$
and therefore $\langle \psi,\varphi_\ell^{(j)}\rangle^2 \ge 1-e^{-K t_L}$, and thus
$$ \| \psi - \varphi_\ell^{(j)} \|_2^2 \le e^{-K t_L}\;.$$
Together with \eqref{diffpsivarphi}, we find
$$ \| \varphi - \varphi_\ell^{(j)}\|^2_2 \le  e^{-K t_L}\;.$$
Let $w(dt) = \Eg \varphi(U+t\Eg)^2 dt$ be the probability measure built from $\varphi$ and recentered at its center of mass denoted by $U$. Recall that $w_\ell^{(j)}$ (introduced in Subsection \ref{Subsec:StrategyPoisson}) is the corresponding object for $\varphi_\ell^{(j)}$. By $\cB_L$ and $\bar{\cB}_L$, we have some exponential decay of $\varphi$ and $\varphi_\ell^{(j)}$ from their centers of mass. We thus apply Lemma \ref{Lemma:Centers} and deduce the existence of a constant $C$, which is at most polynomial in $t_L$, such that
$$ \frac{|U-U_\ell^{(j)}|}{\Eg} \le C e^{-\frac{K}{4}t_L}\;,\quad d_\cM(w,w_\ell^{(j)}) \le C e^{-\frac{K}{4}t_L}\;.$$
These quantities vanish when $L\to\infty$.\\

We have built a map that associates to any eigenvalue $\gl \in \Delta$ of the operator $\cH_L$ an eigenvalue $\gl_\ell^{(j)} \in \Delta$ in such a way that
$$ |\lambda_\ell^{(j)} - \gl|^2 \le \frac{1}{\Eg} e^{-K t_L}\;,\quad \frac{|U-U_\ell^{(j)}|}{\Eg} \le C e^{-\frac{K}{4}t_L}\;,\quad d_\cM(w,w_\ell^{(j)}) \le C e^{-\frac{K}{4}t_L}\;,$$
as well as $\| \varphi - \varphi_\ell^{(j)}\|^2_2 \le  e^{-Kt_L}$. This map is necessarily injective: indeed, we cannot find two orthonormal functions $\varphi$ and $\tilde{\varphi}$ satisfying
$$ \| \varphi - \varphi_\ell^{(j)}\|^2_2 \le  e^{-Kt_L}\;,\quad \| \tilde\varphi - \varphi_\ell^{(j)}\|^2_2 \le  e^{-Kt_L}\;.$$
To conclude the proof, it remains to show that this map is actually bijective, this is the purpose of the second step.\\

\emph{Step 2.} Let $F_L$ be the event $\{\sum_{j=1}^k N_L^{(j)}(\Delta) \ge N_L(\Delta)\}$. By the first step, $\P(F_L) \to 1$ as $L\to\infty$. Proposition \ref{Prop:Moment} ensures uniform integrability of the collection of r.v.~$N_L(\Delta)$, $L>1$, so that
$$ \lim_{L\to\infty}\E[N_L(\Delta) \un_{F_L^\cc}] = 0\;.$$
Since
$$ 0 \le \E\Big[\big(N_L(\Delta)-\sum_j N^{(j)}_L(\Delta)\big)\un_{F_L^\cc}\Big] \le \E\Big[N_L(\Delta)\un_{F_L^\cc}\Big]\;,$$
we deduce that the term in the middle vanishes as $L\to\infty$. Combining this convergence with (1) and (2) of Proposition \ref{Prop:MicroEstimates}, we deduce that
$$ \lim_{L\to\infty} \E\Big[\big(\sum_j N^{(j)}_L(\Delta) - N_L(\Delta)\big)\un_{F_L}\Big] = \lim_{L\to\infty} \E\Big[\sum_j N^{(j)}_L(\Delta) - N_L(\Delta)\Big]= 0\;.$$
Since $(\sum_j N^{(j)}_L(\Delta) - N_L(\Delta))\un_{F_L}$ is a non-negative r.v. taking values in $\N$, we deduce that
$$ \lim_{L\to\infty}\P\Big(\sum_{j=1}^k N_L^{(j)}(\Delta) = N_L(\Delta)\Big) = 1\;.$$
Consequently, the injective map constructed in Step 1 on $\cB_L\cap\bar\cB_L\cap\cG_L$ is actually bijective, possibly on a smaller event, but whose probability still goes to one.
\end{proof}

\subsection{The limiting intensity}\label{Subsec:Intensity}

Let us start by defining the two random processes $Y_E$ and $Y_\infty$ introduced before the statement of Theorem \ref{Th:Shape}.
\medskip

In the Crossover regime, we have seen that for some two-sided Brownian motion $\cB$,
$$ Y_\infty(t) := \exp\Big(- \frac{|t|}{8} + \frac{\cB(t)}{2\sqrt 2}\Big)\;,\quad t\in\R\;.$$

In the Bulk regime, the definition of $Y_E$ requires more notations. Consider two independent adjoint diffusions $(\bar{\theta}^+_E(t), \bar{\rho}^+_E(t); t\ge 0)$ and $(\bar{\theta}^-_E(t), \bar{\rho}^-_E(t); t\ge 0)$, satisfying the SDEs \eqref{Eq:bartheta} and \eqref{Eq:barrho}, and starting from
$$ \bar{\theta}^+_E(0) = \theta\;,\quad \bar{\theta}^-_E(0) = \pi-\theta\;,\quad \bar{\rho}^+_E(0) = \bar{\rho}^-_E(0) = 0\;.$$
In other words, we work under the product measure $\bar{\P}^+_{(0,\theta)} \otimes \bar{\P}^-_{(0,\pi-\theta)}$ with the additional convention that $\bar{r}^\pm_E(0)=1$. We let $\bar{y}_E^\pm(t) := e^{\frac12 \bar{\rho}_E^\pm(t)} \sin \bar\theta_E^\pm(t)$. We define their concatenation\footnote{We follow $\bar{y}^-$ on $[0,\infty)$ and $\bar{y}^+(-\cdot)$ on $(-\infty,0]$ in order to be consistent with the definition of the concatenation of Subsection \ref{subsec:defforwardbackward}.}
$$ \hat{\bar{y}}_E(t) := \begin{cases} \bar{y}_E^-(t) &\mbox{ if } t\ge 0\;,\\
\bar{y}_E^+(t) &\mbox{ if } t\le 0\;.
\end{cases}$$
We then consider a mixture of these concatenations over different values of $\theta$: under the probability measure
$$ \frac{\mu_E(\theta) \mu_E(\pi-\theta) \sin^2 \theta}{n(E)} \bar{\P}^+_{(0,\theta)} \otimes \bar{\P}^-_{(0,\pi-\theta)}(\cdot) d\theta\;,$$
we define the process
\begin{align*}
 Y_E(t) := \hat{\bar{y}}_E(t)\;,\quad t\in\R\;.
\end{align*}

One can now provide the definition of the limiting probability measure $\boldsymbol{\sigma_E}$, which is a unified notation for the probability measures $\sigma_E$ in the Bulk regime and $\sigma_\infty$ in the Crossover regime that appear in the statement of Theorem \ref{Th:Shape}. Let us denote by $\boldsymbol{Y_E}$ the process $Y_E$ in the Bulk regime, and the process $Y_\infty$ in the Crossover regime.
In both regimes, $\boldsymbol{\sigma_E}$ is the law of the random element in $\bar\cM$ defined as
$$ \boldsymbol{w_E}(dt) = \frac{\boldsymbol{Y_E}(t+\boldsymbol{U_E})^2 dt}{\int \boldsymbol{Y_E}(t)^2 dt}\;,$$
where $\boldsymbol{U_E}$ is the associated center of mass
$$ \boldsymbol{U_E} :=  \frac{\int t \boldsymbol{Y_E}(t)^2 dt}{\int \boldsymbol{Y_E}(t)^2 dt}\;.$$

\medskip

The rest of this subsection is devoted to the proof of Proposition \ref{Prop:Intensity}. From now on, let us fix some function $f$ as in Proposition \ref{Prop:Intensity}, and we let $h>0$ be such that $f(\gl,\cdot,\cdot) = 0$ whenever $\gl \notin [-h,h]$. We also set $\Delta := [E- h/(Ln(E)) , E + h/(Ln(E))]$.\\

Recall the notations for the concatenation of the diffusions from Subsection \ref{subsec:defforwardbackward}. For $u \in [-\frac{L}{2k\Eg},\frac{L}{2k\Eg}]$, let us consider the product law $\P^{+}_{(-L/(2k\Eg), 0) \to (u,\theta)} \otimes \P^{-}_{(-L/(2k\Eg), 0) \to (-u,\pi -\theta)}$. With a slight abuse of notation, we still denote this product law by $\P^{(u)}_{(\theta,\pi-\theta)}$ (originally, this notation was for the time interval $[-L/(2\Eg),L/(2\Eg)]$). Then, we define under this product law the probability measure on $\R$ built from the concatenation process $\hat{y}_\gl^{(\Eg)}$:
$$ \frac{\hat{y}_\gl^{(\Eg)}(t)^2 dt}{\int \hat{y}_\gl^{(\Eg)}(t)^2 dt}\;.$$
The support of this measure is $[-\frac{L}{2k\Eg},\frac{L}{2k\Eg}]$. We let $\hat{U}_\gl/\Eg$ be the center of mass of this measure
\begin{equation}\label{Eq:hatUgl} \frac{\hat{U}_\gl}{\Eg} := \frac{\int t \,\hat{y}_\gl^{(\Eg)}(t)^2 dt}{\int \hat{y}_\gl^{(\Eg)}(t)^2 dt}\;.\end{equation}
We then recenter the probability measure by defining:
$$ \hat{w}_\gl^{(\Eg)}(dt) := \frac{\hat{y}_\gl^{(\Eg)}(\hat{U}_\gl/\Eg + t)^2 dt}{\int \hat{y}_\gl^{(\Eg)}(\hat{U}_\gl/\Eg + t)^2 dt}\;.$$

We introduce the measure
$$ \bbQ^{(u)}_\gl(\cdot):= \frac1{\sqrt\Eg \,n(\gl)}\int_{\theta=0}^{\pi} \mu^{(\Eg)}_\gl(\theta) \mu^{(\Eg)}_\gl(\pi-\theta) \sin^2 \theta \;\P^{(u)}_{(\theta,\pi-\theta)}(\cdot) d\theta\;.$$
By Corollary \ref{cor:DOSformula}, this is a probability measure. Our next lemma is the main argument for the identification of the limiting intensity of $\cN_L^{(j)}$.

\begin{lemma}\label{Lemma:CVmgl}
Take $t_L$ such that $\ln( L/\Eg) \ll t_L \ll L/(k\Eg)$. Uniformly over all $\gl \in \Delta$ and all $u\in [-\frac{L}{2k\Eg}+t_L, \frac{L}{2k\Eg}-t_L]$, the law of $\hat{{w}}_\gl^{(\Eg)}$ under $\bbQ^{(u)}_\gl$ converges weakly to the probability measure $\boldsymbol{\sigma_E}$ on $\bar\cM$.
\end{lemma}

With this lemma at hand, we can proceed with the proof of Proposition \ref{Prop:Intensity}.

\begin{proof}[Proof of Proposition \ref{Prop:Intensity}]
First of all, we have
$$ \int f d\cN^{(j)}_L  = \sum_{i\ge 1} f\big(Ln(E)(\lambda_i^{(j)}-E),U_i^{(j)}/L,w_i^{(j)} \big)\;.$$
Note that the operator at stake here is $\cH_L^{(j)}$ on $(t_{j-1},t_j)$. Let $a_j$ be the midpoint of $(t_{j-1},t_j)$. By the GMP formula of Proposition \ref{Prop:GMP}, using forward/backward processes on the interval $[-L/(2k\Eg),L/(2k\Eg)]$ and then shifting the evaluations by $a_j$, we find
\begin{align*}
&\E\Big[\sum_{i\ge 1} f\big(Ln(E)(\lambda_i^{(j)}-E),U_i^{(j)}/L ,w_i^{(j)}\big)\Big]\\
&= \sqrt \Eg \int_{u=-\frac{L}{2 k\Eg}}^{\frac{L}{2k\Eg}} \int_{\gl \in\R} \int_{\theta=0}^\pi p_{\gl,\frac{L}{2k\Eg}+u}^{(\Eg)}(\theta) p_{\gl,\frac{L}{2k\Eg}-u}^{(\Eg)}(\pi-\theta)\sin^2 \theta\\
&\qquad\qquad\times\E^{(u)}_{\theta,\pi-\theta}\Big[f\big(Ln(E)(\lambda-E) , (\hat{U}_\gl+a_j) / L , \hat{w}_\gl^{(\Eg)}\big)\Big]\, d\theta\, d\gl\, du\;.
\end{align*}

Note that $|\hat{U}_\gl/L| \le 1/(2k)$ and $|u\Eg/L| \le 1/(2k)$ so that they vanishes as $L\to\infty$. By the uniform continuity of $f$, up to some negligible term $o(1/k)$, we can replace the expectation term by
\begin{align*}
\E^{(u)}_{\theta,\pi-\theta}\Big[f(Ln(E)(\lambda-E) ,(a_j+u\Eg)/L, \hat{w}_\gl^{(\Eg)})\Big]\;.
\end{align*}
Let $t_L$ be such that $\ln (L/\Eg) \ll t_L \ll L/(k\Eg)$. In the previous expression, the integral over $u \in [-\frac{L}{2 k\Eg}, -\frac{L}{2 k\Eg}+t_L] \cup [\frac{L}{2 k\Eg}-t_L,\frac{L}{2 k\Eg}]$ yields a negligible term $o(1/k)$. On the remaining integral, Theorem \ref{Th:CVDensities} allows to replace the densities by the invariant measure, up to some negligible term $o(1/k)$. We are left with
\begin{align*}
&\sqrt \Eg \int_{u=-\frac{L}{2 k\Eg}+t_L}^{\frac{L}{2k\Eg}-t_L} \int_{\gl \in\R} \int_{\theta=0}^\pi \mu_{\gl}^{(\Eg)}(\theta) \mu_{\gl}^{(\Eg)}(\pi-\theta) \sin^2 \theta\\
&\qquad\qquad\qquad\times\E^{(u)}_{\theta,\pi-\theta}\Big[f(Ln(E)(\lambda-E) ,(a_j+u\Eg)/L, \hat{w}_\gl^{(\Eg)})\Big]\, d\theta \,d\gl\, du\\
&= \Eg \int_{u=-\frac{L}{2 k\Eg}+t_L}^{\frac{L}{2k\Eg}-t_L} \int_{\gl \in\R} n(\gl) \bbQ^{(u)}_{\gl}\Big[f(Ln(E)(\lambda-E) ,(a_j+u\Eg)/L, \hat{w}_\gl^{(\Eg)})\Big]\, d\theta \,d\gl \,du\;.
\end{align*}
Note that $n(\gl) / n(E) \to 1$ uniformly over all $\gl\in\Delta$. By Lemma \ref{Lemma:CVmgl}, the last expression is equivalent as $L\to\infty$ to
\begin{align*}
&\Eg \int_{u=-\frac{L}{2 k\Eg}+t_L}^{\frac{L}{2k\Eg}-t_L} \int_{\gl \in\R} n(E) \boldsymbol{\sigma_E}\Big[f(Ln(E)(\lambda-E) ,(a_j+u\Eg)/L, w)\Big]\, d\gl \,du\;.
\end{align*}
Applying the changes of variables $\gl \mapsto Ln(E)(\lambda-E)$ and $u\mapsto (a_j+u\Eg)/L$, this last expression is equivalent to
$$\int_{\R\times [-\frac{1}{2} + \frac{j-1}{k}, -\frac{1}{2} + \frac{j}{k}]\times \bar\cM} f(\gl,u,w) \; d\gl\otimes du \otimes \boldsymbol{\sigma_E}(dw)\;,$$
as required.
\end{proof}

\subsection{Proof of Lemma \ref{Lemma:CVmgl}}

We are left with proving Lemma \ref{Lemma:CVmgl}. To that end, we need a convergence result in the Crossover regime, that follows from very similar arguments to those presented in~\cite[Prop 4.2]{DLCritical}.
\begin{lemma}\label{Lemma:SDECV}
In the Crossover regime, uniformly over all $\gl \in \Delta$ and all $\theta\in [0,\pi)$ the process
$$ (\bar{\theta}^{(E)}_\gl(t)+E^{3/2}t, \bar{\rho}^{(E)}_\gl(t);t\ge 0)\;,\mbox{ with }\bar{\theta}^{(E)}_\gl(0) = \theta\;,\quad \bar{\rho}^{(E)}_\gl(0) = 0\;;$$
converges in law to $(\bar{\Theta}(t),\bar{R}(t);t\ge 0)$ which solves the following SDEs:
\begin{align*}
d\bar{\Theta}(t) &= -\frac12 d\cB(t) + \frac1{2\sqrt 2}\Re( e^{2i \bar{\Theta}_\gl(t)} d\cW(t))\;,\quad t\ge 0\\
d\bar R(t) &= - \frac14dt + \frac1{\sqrt 2}\Im(e^{2i \bar{\Theta}_\gl(t)} d\cW(t))\;,\quad t\ge 0\;,
\end{align*}
starting from $\bar{\Theta}(0) = \theta$, $\bar R(0) = 0$, where $\cW$ is a complex Brownian motion\footnote{We mean that the real and imaginary parts are independent, standard Brownian motions.} and $\cB$ is an independent real Brownian motion.\\
Furthermore for any compactly supported, continuous function $g:[0,\infty)\to \R$, we have the following convergence in law
$$ \int g(t) e^{\bar{\rho}^{(E)}_\gl(t)} \sin^2 \bar{\theta}_\gl^{(E)}(t) dt \Rightarrow \frac12 \int g(t) e^{\bar R(t)}dt\;.$$
\end{lemma}
\begin{proof}[Proof of Lemma \ref{Lemma:SDECV}]
The proof of the first part is essentially the same as the proof of~\cite[Prop 4.2]{DLCritical} so we only present the main steps. Set $v_\gl(t) = \bar{\theta}^{(E)}_\gl(t)+E^{3/2}t$. Introduce the martingale
$$ d\bar{W}^{(E)}(t) = \sqrt{2} \cos(-2E^{3/2}t) d\bar{B}^{(E)}(t) + i \sqrt{2} \sin(-2E^{3/2}t) d\bar{B}^{(E)}(t)\;.$$
One can write the SDEs \eqref{Eq:bartheta} and \eqref{Eq:barrho} solved by these two processes as follows
\begin{align*}
d v_\gl(t) &= -\frac12 d\bar{B}^{(E)}(t) + \frac1{2\sqrt 2}\Re( e^{2i v_\gl(t)} d\bar{W}^{(E)}(t)) + \cE(v_\gl)(t) dt\\
d \bar{\rho}_\gl^{(E)}(t) &= -\frac14 dt + \frac1{\sqrt 2}\Im(e^{2i v_\gl(t)} d\bar{W}^{(E)}(t)) + \cE(\bar{\rho}_\gl^{(E)})(t)dt\;,
\end{align*}
where $\cE(v_\gl)(t)$ and $\cE(\bar{\rho}_\gl^{(E)})(t)$ should be seen as negligible terms
\begin{align*}
\cE(v_\gl)(t) &= - \sqrt E (\gl-E) \sin^2 \overline{\theta}_\gl^{(E)} +\sin^4 \overline{\theta}_\gl^{(E)} \frac{\partial_\theta \mu_\gl^{(E)}(\overline{\theta}_\gl^{(E)})}{\mu_\gl^{(E)}(\overline{\theta}_\gl^{(E)})} + 3\sin^3(\overline{\theta}_\gl^{(E)})\cos(\overline{\theta}_\gl^{(E)})\;,\\
\cE(\bar{\rho}_\gl^{(E)})(t) &= \sqrt E (\lambda -E)\sin 2\overline{\theta}_\gl^{(E)} +\frac{3}{4} \cos 4\bar{\theta}_\gl^{(E)}(t) - \frac12 \cos 2\bar{\theta}_\gl^{(E)}(t)\\
&\quad- 2 \sin^3 \bar\theta^{(E)}_\gl \cos \bar\theta^{(E)}_\gl \frac{\partial_\theta \mu_\gl^{(E)}(\bar{\theta}_\gl^{(E)})}{\mu_\gl^{(E)}(\bar{\theta}_\gl^{(E)})}\;.
\end{align*}

The proof consists of three steps. First, the tightness of these processes can be derived by estimating the moments of their increments, with the help of the Burkholder-Davis-Gundy inequality. Second, one shows that the integrals in time of the terms $\cE(\cdot)$ are negligible: the main argument is a quantitative Riemann-Lebesgue Lemma that takes advantage of the rapid oscillations of $\bar{\theta}_\gl^{(E)}(t) = v_\gl(t) - E^{3/2}t$. Finally, one identifies the limit of any converging subsequence by a martingale problem. We refer to~\cite[Prop 4.2]{DLCritical} for more details.\\

We turn to the last part of the statement. Fix $g : [0,\infty) \to \R$ a compactly supported continuous function. By Skorohod's Representation Theorem, we can assume that the already proven convergence holds almost surely. Then, using that the function $\sin^2$ is Lipschitz, the term
$$ \int g(t) e^{\bar{\rho}^{(E)}_\gl(t)} \sin^2 \bar{\theta}_\gl^{(E)}(t) dt\;,$$
is asymptotically as close as desired to
$$
\int g(t) e^{\bar{R}(t)} \sin^2 \big(\bar{\Theta}(t) - E^{3/2}t\big) dt = \frac12 \int g(t) e^{\bar{R}(t)} dt - \frac12 \int g(t) e^{\bar{R}(t)} \cos\big(2(\bar{\Theta}(t) - E^{3/2}t)\big) dt\;.$$
The second term on the r.h.s.~rewrites
$$ \frac12 \int g(t) e^{\bar{R}(t)}\Big( \cos\big(2\bar{\Theta}(t)\big)\cos(2E^{3/2}t) + \sin\big(2\bar{\Theta}(t)\big)\sin(2E^{3/2}t)\Big) dt\;,$$
and this goes to $0$ as $L\to\infty$ by the Riemann-Lebesgue Lemma.
\end{proof}
We now proceed with the proof of Lemma \ref{Lemma:CVmgl}.
\begin{proof}[Proof of Lemma \ref{Lemma:CVmgl}]
Assume that we can show:\begin{enumerate}
\item For every compactly supported, continuous function $g : \R\to\R$ the r.v.~
$$ \int g(t) \yu_\gl^{(\Eg)}(t+u)^2dt\;,$$
under $\bbQ^{(u)}_\gl$ converges in law to
$$\int g(t) \boldsymbol{Y_E}(t)^2 dt\;,$$
uniformly over all $u$ and all $\gl$.
\item \emph{Control of the mass at infinity.} There exists $q>0$ such that
$$ \limsup_{x\to\infty}\limsup_{L\to\infty} \,\sup_{u,\gl}\, \bbQ^{(u)}_\gl\Big(\int e^{q|t-u|} \yu_\gl^{(\Eg)}(t)^2dt > x \Big) = 0\;.$$
\end{enumerate}

The control of the mass at infinity (2) allows to lift the convergence in law stated in (1): namely, for every continuous and bounded function $g:\bar\R \to \R$ the r.v. 
$$ \int g(t) \yu_\gl^{(\Eg)}(t+u)^2dt\;,$$
under $\bbQ^{(u)}_\gl$ converges in law to $\int g(t) \boldsymbol{Y_E}(t)^2 dt$ uniformly over all $u$ and all $\gl$. In turn, this implies that the corresponding probability measures converge in law in $\bar\cM$
\begin{equation}\label{Eq:CVq}
\frac{\yu_\gl^{(\Eg)}(t+u)^2 dt}{\int \yu_\gl^{(\Eg)}(t+u)^2dt} \Rightarrow  \frac{\boldsymbol{Y_E}(t)^2 dt}{\int \boldsymbol{Y_E}(t)^2 dt}\;,
\end{equation}
uniformly over all $u$ and all $\gl$.\\
Recall \eqref{Eq:hatUgl} and note that the center of mass of the probability measure on the left is given by $\frac{\hat{U}_\gl}{\Eg}-u$. Assertion (2) then ensures the convergence in law of the centers of mass:
$$ \frac{\hat{U}_\gl}{\Eg}-u \Rightarrow \boldsymbol{U_E}\;.$$
Combined with \eqref{Eq:CVq}, this implies that the following convergence in law in $\bar\cM$ holds
$$\frac{\yu_\gl^{(\Eg)}(t+\hat{U}_\gl /\Eg)^2 dt}{\int \yu_\gl^{(\Eg)}(t+\hat{U}_\gl /\Eg)^2dt} \Rightarrow  \frac{\boldsymbol{Y_E}(t+\boldsymbol{U_E})^2 dt}{\int \boldsymbol{Y_E}(t+\boldsymbol{U_E})^2 dt}\;,$$
as required.\\
Let us now prove the two assumptions above. Regarding Assumption (2), it suffices to show
$$ \limsup_{x\to\infty}\limsup_{L\to\infty} \sup_{u,\gl,\theta} {\P}^{(u)}_{\theta,\pi-\theta}\Big(\int e^{q|t-u|} \yu_\gl^{(\Eg)}(t)^2dt > x \Big)= 0\;.$$
Using the symmetry between the forward and the backward diffusions, it suffices to show
$$ \limsup_{x\to\infty}\limsup_{L\to\infty} \sup_{u,\gl,\theta} {\P}_{(-\frac{L}{2k \Eg},0) \to (u, \theta)}\Big(\int_{-\frac{L}{2k\Eg}}^u e^{q|t-u|} \frac{r_\gl^{(\Eg)}(t)^2}{r_\gl^{(\Eg)}(u)^2} dt > x \Big) = 0\;.$$
Writing $\tilde{u} = u + L/(2k \Eg)$, we deduce that this is implied by the bound on \eqref{Eq:BdMomentTildeu} (using $\tilde u \geq t_L$ and observing that $t_L \gg t_0$ therein) and Markov's inequality.\\
We concentrate on Assumption (1). First of all, the adjunction relation \eqref{Eq:Adjoint} together with Theorem \ref{Th:CVDensities} show that we can deal with the process $\hat{\bar{y}}_\gl^{(\Eg)}$ under
$$\frac1{\sqrt\Eg n(\gl)} \mu^{(\Eg)}_\gl(\theta) \mu^{(\Eg)}_\gl(\pi-\theta) \sin^2 \theta \; \bar{\P}^+_{(0,\theta)\to (u+\frac{L}{2k\Eg}, 0)} \otimes \bar{\P}^-_{(0,\pi-\theta)\to (\frac{L}{2k\Eg}-u, 0)}(\cdot) d\theta\;.$$
By the last part of Lemma \ref{Lemma:CondToUncond} we can disregard the conditionning of the diffusions and work under
$$\frac1{\sqrt\Eg n(\gl)} \mu^{(\Eg)}_\gl(\theta) \mu^{(\Eg)}_\gl(\pi-\theta) \sin^2 \theta \; \bar{\P}^+_{(0,\theta)} \otimes \bar{\P}^-_{(0,\pi-\theta)}(\cdot) d\theta\;.$$

The rest of the proof is presented separately for the Bulk and the Crossover regimes. In the Bulk regime, recall the definition of the process $Y_E$ from the beginning of Subsection \ref{Subsec:Intensity}. We simply need to show that the process $Y_\gl$ under $\sigma_\gl$ converges in law, for the local uniform topology on $\R$, to the process $Y_E$ under $\sigma_E$: this easily follows from the almost sure continuity (for the local uniform topology) of $\gl \mapsto \bar{y}_\gl$.\\

We turn to the Crossover regime. It suffices to show that uniformly over $\theta$ and $\gl$, the two r.v.
$$ \Big(\int_{[0,\infty)} g(t)\hat{\bar{y}}_\gl^{(E)}(t)^2 dt , \int_{(-\infty,0]} g(t) \hat{\bar{y}}_\gl^{(E)}(t)^2 dt \Big)\;,$$
under $\bar{\P}^+_{(0,\theta)} \otimes \bar{\P}^-_{(0,\pi-\theta)}(\cdot)$ converge in law to
$$ \Big(\int_{[0,\infty)} g(t) Y_\infty(t)^2 dt , \int_{(-\infty,0]} g(t) Y_\infty(t)^2 dt \Big)\;.$$
Note that in both cases, the two r.v.~are independent from each other, and that $(Y_\infty(t),t\ge 0)$ and $(Y_\infty(-t),t\ge 0)$ have the same law. Then it suffices to apply Lemma \ref{Lemma:SDECV} to conclude.
\end{proof}

\appendix

\section{Technical results}\label{Appendix}

\subsection{Simple estimate on the Laplace transform}
The following elementary lemma provides a quantitative estimate on the difference between the Laplace transform of a r.v.~and its Taylor expansion at $0$ (here we only push the Taylor expansion to the second order).

 \begin{lemma}\label{Lemma:ExpoTaylor}
Let $X$ be a real-valued r.v. Assume that there exist $C_0>0$ and $q_0>0$ such that $\E[e^{q|X|}] < C_0$ for all $q \le q_0$. Then, for any $q_1 < q_0$ there exist a constant $C_1 > 0$, that only depends on $C_0$, $q_0$ and $q_1$, which is such that for all $q\in [-q_1,q_1]$
$$ \big| \E[e^{qX}] - 1 - q\E[X]\big| \le C_1 q^2\;.$$
\end{lemma}
\begin{proof}
The map $q\mapsto \E[e^{qX}]$ is real-analytic on $(-q_0,q_0)$ and therefore for any such $q$ we have
$$\E[e^{qX}] - 1 - q\E[X] = \sum_{k\ge 2} q^k \frac{\E[X^k]}{k!}\;.$$
By assumption $\sup_{k\ge 2} q_0^k\frac{|\E[X^k]|}{k!} \le C_0$ so that for any $|q| < q_0$ we have
$$ \big|\E[e^{qX}] - 1 - q\E[X]\big| \le C_0 \frac{(q/q_0)^2}{1-q/q_0}\;,$$
and the bound of the statement follows.
\end{proof}

\subsection{Elementary deviation estimate for general SDEs}

\begin{lemma}\label{Lemma:Mgale}
Let $X$ be the solution of $dX(t) = \alpha(t) dt + \beta(t) dB(t)$ with $X(0) = 0$. Assume that $\alpha,\beta$ are adapted processes and that there exist $C_1,C_2>0$ such that almost surely for all $t\ge 0$
$$ |\alpha(t)| \le C_1\;,\quad |\beta(t)|^2 \le C_2\;.$$
Then for all $x> 2C_1 t$ we have
$$ \P(\sup_{s\in [0,t]} |X(s)| > x) \le 2\exp(-\frac{x^2}{8 C_2 t})\;.$$
\end{lemma}
\begin{proof}
Let $M(t) := \int_0^t \beta(s) dB(s)$. Given the bound on the drift, it suffices to show that for all $y>C_1 t$
$$ \P(\sup_{s\in [0,t]} |M(s)| > y) \le 2\exp(-\frac{y^2}{2 C_2 t})\;.$$
This is precisely the exponential martingale inequality of~\cite[p.153-154]{RevuzYor}.
\end{proof}

\subsection{Moments estimates}\label{Subsec:Moments}

\begin{lemma}\label{Lemma:Moments}
For every $\alpha \in [0,1)$
$$ \sup_{\gl \in \R} \,\E\Big[\exp\Big(\frac{\alpha}{m_\gl^{(\Eg)}} \zeta_\gl^{(\Eg)}\Big)\Big] \le \frac{1}{1-\alpha}\;.$$
\end{lemma}
\begin{proof}
Since $\zeta_\gl^{(E)}/m_\gl^{(E)}$ is equal in law to $\zeta_\gl/m_\gl$, it suffices to work with the original coordinates. It is shown in~\cite[Appendix A]{AllezDumazTW} that for all $\alpha < 0$ we have
$$\E[e^{\frac{\alpha}{m_\gl} \zeta_\gl}] = 1 + \alpha + \sum_{n\ge 2} \Big(\frac{\alpha}{m_\gl}\Big)^n v(n,\gl)\;,$$
for some $0 \le v(n,\gl) \le m_\gl^n$. The function on the r.h.s.~is analytic in $\alpha$ on $\{z\in\C: |z| < 1\}$ so that a standard analytic continuation argument allows to deduce that the identity actually holds for all $\alpha \in (-\infty,1)$. The bound of the statement then follows easily.
\end{proof}

\begin{lemma}[Uniform bound of Laplace transform of the number of eigenvalues]\label{Eq:BdExpoDS}
There exists $q >0$, such that for every $E\in\R$ 
\begin{align*} 
\sup_{L/m_E\geq 3} \,\E\Big[\exp\Big(\frac{q m_E}{L} \#\{\gl_i: \gl_i \le E\}\Big)\Big] < \infty\;.
\end{align*}
In particular, for any given $E\in\R$ the collection of r.v.~$(\#\{\gl_i: \gl_i \le E\}/L)_{L \geq 1}$ is uniformly integrable.
\end{lemma}
\begin{proof}
We need exponential tails on the probability that there are more than $k$ eigenvalues below $E$. Let $\theta_E$ start from $\theta_E(0)=0$ and denote by $\zeta^{(k)}_E, k\ge 1$ the i.i.d.~r.v.~which are such that $\sum_{i=1}^k \zeta^{(k)}_E$ is the hitting time of $k\pi$ by $\theta_E$.  By the Sturm-Liouville and the strong Markov properties, we have for all $q >0$,
\begin{align*}
\P( \#\{\gl_i: \gl_i \le E\} \geq k) &= \P\Big(\sum_{i=1}^k \zeta_E^{(i)} \leq L\Big) =\P\Big(q \sum_{i=1}^k (1 - \zeta_E^{(i)}/m_E) \geq q k - q L/m_E\Big) \\
&\leq \E[\exp(q (1 - \zeta_1/m_E))]^k \exp(-q k + q L/m_E)\,.
\end{align*}
By Lemma \ref{Lemma:ExpoTaylor} and \ref{Lemma:Moments}, there exists a constant $C_1$ (independent of $E$) such that for all $q >0$ small enough,
\begin{align*}
\P( \#\{\gl_i: \gl_i \le E\} \geq k)&\leq (1+ C_1 q^2)^k \exp(-q k + q L/m_E)\,.
\end{align*}
Taking $q_0 >0$ small enough, we deduce that for all  for all $q \in [0,q_0]$ and $k \geq 4 L/m_E$,
\begin{align*}
\P( \#\{\gl_i: \gl_i \le E\} \geq k) \leq \exp(- q k/2)\,,
\end{align*}
which implies the result, since
$$ \E\Big[\exp\Big(\frac{q m_E}{L} \#\{\gl_i: \gl_i \le E\}\Big)\Big] = 1 + \int_{u\ge 0} \frac{qm_E}{L} e^{\frac{qm_E}{L} u} \P(\#\{\gl_i: \gl_i \le E\}\ge u) du\;.$$
\end{proof}

\subsection{Integral formulas}\label{Appendix:integral}
In this paragraph, we gather useful integral formulas about the invariant measure of the phase function and the expectations of several estimates on the diffusions.

\subsubsection{Invariant measure}\label{Appendix:invariantmeasure}
The phase function $\{\theta^{(\Eg)}_\gl\}_\pi$ can be mapped to a simple additive SDE via the function $\cotan$, namely $X^{(\Eg)}_\gl := \cotan \{\theta^{(\Eg)}_\gl\}_\pi$ satisfies the SDE:
 \begin{align*}
 dX^{(\Eg)}_\gl = - (V_\gl^{(\Eg)})'(X^{(\Eg)}_\gl)dt + dB^{(\Eg)}(t),\quad \mbox{with } V_\gl^{(\Eg)}(x) := \gl \sqrt{\Eg} x + \Eg^{3/2} \frac{x^3}{3}\,,
\end{align*}
and where $X^{(\Eg)}_\gl$ immediately restarts from $+\infty$ when it blows up to $-\infty$. The diffusion $X_\gl^{(\Eg)}$ admits a unique invariant measure whose density writes $ f_\gl^{(\Eg)}(x)/m_\gl^{(\Eg)}$ where:
\begin{align*}
 &f_\gl^{(\Eg)}(x) := 2 e^{-2V_\gl^{(\Eg)}(x)} \int_{-\infty}^x e^{2V_\gl^{(\Eg)}(y)} dy\;,\quad x\in \R\;,
\end{align*}
 (see again~\cite{AllezDumazTW,DL17} for more details). Note that $\int_{-\infty}^{\infty} f_\gl^{(\Eg)}(x) dx$ coincides with $m_\gl^{(\Eg)}$, which is defined as the expectation of $\zeta_\gl^{(\Eg)}$.
 
Therefore, the unique invariant measure $\mu_\gl^{(\Eg)}$ of the Markov process $\{\theta_\gl^{(\Eg)}\}_\pi$ is the image of $f_\gl^{(\Eg)}(x) dx/m_\gl^{(\Eg)}$ through the map $x\mapsto \arccotan x$. Its density writes
$$ \mu_\gl^{(\Eg)}(\theta) = \frac{f_\gl^{(\Eg)}(\cotan \theta)}{\sin^2 \theta \;m_\gl^{(\Eg)}} =g^{(\Eg)}_\gl(\cotan \theta) \;,\quad \theta \in [0,\pi)\;,$$
where $g^{(\Eg)}_\gl(x) = (1+x^2) f_\gl^{(\Eg)}(x) /m_\gl^{(\Eg)}$.
\begin{proof}[Proof of Lemma \ref{Lemma:InvMeas}]
We concentrate on the distorted coordinates: indeed, since the bounds need to be taken uniformly over the unbounded parameter $E>1$, this is the most involved setting. For simplicity, we make the further assumption that $E$ is large: that is, for fixed $h>0$ we will assume that $E>E_0(h)$ for some implicit quantity $E_0(h)\ge 1$. The complementary case follows by adapting the arguments (and is actually simpler).\\
Let us drop abbreviate $V_\gl^{(E)}$ by $V$ to ease the notations. An integration by parts prompts
\begin{align*}
f_\gl^{(E)}(x) &= \frac{1}{V'(x)} +  2 e^{-2 V(x)}\int_{-\infty}^x e^{2 V(y)} \frac{V''(y)}{2 (V'(y))^2} dy\;.
\end{align*}
One can check that
$$\int_{-\infty}^x e^{2 V(y)} \frac{V''(y)}{2 (V'(y))^2} dy = \int_{-\infty}^x \Big(\frac{e^{2 V(y)}}{V'(y)}\Big)' u(y) dy\;,$$ 
where
$$ u(y) := \frac{V''(y) / V'(y)^2}{4 - 2V''(y) / V'(y)^2}\;.$$
There exists a constant $C_E>0$, that goes to $0$ as $E\to\infty$, such that for all $y\in\R$, we have $0\le u(y) \le C_E$. Moreover, $\Big(\frac{e^{2 V(y)}}{V'(y)}\Big)' \ge 0$ whenever $E\ge E_0(h)$. Consequently
$$ \frac{1}{V'(x)} \le f_\gl^{(E)}(x) \le \frac{1+2C_E}{V'(x)}\;.$$
Recall from \eqref{Eq:Asympmgl} that $m_\gl^{(E)} \sim \pi/(\sqrt \gl\, E)$ as $E\to\infty$. Given that $V'(x) = \gl \sqrt E + E^{3/2} x^2$ one can check that $g_\gl^{(E)}(x)$ is bounded from above and below uniformly over all $x\in\R$ and all $E>E_0(h)$. Moreover, since $C_E \to 0$ as $E\to\infty$, we have $g_\gl^{(E)}(x) \to 1/\pi$ as $E\to \infty$ uniformly over $x$.\\
Regarding the bound on the derivative of $\mu_\gl^{(E)}$, first observe that ${f_\gl^{(E)}}' = 2 - 2 f_\gl^{(E)} V'$. Consequently $\partial_\theta \mu_\gl^{(E)}(\theta) = q_\gl^{(E)}(\cotan \theta)$ where
\begin{align*}
q_\gl^{(E)}(x) = \frac{(1+x^2)^2}{m_\gl^{(E)}} \Big(-2 + 2 f_\gl^{(E)}(x) V'(x) -  \frac{2x}{1+x^2} f_\gl^{(E)}(x)\Big)\;.
\end{align*}
Applying successive integrations by parts we obtain
$$ f_\gl^{(E)}(x) =\frac{1}{V'(x)} + \frac{V''(x)}{2(V'(x))^3} + \frac{3 (V'')^2 - V''' V'}{4 (V')^5} - 2 e^{-2 V(x)}\int_{-\infty}^x e^{2 V(y)}\Big( \frac{3 (V'')^2 - V''' V'}{8 (V')^5}\Big)'dy\;,$$
so that straightforward computations allow to show that $q_\gl^{(E)}(x)$ is bounded from above uniformly in $x$ and $E>E_0(h)$.
\end{proof}

\subsubsection{Expression of the Lyapounov exponent}\label{Subsec:Lyapounov}

\begin{proof}[Proof of Proposition \ref{propo:Lyapounov}]
Assume that $\rho_\gl^{(\Eg)}, \zeta_\gl^{(\Eg)}$ are integrable. Decomposing the trajectory of $\rho_\gl^{(\Eg)}$ into i.i.d.~excursions in between the successive hitting times of $\pi\Z$ by $\theta_\gl^{(\Eg)}$, one deduces from the law of large numbers that almost surely
$$ \frac{\ln \rho_\gl^{(\Eg)}(t)}{t} \to \frac{\E[\rho_\gl^{(\Eg)}(\zeta_\gl^{(\Eg)})]}{\E[\zeta_\gl^{(\Eg)}]}\;,\quad t\to\infty\;.$$
In particular, the limit coincides with $\nu_\gl^{(\Eg)}$.\\
The drift and diffusions coefficients of the (additive) SDE satisfied by $\rho_\gl^{(\Eg)}$ are bounded. Furthermore we already saw that $\zeta_\gl^{(\Eg)}$ is integrable. Standard stochastic calculus arguments therefore show that 
\begin{align}\label{Eq:Expectlnr}
\E[\rho_\gl^{(\Eg)}(\zeta_\gl^{(\Eg)})] = \E\Big[ \int_0^{\zeta_\gl^{(\Eg)}} \Big( -\sqrt \Eg(\gl - \Eg)\sin 2\theta_\gl^{(\Eg)} -\frac12 \sin^2 2\theta_\gl^{(\Eg)} + \sin^2 \theta_\gl^{(\Eg)}\Big) dt \Big]\;.
\end{align}
Recall that $m_\gl^{(\Eg)} = \E[\zeta_\gl^{(\Eg)}]$ admits the explicit integral expression \eqref{expressionmgl}. To complete the proof, it suffices to show that
$$\E[\rho_\gl^{(\Eg)}(\zeta_\gl^{(\Eg)})] = \sqrt{2 \pi} \int_0^{+\infty} \sqrt{u} \exp(-2 \lambda u - \frac{u^3}{6} ) du\;.$$
By the standard characterization of the invariant probability measure of positive recurrent Markov processes, we have
\begin{align*}
\E[\rho_\gl^{(\Eg)}(\zeta_\gl^{(\Eg)})] 
&= \int_{0}^\pi \Big( -\sqrt \Eg(\gl - \Eg)\sin 2\theta -\frac12 \sin^2 2\theta + \sin^2 \theta \Big)\mu_\gl^{(\Eg)}(\theta) m_\gl^{(\Eg)}d\theta\;.
\end{align*}
Applying the change of variable $x = \cotan \theta$, we get
\begin{equation}\label{Eq:IntMoche}
\E[\rho_\gl^{(\Eg)}(\zeta_\gl^{(\Eg)})] = \int \Big(\frac{2x \sqrt \Eg(\Eg-\lambda)}{1+x^2} + \frac{1-x^2}{(1+x^2)^2}\Big) f_\lambda^{(\Eg)}(x)dx\;.
\end{equation}
We compute the r.h.s.~of \eqref{Eq:IntMoche}. By the Dominated Convergence Theorem, we have
$$ \E[\rho_\gl^{(\Eg)}(\zeta_\gl^{(\Eg)})] = \lim_{A\to\infty} \int_{-A}^A \Big(\sqrt \Eg(\Eg-\lambda)\frac{2x}{1+x^2} + \frac{1-x^2}{(1+x^2)^2}\Big) f_\lambda^{(\Eg)}(x)dx\;.$$
Fix $A>0$. We have (from now on, we abbreviate $f_\gl^{(\Eg)},V_\gl^{(\Eg)}$ in $f_\gl,V_\gl$)
$$ \sqrt \Eg(\Eg-\lambda)\frac{2x}{1+x^2} = 2x \Eg^{3/2} - \frac{2x}{1+x^2} V'_\gl(x)\;.$$
Using the identity $f_\gl'(x) = -2V'_\gl(x) f_\gl(x) + 2$, an integration by parts yields
\begin{align*}
\int_{-A}^A - \frac{2x}{1+x^2} V'_\gl(x) f_\gl(x) dx &= \int_{-A}^A \frac{x}{1+x^2} (f'_\gl(x)-2) dx\\
&= \Big[\frac{x}{1+x^2} f_\gl(x)\Big]_{-A}^A -  \int_{-A}^A \frac{1-x^2}{(1+x^2)^2} f_\gl(x) dx - \int_{-A}^A  \frac{2x}{1+x^2}\;.
\end{align*}
The last term vanishes by symmetry. We deduce from the previous computation that
$$ \int_{-A}^A \Big(\sqrt \Eg(\Eg-\lambda) \frac{2x}{1+x^2} + \frac{1-x^2}{(1+x^2)^2}\Big) f_\lambda(x)dx = \int_{-A}^A 2x \Eg^{3/2} f_\gl(x) dx + \Big[\frac{x}{1+x^2} f_\gl(x)\Big]_{-A}^A\;.$$
The second term vanishes as $A\to\infty$. Indeed, an integration by parts on the defining expression of $f_\gl$ and some elementary computations show that
\begin{align*}
f_\gl(x) \sim \frac1{V_\gl'(x)} = \frac{1}{\gl \sqrt\Eg + \Eg^{3/2} x^2}, \quad \mbox{when }x \to \pm \infty\;.
\end{align*}
Regarding the first term, we have
\begin{align*}
\int_{-A}^A 2x f_\gl(x) dx &= \int_{-A}^A 4x e^{-2V_\gl(x)} \int_{-\infty}^x e^{2V_\gl(y)} dy dx\;.
\end{align*}
We then apply the change of variables $(u,v) := (x-y,x+y)$ and get
\begin{align*}
\int_{-A}^A 2x f_\gl(x) dx &= \int_{u \in (0,\infty)} \int_{v\in \R} \un_{\{u+v \in [-2A,2A]\}} u e^{-2 \lambda \sqrt \Eg u - \frac{\Eg^{3/2}}{6}u^3} e^{- \frac{\Eg^{3/2}}{2} u v^2} dv du\\
&+ \int_{u \in (0,\infty)} \int_{v\in \R} \un_{\{u+v \in [-2A,2A]\}} v e^{-2 \lambda \sqrt \Eg u - \frac{\Eg^{3/2}}{6}u^3} e^{- \frac{\Eg^{3/2}}{2} u v^2} dv du\;.
\end{align*}
It is straightforward to check that the second term on the r.h.s.~goes to $0$ as $A\to\infty$. By the Dominated Convergence Theorem, the first term converges as $A\to\infty$ to
$$ \int_{0}^{+\infty} u e^{-2 \lambda \sqrt \Eg u - \frac{\Eg^{3/2}}{6} u^3} (\int_{-\infty}^{+\infty}e^{- \frac{\Eg^{3/2}}{2} u v^2} dv) du =  \frac{\sqrt{2 \pi}}{ \Eg^{3/2}} \int_0^{+\infty} \sqrt{u} e^{-2 \lambda u - \frac{u^3}{6}} du \;,$$
as required.
\end{proof}

\subsection{Control of the time spent near $\pi \Z$}\label{subsec:controlpiZ}

In the previous sections, we needed a moment bound on the r.v.
$$\Big(\int_0^1 F(\theta_\gl^{(\Eg)}(t))dt\Big)^{-1}\quad \mbox{and}\quad \Big(\int_0^1 F(\bar\theta_\gl^{(\Eg)}(t))dt\Big)^{-1}\;,$$
where $F$ is either $F(x) =\sin^2(x)$ or $F(x) = \sin^4(x)$. To prove such a bound, one needs to show that $\theta_\gl^{(\Eg)}$ does not spend too much time near $\pi\Z$ (since the function $F$ vanishes there).\\
In both cases, $F:\R \to \R_+$ is a $\pi$-periodic smooth function that satisfies, for some $k\ge 1$ and some $c>0$
\begin{equation}\label{Eq:Fck} F(x) \ge c\, d(x,\pi\Z)^k\;,\quad x\in \R\;,\end{equation}
where $d$ is the Euclidean distance. Our proof will only rely on this property, and therefore applies to a large class of functions $F$.

\begin{lemma}[Control of the time spent near $\pi \Z$ by the phase function]\label{Lemma:ExpoBoundTheta}
\emph{Original coordinates.} For any compact interval $\Delta \subset\R$, there exists a constant $q>0$ such that
$$ \sup_{\lambda\in \Delta} \sup_{\theta\in [0,\pi)} \E_{(0,\theta)}\Big[\exp\Big(q \big(\int_0^1 F(\theta_\gl(t)) dt\big)^{-\frac1{k+1}} \Big)\Big] < \infty\;.$$
\emph{Distorted coordinates.} Fix $h>0$ and set $\Delta :=  [E-h/(n(E) E) , E + h/(n(E)E)]$. There exists a constant $q>0$ such that
$$ \sup_{E>1} \sup_{\lambda\in \Delta}  \sup_{\theta\in [0,\pi)} \E_{(0,\theta)}\Big[\exp\Big(q \big(\int_0^1 F(\theta_\gl^{(E)}(t)) dt\big)^{-\frac1{k+1}} \Big)\Big] < \infty\;.$$
The same holds with $\theta_\gl, \theta_\gl^{(E)}$ replaced by $\bar{\theta}_\gl, \bar{\theta}_\gl^{(E)}$, the adjoint diffusions defined in \eqref{Eq:bartheta}.
\end{lemma}
\begin{proof}
From now on, ``all the parameters'' will refer to $\gl$ and $\theta$ when working with the original coordinates, and $E,\gl,\theta$ when working with the distorted coordinates.\\
We only need to prove that there exists a constant $c' >0$ such that for all $\eps >0$ small enough, and for all the parameters
\begin{align}\label{ineqonprobalargetime}
 \P_{(0,\theta)}\Big(\int_0^1 F(\theta^{(\Eg)}_\gl(t))dt < \epsilon\Big) \leq \exp\big(- c' \;\eps^{-\frac1{k+1}}\big)\;.
\end{align}

For simplicity, we now drop the superscript $(\Eg)$ although our proof is carried out simultaneously for the two sets of coordinates. 
The idea of the proof is to consider the solution $\gamma_\gl$ of the ODE corresponding to the deterministic part of our diffusion, to estimate the integral of $F(\gamma_\gl)$ on some interval $[0,T]$, with $T<1$, and to control the probability that $\gamma_\gl$ and $\theta_\gl$ differ on $[0,T]$.\\
Let $\gamma_\gl$ be the solution of
$$ d\gamma_\gl(t) = \Big(\Eg^{3/2} + \sqrt{\Eg} (\gl-\Eg) \sin^2 \gamma_\gl +  \sin^3 \gamma_\gl \cos\gamma_\gl\Big) dt\;,$$
starting from $\gamma_\gl(0)=\theta_\gl(0) = \theta$. We would like to bound from below the time spent by $\gamma_\gl$ near $\pi \Z$ on the time interval $[0,T]$.

\smallskip

Let us introduce the largest value (mod $\pi$) that $\gamma_\gl$ reaches on $[0,T]$, that we denote by $a := \sup_{t\in [0,T]} d(\gamma_\gl(t),\pi \Z)$. We claim that: 
\begin{itemize}
\item The distance $a$ is not too small i.e. there exists a constant $c_0\in (0,1)$ independent of all the parameters and of $T$ such that
\begin{equation}\label{Eq:aT} a \ge c_0 \min(T\Eg^{3/2},1) \ge c_0 T\;.\end{equation}
\item The function $\gamma_\gl$ spends some time above $a/2$: There exists a constant $c_1>0$ independent of all the parameters and of $T$ such that
\begin{equation}\label{Eq:infa} \int_0^T \un_{\{d(\gamma_\gl(t),\pi\Z) > a/2\}} dt \ge c_1 T\;.\end{equation}
\end{itemize}
Given these two claims, we now control the difference between $\theta_\gl$ and $\gamma_\gl$. By Gr\"onwall's Lemma there exists a constant $C>0$ such that for all $T\in [0,1]$,
$$ \sup_{t\in [0,T]} |\theta_\gl(t) - \gamma_\gl(t)| \le C \sup_{t\in[0,T]} |M(t)|\;,$$
where $M(t) := - \int_0^t \sin^2 \theta_\gl dB(s)$. Let $\tau := \inf\{t\ge 0: |\theta_\gl(t)-\gamma_\gl(t)| \ge a/4\}$. We deduce that
\begin{align*}
 \P(\tau \le T) \leq \P\Big( \sup_{t\in[0,T\wedge \tau]} |M(t)| \ge \frac{a}{4C}\Big)\;.
\end{align*}

Recall that $a = \sup_{t\in [0,T]} d(\gamma_\gl(t),\pi \Z)$. There exists a constant $C'>0$ such that for all $t \leq T\wedge \tau$, $\partial_t\langle M \rangle_t \le C' a^4$. Consequently Lemma \ref{Lemma:Mgale} yields the existence of $C_0>0$ such that
$$ \P(\tau \le T) \le 2e^{-\frac{C_0}{T}}\;.$$

\smallskip

Combining \eqref{Eq:Fck}, \eqref{Eq:aT} and \eqref{Eq:infa}, we deduce that on the event $\{\tau > T\}$ we have
$$ \int_0^T F(\theta_\gl(t))dt = \int_0^T F(\gamma_\gl(t) + \theta_\gl(t)-\gamma_\gl(t)) dt \ge c_1 c \; T (a/4)^k \ge c_2 T^{k+1}\;,$$
for some constant $c_2>0$ independent of all parameters and of $T$. For any $\eps \in (0,c_2]$, one can choose $T\in [0,1]$ such that $\eps= c_2 T^{k+1}$. Then we have
$$ \P\Big(\int_0^T F(\theta_\gl(t))dt < \eps\Big) \le \P(\tau \le T) < 2\exp\Big(- C_0 \Big(\frac{c_2}{\eps}\Big)^{\frac1{k+1}}\Big)\;,$$ 
as required. It remains to prove the two claims stated above.

\medskip

For the first point, there exists a constant $c'>0$ such that uniformly over the parameters the derivative of $\gamma_\gl$ is larger than $c'\Eg^{3/2}$ whenever $\gamma_\gl$ lies in some fixed neighborhood of $
\pi \Z$, say $[-\delta,\delta]  + \pi \Z$ with $\delta \leq \pi$. Set $c_0 := \min(c'/4,\delta)$. We now distinguish two cases. If $d(\theta,\pi\Z)$ is larger than $\min( (c'/4)T\Eg^{3/2} , \delta)$ then
$$ a \ge d(\theta,\pi\Z) \ge \min( (c'/4)T\Eg^{3/2} , \delta) \ge c_0 \min(T\Eg^{3/2},1)\;.$$
If $d(\theta,\pi\Z)$ is smaller than $\min( (c'/4)T\Eg^{3/2} , \delta)$, then until $\gamma_\gl$ reaches $\delta$, its derivative is larger than $c'\Eg^{3/2}$. Therefore by time $T$, the distance of $\gamma_\gl$ to $\pi\Z$ passes above $\min((c'/2)T\Eg^{3/2},\delta)$, and this yields $a \ge c_0 \min(T\Eg^{3/2},1)$. Finally, since $T\in [0,1]$ and $\Eg\ge 1$, the bound $c_0 \min(T\Eg^{3/2},1) \ge c_0 T$ is immediate.

\smallskip

For the second point, the absolute value of the derivative of $\gamma_\gl$ is bounded from above by $K \Eg^{3/2}$ for some $K>0$ . Consequently (the distance to $\pi\Z$ of) $\theta_\gl$ takes a time at least $a/(2K \Eg^{3/2})$ to go from $a/2$ to $a$ (or from $a$ to $a/2$). With the original coordinates, using \eqref{Eq:aT} we thus get $a/(2K) \ge c_0 T/(2K)$ as required.\\
Let us now consider the distorted coordinates: for convenience, we work under the further assumption that $E>E_0(h)$ where $E_0(h)\ge 1$ is a constant that depends only on $h$; the complementary case follows from an easy adaptation. Then the derivative of $\gamma_\gl$ is larger than $E^{3/2}/2$ everywhere. First assume that $TE^{3/2} < 4\pi$. Then by \eqref{Eq:aT} we have $a \ge c_0 T E^{3/2}/(4\pi)$ and thus we find $a/(2K E^{3/2}) \ge c_0 T /(8 \pi K)$ as required. We now assume that $TE^{3/2}\ge 4\pi$. The function $\gamma_\gl$ makes a number of rotations over the circle which is at least $TE^{3/2}/(2\pi) \ge 2$. Necessarily $a=\pi/2$. On each rotation, it spends a time at least $a/(2KE^{3/2})$ in the favorable region. We easily conclude to the second point of our claim.
\end{proof}

\subsection{A bound on the centers of mass}

Let $f,g$ be two functions with $L^2(\R)$-norm equal to one. Let $V_f,V_g$ be their centers of mass, that is, $V_f := \int t f^2(t) dt$ and $V_g:=\int t g^2(t)dt$, and let $w_f,w_g$ be the associated probability measures on $\R$ recentered at $V_f, V_g$, that is, $w_f(dt) = f^2(V_f +t)dt$ and $w_g(dt) = g^2(V_g+t)dt$. 

\begin{lemma}\label{Lemma:Centers}
Assume that there exist some constants $q,c>0$ and some $\delta \in (0,1/10)$ such that
$$ \int |f(t)-g(t)|^2 dt \le \delta^2\;,\quad \mbox{ and for all }t\in \R\;,\quad  |f(t)| \le c\, e^{-q|t-V_f|}\;,\quad |g(t)| \le c\, e^{-q|t-V_g|}\;.$$
Then, there exists a constant $C>0$, which grows at most polynomially in $c$, such that $|V_f - V_g| \le C \delta^{1/2}$ and $d_\cM(w_f,w_g) \le C\delta^{1/2}$, where $d_\cM$ is the L\'evy-Prokhorov distance introduced in \eqref{Eq:LevyPro}.
\end{lemma}
\begin{proof}
Let $X,Y$ be two r.v.~with densities $f^2$ and $g^2$ respectively. Note that $V_f = \E[X]$ and $V_g = \E[Y]$. By assumption, we have for any $A>0$
$$ \P(|X-\E[X]| > A) < \frac{c^2}{q} e^{-qA}\;,\quad   \P(|Y-\E[Y]| > A) < \frac{c^2}{q} e^{-qA}\;.$$
Furthermore
\begin{align*}
\big| \P(|X-\E[Y]| \le A) - \P(|Y-\E[Y]| \le A) \big| &= \Big|\int_{[\E[Y]-A,\E[Y]+A]} (f^2(t) - g^2(t))dt\Big|\\
&\le \big(\int |f(t) - g(t)|^2dt\big)^{1/2} \big(\int |f(t) + g(t)|^2dt\big)^{1/2}\\
&\le 2{\delta}\;.
\end{align*}
Choose $A>0$ such that (it can be chosen in such a way that it grows logarithmically in $c$)
$$2\frac{c^2}{q} e^{-qA} + 2{\delta} <  2\frac{c^2}{q} e^{-qA} + 2\frac1{10}< 1\;.$$
Assume that $|\E[X]-\E[Y]|>2A$. Then
$$ \P(|X-\E[Y]| \le A) \le  \P(|X-\E[X]| > A) < \frac{c^2}{q} e^{-qA}\;,$$
while
$$ \P(|Y-\E[Y]| \le A) \ge 1-\frac{c^2}{q} e^{-qA}\;,$$
so that
$$ \big| \P(|X-\E[Y]| \le A)-\P(|Y-\E[Y]| \le A)\big| > 1 - 2\frac{c^2}{q} e^{-qA} > 2{\delta}\;,$$
thus raising a contradiction. Consequently, $|\E[X]-\E[Y]|\le 2A$.

This being given, take $A' >0$ and define the interval $I := [\min(\E[X],\E[Y]) -A', \max(\E[X],\E[Y]) +A']$, and let $m$ be its midpoint. We have
\begin{align*}
\E[X] - m = \int_{I} (t-m) f^2(t) dt + \int_{\R\backslash I} (t-m) f^2(t) dt\;,\\
\E[Y] - m = \int_{I} (t-m) g^2(t) dt + \int_{\R\backslash I} (t-m) g^2(t) dt\;.
\end{align*}
Note that
$$ \Big| \int_{I} (t-m) f^2(t) dt - \int_{I} (t-m) g^2(t) dt\Big| \le (A+A') \int_\R |f^2(t)-g^2(t)| dt \le 2(A+A') \delta\;.$$
Moreover,
\begin{align*}
\big|\int_{\R\backslash I} (t-m) f^2(t) dt \big| &\le \int_{\R\backslash [\E[X]-A',\E[X]+A'])}(|t-\E[X]| + A) c^2 e^{-2q|t-\E[X]|}\\
&\le \Big(\frac{A'}{q} + \frac1{2q^2} + \frac{A}{q}\Big)c^2e^{-2qA'}\;.
\end{align*}
The same bound holds for $\big|\int_{\R\backslash I} (t-m) g^2(t) dt\big|$. This ensures that
$$ \big| \E[X]-\E[Y]\big| \le 2(A+A') \delta + 2\Big(\frac{A'}{q} + \frac1{2q^2} + \frac{A}{q}\Big)c^2e^{-2qA'}\;.$$
Choosing $A'=\delta^{-1/2}$, we get the desired bound on $|V_f-V_g|$. Let $w'_f$ be the probability measure associated with $f^2$ recentered at $V_g$, that is, $w'_f(dt) = f^2(V_g+t)dt$. Then, it is easy to check that
$$ d_\cM(w_g,w'_f) \le \int |f^2-g^2|(t) dt \le 2 (\int |f-g|^2dt)^{1/2} \le 2 \delta\;.$$
Furthermore for any Borel set $B\subset \R$, note that for any $\epsilon > |V_f-V_g|$ we have $V_f + B \subset V_g + B^{\epsilon}$ and $V_g + B \subset V_f + B^\epsilon$ so that
\begin{align*}
w_f(B) = \int_{V_f+B} f^2(t) dt \le \int_{V_g+B^{\epsilon}} f^2(t) dt = w_f'(B^{\epsilon})\;,\\
w_f'(B) = \int_{V_g+B} f^2(t) dt \le \int_{V_f+B^{\epsilon}} f^2(t) dt = w_f(B^{\epsilon})\;,
\end{align*}
so that $d_\cM(w_f,w'_f) \le |V_f-V_g|$, thus concluding the proof.
\end{proof}

\bibliographystyle{Martin}
\bibliography{library}

\end{document}